\documentclass[final,3p,times]{elsarticle}
\usepackage{amssymb}
\usepackage{amsmath}
\usepackage{amsthm}
\usepackage{booktabs}
\usepackage{multirow}
\usepackage{tikz}
\usepackage{bm}
\usepackage{float}
\usepackage{algorithm}
\usepackage{algpseudocode}
\newtheorem{theorem}{Theorem}
\newtheorem{lemma}{Lemma}
\theoremstyle{remark}
\newtheorem{remark}{Remark}
\usepackage{tabularx}
\usepackage{hyperref}
\usepackage{xcolor}

\makeatletter
\g@addto@macro\normalsize{
  \setlength\abovedisplayskip{5pt}
  \setlength\belowdisplayskip{5pt}
  \setlength\abovedisplayshortskip{3pt}
  \setlength\belowdisplayshortskip{3pt}
}
\makeatother
\journal{}
\begin{document}

\begin{frontmatter}



  \title{Stable self-adaptive timestepping for Reduced
  Order Models for incompressible flows}


  \author[cttc]{Josep Plana-Riu} %% Author name
  \author[cwi,tue]{Henrik Rosenberger}
  \author[cwi,tue]{Benjamin Sanderse}
  \author[cttc]{F.Xavier Trias} %% Author name
%  \author[cttc]{Assensi Oliva}

  \affiliation[cttc]{organization={Heat and Mass Transfer Technological Centre, Technical University of Catalonia},%Department and Organization
            addressline={Carrer de Colom 11}, 
            city={Terrassa},
            postcode={08222}, 
            country={Spain}}
            \affiliation[cwi]{organization={Centrum Wiskunde \& Informatica},%Department and Organization
            addressline={Science Park 123}, 
            city={Amsterdam},
            postcode={1098 XG}, 
            country={the Netherlands}}
            \affiliation[tue]{organization={Centre for Analysis, Scientific
            Computing and Applications, Eindhoven University of Technology},%Department and Organization
            addressline={PO Box 513}, 
            city={Eindhoven},
            postcode={5600 MB}, 
            country={the Netherlands}}

  %% Abstract
\begin{abstract}
  This work introduces \texttt{RedEigCD}, the first self-adaptive timestepping technique
  specifically tailored for reduced-order models (ROMs) of the incompressible Navier-Stokes
  equations that employ stucture-preserving discretizations. Building upon linear stability concepts, the method adapts the timestep by directly
  bounding the stability function of the employed time integration scheme using exact spectral information
  of matrices related to the reduced operators. Unlike traditional error-based adaptive methods, \texttt{RedEigCD} relies on
  the eigenbounds of the convective and diffusive ROM operators, whose computation is feasible at reduced scale and fully preserves the online efficiency of the ROM.
  A central theoretical contribution of this work is the proof, based on the combined theorems of Bendixson and Rao,
  that, under linearized assumptions, the maximum stable timestep for projection-based ROMs is shown to be larger than or equal
  to that of their corresponding
  full-order models (FOMs).

  Numerical experiments for both periodic and non-homogeneous boundary conditions demonstrate that \texttt{RedEigCD} yields
  stable timestep increases up to a factor 40 compared to the FOM, without compromising accuracy. The methodology thus
  establishes a new link between linear stability theory and reduced-order modeling, offering a systematic path towards efficient,
  self-regulating ROM integration in incompressible flow simulations.
\end{abstract}

%%  \input{graphicalAbstract}

\begin{highlights}
\item Introduces \texttt{RedEigCD}, a stability-based self-adaptive timestepper for ROMs.
\item Provides a theoretical proof that ROMs permit larger stable timesteps than FOMs.
\item Extends Bendixson and Rao theorems to analyze ROM stability bounds.
\item Achieves up to 40x larger stable timesteps without accuracy loss.
\item Establishes a general, stable, efficient framework for adaptive ROM time integration.
\end{highlights}

%% Keywords
\begin{keyword}
  Computational Fluid Dynamics \sep Incompressible Navier-Stokes \sep
  Reduced-Order Modeling \sep Linear
  stability analysis \sep Self-adaptive time integration



\end{keyword}

\end{frontmatter}


% >>> Begin input from Introduction.tex
\section{Introduction}
The numerical integration of nonlinear ordinary differential equations (ODEs), $\frac{d}{dt}\mathbf u=f(\mathbf u(t),t)$, like those arising from spatial discretization of the Navier-Stokes equations, requires
setting the timestep of the integration.
In the classical works of \citet{hairer_solving_1993, butcher_numerical_2016},
this timestep, $\Delta t$, is generally assumed to be predetermined by the user, while later on, methods such as the Dormand-Price method (RKDP)
\cite{dormand_family_1980} or Fehlberg (RKF) \cite{fehlberg_klassische_1969}
were introduced to provide adaptive timestepping, based on truncation error estimates.
As an alternative, the timestep $\Delta t$ can also be determined by
studying the stability of the linear system

\begin{equation}
  \frac{d\mathbf{u}}{dt} = A\mathbf{u},
  \label{eq:linear_ode}
\end{equation}

\noindent where $A$ follows from locally linearizing $f(\mathbf{u}(t),t)$.
Restricting the product of the eigenbound of $A$, $\rho(A)$, and the timestep $\Delta t$ to lie in the boundary of the stability 
region of the numerical time integration method thus maximizes the
timestep when using an explicit time integration. 
\citet{courant_uber_1928} first introduced this concept in context of Computational Fluid Dynamics (CFD) with 
the development of the Courant-Friedrichs-Lewy (CFL) condition,
where the stability condition is met by estimating the eigenvalues of the 1D convection-diffusion equation. Later on,
this was generalized to the 3D Navier-Stokes equations by 
\citet{trias_self-adaptive_2011}, where the timestep of the integration
was modified each timestep to lie within the limit of the stability region of a second-order explicit multistep
scheme. This was done by estimating the eigenvalues of the linearized system
by means of the Gershgorin circle theorem. Later works such as \cite{trias_efficient_2024}
improve the method by estimating the eigenbound of the linearized system without the need to
fully reconstruct the matrix by exploiting its structure, thus making it more efficient.

Nonetheless, scale-resolving simulations, i.e., direct numerical simulation (DNS) and large-eddy simulation (LES), are
generally prohibitively costly, especially for tasks that require many simulations such as control, design, optimization, or uncertainty quantification
\cite{sanderse_non-linearly_2020}. Hence, reduced-order models (ROMs) seek to
alleviate the computational costs by projecting the high-dimensional system, also known as
full-order model (FOM), to a reduced subspace \cite{leblond_optimal_2011, reyes_projection-based_2020,prakash_projection-based_2024}, which depends on the
construction of a basis. The methodology used for constructing this basis
depends mostly on the characteristics of the problem, i.e., in an unsteady case
where no parameters are involved, usually this basis is built as a set of
Proper Orthogonal Decomposition (POD)
modes from a matrix of snapshots
\cite{lumley_structure_1967,sirovich_turbulence_1987,holmes_turbulence_1996}, while in a case where the problem is
parameter-based, this basis is built based on a (greedy) Reduced Basis algorithm \cite{quarteroni_numerical_2007,manzoni_efficient_2014}. However, the procedure after constructing the basis is
equivalent. For more details, the reader is referred to
\citet{lassila_model_2014}.

Beyond timestep control at the FOM level, several contributions have addressed adaptive or stability-aware strategies in
the reduced-order modeling framework. \citet{amsallem_stabilization_2012} proposed interpolation and stabilization techniques for
parametric ROMs, emphasizing the importance of maintaining spectral consistency of the reduced operators. Their work highlighted
the importance of maintaining consistency between the ROM and FOM eigenvalues to ensure numerical robustness. Nevertheless, this stabilization strategy does not explicitly link timestep adaptivity to linear stability theory, nor does it exploit the exact spectral information that can be computed efficiently at a reduced scale. Building on the concept of using the eigenvalues of the linear operator to guide subspace rotation {\cite{balajewicz2016minimal}}, \citet{rezaian_global_2021} demonstrated that analyzing and reassigning the eigenvalues of the reduced system's linear operator provides an effective means to control instability while preserving dominant dynamical features, establishing spectral analysis as a central tool for stability-aware model reduction.

Even though lots of efforts are devoted to pushing the boundaries of
POD-Galerkin ROMs applied to CFD (closure models \cite{wang_proper_2012,ahmed_closures_2021}, structure-preserving operators \cite{rosenberger_no_2023, klein_energy-conserving_2024, farhat_structure-preserving_2015}, hyper-reduction \cite{barrault_empirical_2004,chaturantabut_state_2012,farhat_structure-preserving_2015,klein_energy-conserving_2024}, stabilization \cite{ahmed_stabilized_2018,stabile_reduced_2019,li_pressure-stabilized_2025}, etc.), little effort has been made to finetune the timestep
to maximize the "performance" of the time integration for these ROMs.
\citet{sanderse_non-linearly_2020} indicated the use of the Gershgorin circle theorem as a possible way to estimate the eigenbounds of the combined reduced system to set the timestep, however no tests or implementation were
presented. While recent non-intrusive frameworks have employed Gershgorin-based regularization to ensure the stability of learned operators, e.g., Gkimisis et al. {\cite{gkimisis2025}}, the present work realizes Sanderse's vision {\cite{sanderse_non-linearly_2020}} in an intrusive context by implementing a stability-based adaptive timestepper. This aligns with efforts in other fields, such as structural mechanics, where \citet{bach_stability_2018}
performed an extensive analysis of the stability conditions for projection-based ROMs, where the maximum stable timesteps for both FOM and ROM were compared analytically. It was observed that the maximum stable timestep at the ROM level, assuming a linear problem, would not be smaller than the maximum stable timestep at the FOM level. To the best of our knowledge, such an analysis has not been performed for the nonlinear incompressible Navier-Stokes equations.

Therefore, this work aims to bridge the gap between self-adaptive
timestepping techniques developed for scale-resolving CFD simulations
\cite{trias_self-adaptive_2011,trias_efficient_2024} and reduced-order
models, as according to \citet{bach_stability_2018}, it should be possible to generally obtain larger timesteps. This observation is further supported by recent developments across various ROM paradigms. Beyond machine-learning-augmented frameworks that utilize coarsened timestepping \cite{Puri2025, pawar2025neural,san2018extreme}, both classical Galerkin {\cite{carlberg2017galerkin}} and non-intrusive projections {\cite{xiao2015non}} have demonstrated the ability to maintain stability and qualitative physics at temporal resolutions significantly coarser than the FOM.
Thus, the potential benefit of using a ROM is twofold: the reduced cost due to dimension reduction, and 
marching faster in time due to larger timesteps. 

To extend existing self-adaptive timestepping methodologies to ROMs, the estimation of
the eigenbounds should not break the leitmotif of reduced-order modeling
\cite{stabile_finite_2018}: online FOM-size operations should be avoided, in order to preserve the efficiency of the
ROM compared to the FOM. %Thus, the current techniques for FOMs should be studied to guarantee that this property is preserved before applying them straight away to ROMs.

The present work introduces a new class of stability-driven self-adaptive timestepping algorithms
for reduced-order models of the incompressible Navier-Stokes equations. The key novelty lies in (i)
extending existing self-adaptive timestep control strategies - initially developed for full-order CFD solvers - to
projection-based ROMs while fully preserving reduced-order efficiency, and (ii) establishing a theoretical proof, based on
the combined theorems of Bendixson \cite{bendixson_sur_1902} and Rao \cite{rao_separation_1979}, that the maximum stable timestep
of a ROM is always larger or equal to that of its FOM counterpart. This connection between linear stability theory and reduced-order
modeling has not been reported before for nonlinear incompressible flow systems. Furthermore, the proposed algorithm, termed 
\texttt{RedEigCD}, achieves stability control without resorting to full-order model operations. Lastly, it applies to both periodic and
inflow-outflow configurations. While the current derivation assumes a structure-preserving, discretely divergence-free discretization, the underlying spectral bounding framework remains generalizable to non-divergence-free formulations through standard operator splitting. Together, these elements form a novel stability-aware time-integration framework that advances the state-of-the-art in efficient and reliable ROM-based CFD simulations.

This article is organized as follows. In Section \ref{sec:preliminaries}, the
full-order model, the theory for linear stability analysis, and projection-based ROMs
are introduced. In Section \ref{sec:redeig}, our novel self-adaptive timestepping method for ROMs is presented, as well as an explanation why maximum
stable timesteps for ROMs are generally larger than for their FOM counterparts. In
Section \ref{sec:results}, the novel self-adaptive timestepping method is
tested in two test cases with different boundary conditions. Eventually, in Section \ref{sec:conclusion}, 
 conclusions are drawn and future work is outlined.
% <<< End input from Introduction.tex
% >>> Begin input from Preliminaries.tex
\section{Preliminaries: full-order model, linear stability analysis, and projection-based reduced-order model} \label{sec:preliminaries}

\subsection{Full-order model} \label{sec:fom}

The incompressible Navier-Stokes equations describe the behaviour of incompressible flows. In dimensionless form, they read
\begin{subequations}
\begin{align}
  \nabla\cdot\mathbf{u} &= 0, \\
  \frac{\partial\mathbf{u}}{\partial t} + (\mathbf{u}\cdot\nabla)\mathbf{u} &= -\nabla p + \frac{1}{\text{Re}}\nabla^2\mathbf{u} + \mathbf{f}, \label{eq:mom_ns}
\end{align}
\end{subequations}
\noindent where $\mathbf{u}=\mathbf{u}(\mathbf{x},t)$ is the flow velocity vector field, $p=p(\mathbf{x},t)$ 
is the pressure scalar field, Re is the (constant and uniform) Reynolds number 
$\text{Re} = UL/\nu$, where $U$ is a reference velocity, $L$ is a reference 
lengthscale, and $\nu$ is the kinematic viscosity; and $\mathbf{f}(t)$ is the 
forcing term.

Following the notation of \cite{Trias2014}, after spatial discretization with
a finite-volume method, the incompressible
Navier-Stokes equations can be expressed in a semi-discrete formulation as
\begin{subequations}
  \begin{align}
    M\mathbf{u} &= \mathbf{0}, \label{eq:ns-a}\\
    \Omega\frac{d\mathbf{u}}{dt} + C(\mathbf{u})\mathbf{u}
    - D\mathbf{u} + \Omega G\mathbf{p} - \Omega\mathbf f &= \mathbf{0}, \label{eq:ns-b}
  \end{align}
  \label{eq:ns}
\end{subequations}
\noindent where $C$ is the convective operator, $D$ is the diffusive operator,
$G$ is the gradient operator, $M$ is the divergence operator, $\mathbf f$ is the discrete forcing term, and $\Omega$ is a diagonal matrix with
the volumes of the cells as entries; $\mathbf{u}=\mathbf u(t), \mathbf{p}=\mathbf p(t)$ are the velocity and
kinematic pressure.  Note that this formulation can be easily adapted to methods that are not
the finite-volume method. The exact form of these operators depends on the details of the chosen discretization.
Within the framework of model order reduction, Eqs. \eqref{eq:ns-a}-\eqref{eq:ns-b} are
identified as the FOM.

Eqs. \eqref{eq:ns-a}-\eqref{eq:ns-b} represent a system of differential-algebraic equations (DAE) of index 2, as the incompressibility condition is an algebraic constraint. Assuming no contribution from boundary conditions, this
system reads as follows,
\begin{subequations}
  \begin{align}
    M\mathbf{u} &= 0, \\
    \frac{d\mathbf{u}}{dt} &= \tilde{F}(\mathbf{u},t) - G\mathbf{p},
  \end{align}
  \label{eq:ns_dae}
\end{subequations}
\noindent where $\tilde{F}(\mathbf{u},t) = \Omega^{-1}(-C(\mathbf{u})+D)\mathbf{u} + \mathbf f$.

From this point onwards, the discretization of the operators is assumed to be
performed in a symmetry-preserving manner for the sake of simplicity. Symmetry-preserving schemes
\cite{Verstappen2003, Trias2014} are a family of schemes that preserve the symmetry properties of the 
continuous operators at the discrete level. This is usually achieved by ensuring that the diffusive operator is discretized as a symmetric matrix and corresponds to the composition of the gradient and divergence operators, resulting in a negative semi-definite matrix, the convective operator corresponds to a skew-symmetric matrix, and the discrete gradient and divergence operators are dual to each other, i.e., $-\Omega M^T = G$.  In physical terms, this duality ensures that the pressure-gradient work and the velocity divergence cancel exactly in the global energy balance, preventing artificial energy sources or sinks in the discrete system {\cite{Verstappen2003}}.

\subsection{Linear stability analysis for the FOM}
To integrate Eq.\eqref{eq:ns} in time, explicit one-step methods, such as any method of the
explicit Runge-Kutta (ERK) family applied to projection methods \cite{chorin_numerical_1967}, are considered. For further details 
on how this integration is performed in the context of the incompressible Navier-Stokes equations, the reader is
referred to \cite{sanderse_accuracy_2012}. A similar methodology
for explicit multi-step methods, such as the second-order Adams-Bashforth (AB2), is developed in \cite{trias_self-adaptive_2011}.

For a general ERK applied to Eq.\eqref{eq:ns}, the methodology can be contracted to
\begin{equation}
\mathbf{u}^{n+1} \approx R(\Delta t F)\mathbf{u}^n,
\label{eq:aug_fn}
\end{equation}
\noindent where $F=\Omega^{-1}(D-C(\mathbf{u}))\in\mathbb{R}^{m\times m}$ is the operator arising from linearizing $\tilde{F}(\mathbf{u})$ in Eq.\eqref{eq:ns_dae} via Picard linearization, and $R(z)$ is the stability (or augmentation) function of the given
(linearized) scheme. Note that there is an approximation symbol instead of an equal symbol as the $n+1$ value requires projecting every stage in the set of incompressible velocity fields \cite{sanderse_accuracy_2012,agdestein_discretize_2025}, which is left out in Eq.\eqref{eq:aug_fn}.
However, the projection step in a projection method is implicit and does not affect the stability region in a zero-divergence framework.

Following linear stability analysis theory for 
differential equations, the linear method is stable \cite{hairer_solving_1996} if 
\begin{equation}
  ||R(\Delta tF)||_2 \leq 1.
\end{equation}
\noindent As explained in \cite{hairer_solving_1996}, this requirement can be equivalently expressed on the largest in magnitude eigenvalue $\lambda_F$ of $F$,
\begin{equation} 
|R(\Delta t\lambda_F)| \leq 1.
\end{equation}

The maximum stable timestep for a Runge-Kutta integration is such that, for $z_{max}=\Delta t||\lambda_F||_2$, it fulfills the equation $||R(z_{max})||_2=1$. By doing so, the eigenvalue $\Delta t\lambda_F$ is placed over the stability region (Figure \ref{fig:bendixson}). $z_{max}$ is obtained by solving $||R(z)||_2-1=0$ using a bisection method along the axis defined by the origin of coordinates ($0+0i)$ and $\lambda_F$ every timestep. Hence, the self-adaptive timestep can be computed as
\begin{equation}
\Delta t = \frac{z_{max}}{||\lambda_F||_2}.
  \label{eq:dt}
\end{equation}

Thus, in general, the goal of any self-adaptive timestepping method is to calculate (or estimate) the eigenbounds of $F$ so that the maximum stable timestep is computed using Eq.\eqref{eq:dt}.

While the incompressible Navier-Stokes equations are inherently nonlinear, the ultimate goal of this analysis is to ensure numerical stability. This approach is justified by the fact that numerical stability is fundamentally a local phenomenon. By performing a Taylor expansion of the nonlinear spatial operator around the current state $\mathbf u^n$, $\dot{\mathbf u}\approx F(\mathbf u^n)+\frac{\partial F}{\partial \mathbf u}(\mathbf u_n)(\mathbf u-\mathbf u^n)$, it is evident that the error growth is dictated by the eigenvalues of the linearized operator rather than the global nonlinear system. This justifies that these methods are restricted to linear stability analysis instead of introducing nonlinear methods.

The methodology explained in this subsection assumes the discretization is symmetry-preserving. The reader is referred to Appendix B of
\cite{trias_efficient_2024} to generalize the method to
non-symmetry-preserving schemes, such as upwinding schemes, by splitting the
convective operator in its symmetric and skew-symmetric contributions. As explained in Section \ref{sec:fom}, the symmetry-preserving discretization ensures that certain property from the continuous operators are preserved \cite{Verstappen2003}: a symmetry-preserving diffusive operator is represented
by a real semi-negative definite symmetric matrix, and the convective operator is represented by
a real skew-symmetric matrix \cite{Verstappen2003, Trias2014}. Hence, the eigenvalues of the former lie on the negative real axis in the complex plane, while the eigenvalues of the latter lie on the imaginary axis.
As eigenbounds are subadditive, i.e. $\rho(A+B)\leq\rho(A)+\rho(B)$, where $\rho(A)$ determines the eigenbounds of the matrix $A$, the
eigenbound of the operator $F$, $\lambda_F$, lies, according to Bendixson's inequality 
\cite{bendixson_sur_1902}, in the rectangle defined by the eigenbounds of the convective and diffusive operators as
shown in Figure \ref{fig:bendixson}. 

\begin{figure}[h]
  \centering
  \begin{tikzpicture}
    \draw[->] (-3,0) -- (3,0) node[right] {Re};
    \draw[->] (0,-3) -- (0,3) node[above] {Im};
    \fill[red!70,fill opacity=0.3] (-2.15,-1.8) rectangle (0,1.8);
    \draw (-2.15,-1.8) rectangle (0,1.8);
    \fill[gray, fill opacity=0.5] (-1.2,-1) rectangle (0,1);
    \draw (-1.2,-1) rectangle (0,1);
    \node (c) at (-1.2, -1.25) {$I$}; %Omega^{-1}D
    \node (c) at (0.3,1) {$II$}; %\rho(\Omega^{-1}C(\mathbf{u}))$
    \node (c) at (0.3,-1) {$III$}; %-\rho(\Omega^{-1}C(\mathbf{u}))$
    \draw[red,->,ultra thick] (0,0) -- (-2.15,1.8) node[above] {\hspace{-0.65in}$z_{max}=\lambda_F\Delta t$};
    \draw[->,thick] (0,0) -- (-1.2,1) node [midway, above] {\hspace{0.1in} $\lambda_F$}; 

\draw[blue, thick, smooth] 
    plot coordinates {
    (-3.015384615384615, 0)
    (-2.993846153846154, 0.26526315789473683)
    (-2.943589743589743, 0.48631578947368415)
    (-2.886153846153846, 0.6568421052631579)
    (-2.807179486179487, 0.8273684210526315)
    (-2.7210256410256413, 0.991578947368421)
    (-2.6492307692307697, 1.0989473684210526)
    (-2.5199999999999996, 1.2947368421052632)
    (-2.369230769230769, 1.4968421052631578)
    (-2.1825637021821025, 1.7615789473684215)
    (-2.0389743589743588, 1.9768421052631578)
    (-1.8882051282051282, 2.2105263157894736)
    (-1.723076923076923, 2.419)
    (-1.4933333333333333, 2.6399999999999997)
    (-1.3282051282051278, 2.7473684210526313)
    (-1.0912820512820509, 2.829473684210526)
    (-0.8902564102564098, 2.854736842105263)
    (-0.7107692307692305, 2.8357894736842105)
    (-0.5528205128205125, 2.7852631578947367)
    (-0.4307692307692303, 2.7221052631578947)
    (-0.287179486179587, 2.6084210526315787)
    (-0.1794871794871794, 2.482105263157895)
    (-0.06461538461538438, 2.286315789473684)
    (-0.007179487179486621, 2.121578947368421)
    (0.0358974358975143, 1.8757894736842104)
    (0.05025641025641061, 1.711578947368421)
    (0.05025641025641061, 1.534736842105263)
    (0.04307692307692386, 1.3326315789473683)
    (0.0358974358975143, 1.073684210526316)
    (0.02871794871741858, 0.8463157894736842)
    (0.021538461538461996, 0.6447368421052632)
    (0.01435853435817484, 0.35368421052631574)
    (0.007179487179487687, 0.20210526315789474)
    (0, 0)
    };
\draw[blue, thick, smooth] 
    plot coordinates {
    (-3.015384615384615, 0)
    (-2.993846153846154, -0.26526315789473683)
    (-2.943589743589743, -0.48631578947368415)
    (-2.886153846153846, -0.6568421052631579)
    (-2.807179486179487, -0.8273684210526315)
    (-2.7210256410256413, -0.991578947368421)
    (-2.6492307692307697, -1.0989473684210526)
    (-2.5199999999999996, -1.2947368421052632)
    (-2.369230769230769, -1.4968421052631578)
    (-2.1825637021821025, -1.7615789473684215)
    (-2.0389743589743588, -1.9768421052631578)
    (-1.8882051282051282, -2.2105263157894736)
    (-1.723076923076923, -2.419)
    (-1.4933333333333333, -2.6399999999999997)
    (-1.3282051282051278, -2.7473684210526313)
    (-1.0912820512820509, -2.829473684210526)
    (-0.8902564102564098, -2.854736842105263)
    (-0.7107692307692305, -2.8357894736842105)
    (-0.5528205128205125, -2.7852631578947367)
    (-0.4307692307692303, -2.7221052631578947)
    (-0.287179486179587, -2.6084210526315787)
    (-0.1794871794871794, -2.482105263157895)
    (-0.06461538461538438, -2.286315789473684)
    (-0.007179487179486621, -2.121578947368421)
    (0.0358974358975143, -1.8757894736842104)
    (0.05025641025641061, -1.711578947368421)
    (0.05025641025641061, -1.534736842105263)
    (0.04307692307692386, -1.3326315789473683)
    (0.0358974358975143, -1.073684210526316)
    (0.02871794871741858, -0.8463157894736842)
    (0.021538461538461996, -0.6447368421052632)
    (0.01435853435817484, -0.35368421052631574)
    (0.007179487179487687, -0.20210526315789474)
    (0, 0)
    };

  \end{tikzpicture}
  \caption{Schematic illustration of Bendixson's rectangle (gray) \cite{trias_self-adaptive_2011} for the eigenvalues of $F$. $I: \rho(\Omega^{-1}D)$, $II:\rho(\Omega^{-1}C(\mathbf u))$, $III: -\rho(\Omega^{-1}C(\mathbf u))$. The red rectangle represents the stability region bounds. The blue curve represents the boundary of the stability region for 
  a third-order Runge-Kutta scheme.}
  \label{fig:bendixson}
\end{figure}
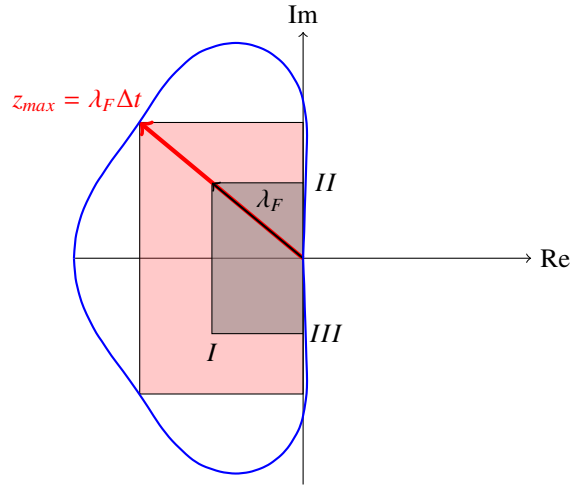

The stability region is determined by the timestepping scheme. For ERK schemes defined by the coefficients $A=[a_{ij}]_{i,j=1:s}$ and $\mathbf{b}
= (b_1~b_2~\dots~b_s)^T$, with $s$ stages \cite{kraaijevanger_contractivity_1991}, if $s\leq 4$ and the
order of accuracy $p=s$, the stability function is the Taylor polynomial of $e^z$ up to the number of stages \cite{hairer_solving_1993},
\begin{equation}
  R(z) = 1 + \sum_{j=1}^s{\frac{1}{j!}z^j},
\end{equation}
\noindent which makes the implementation much easier to test. Moreover, the
most widely used RK schemes: Heun's 2nd- and 3rd-order schemes, Kutta's 3rd
order scheme, the classic 4th-order scheme; satisfy this condition of
$s=p\leq 4$.

\subsection{Existing self-adaptive timestepping techniques for the FOM}
\label{sec:existing}

As previously detailed, the goal of a self-adaptive timestepping technique is determining the eigenbounds of the operators as accurately as possible, given that the exact computation of the eigenvalues is prohibitive for a classical FOM problem size, which
can reach up to several billion unknowns.

\texttt{EigenCD} \cite{trias_self-adaptive_2011} estimates the eigenbounds by making use
of the Gershgorin circle theorem \cite{gerschgorin_uber_1931}, which is relatively inexpensive compared to
the exact computation of the eigenvalues, while simultaneously providing a relatively accurate bound so that it
allows having a much larger timestep than the original CFL condition \cite{courant_uber_1928}. To do so, the method relies on the explicit construction of the convective and diffusive operators at every timestep, which allows the application of the Gershgorin circle theorem to both matrices to compute the eigenbounds of the operators.

In contrast to \texttt{EigenCD}, \texttt{AlgEigCD} \cite{trias_efficient_2024} exploits the structure of the convective and diffusive operators to estimate the eigenbounds without explicitly constructing these operators. On a collocated grid, the duality
of the gradient and the divergence, and the definition of the interpolation of the velocity at the faces for a symmetry-preserving scheme allows the (offline) computation of a matrix depending only on the incidence matrices as well as $\Omega$. Therefore, the eigenbounds are computed online by simply performing a matrix-vector product of this offline-computed matrix with the diffusive and convective fluxes.
This construction, opposite to the one from \texttt{EigenCD}, allows simplifying the estimation of
the eigenbound to a sparse matrix-vector product, where the sparse matrix has been computed \textit{a priori}. This allows making the eigenbound estimation much more efficient.

The main question in this article is whether the \texttt{EigenCD}/\texttt{AlgEigCD} method can also be applied in a reduced-order model context. For this purpose, the reduced-order model framework is introduced in the following subsection.

\subsection{Energy-conserving POD-Galerkin method} 
\label{sec:rom}

The ROM considered in this work follows the ODE-based projection approach 
\cite{benner_survey_2015}, where the PDE is first discretized in space, as in
Eq.\ \eqref{eq:ns}, and afterwards is projected to the subspace obtained with
POD \cite{sanderse_non-linearly_2020}.
Let the velocity vector be defined as $\mathbf{u}\in\mathbb{R}^{N_V}$, and the POD-Galerkin approximation is obtained by the Ansatz that the velocity field
can be approximated by
\begin{equation}
\mathbf{u}_c(t) \approx \mathbf{u}_r(t) \equiv \Phi\mathbf{a}(t),
\label{eq:ansatz}
\end{equation}
where $\Phi\in\mathbb{R}^{N_V\times M}$ is the projection matrix, and
$\mathbf{a}(t)\in\mathbb{R}^M$ is the coefficient vector; assuming $M \ll
N_V$. In this work, the basis is obtained from the
so-called snapshot matrix $X$ which contains $K$ snapshots of the velocity
field, $X = [\mathbf{u}_c^1 \dots
\mathbf{u}_c^n\dots\mathbf{u}_c^K]\in\mathbb{R}^{N_V\times K}$, obtained from
a (so-called `offline') solution of the FOM from Eq.\ \eqref{eq:ns}.

The construction of the POD basis is performed using a weighted
orthonormality condition, $\Phi^T\Omega\Phi = I_M$, where $I_M$ is the identity
matrix of size $M$, which is consistent with the inner product of the velocity
as well as the ROM kinetic energy equation \cite{sanderse_non-linearly_2020}.
As \texttt{AlgEigCD} provides an efficient way to integrate the FOM using variable timesteps, it is important to construct the POD basis by considering that there is
a variable timestep \cite{kunisch_optimal_2010}. This is taken into account in the construction of the POD basis as described in \ref{sec:pod}.

Given $\Phi$, and the Ansatz, both mass and momentum equations from Eq.\
\eqref{eq:ns} can be projected \cite{rosenberger_no_2023} to obtain the reduced-order
equations. Concerning the conservation of mass, this is preserved automatically if the
snapshots are divergence free ($MX_j=0$), regardless of $\mathbf{a}$.

First, consider a case with homogeneous boundary conditions. Concerning
the conservation of momentum, the ROM momentum equation is given by
\begin{equation}
\frac{d\mathbf{a}}{dt}
= -\Phi^TC(\Phi\mathbf{a}_c)\Phi\mathbf{a}+\Phi^TD\Phi\mathbf{a},
\label{eq:rom_inc}
\end{equation}
\noindent where the weighted orthonormality condition has been applied.  $\mathbf{a}$ stands for the convected velocity in the reduced space while $\mathbf{a}_c=\mathbf a$ stands for the convecting velocity in the reduced space. While physically identical in the momentum equation, this notation highlights the functional roles as the state vector and the linearization point, respectively. This separation ensures the framework remains generalizable to non-autonomous transport equations (see \ref{sec:jacobian}), where the convecting field may be a prescribed velocity. Applying the duality of the divergence and gradient operators inherent in
a symmetry-preserving discretization, the pressure term is exactly zero, and can
be dropped if the snapshots are divergence-free
\cite{sanderse_non-linearly_2020}, as shown by
\begin{equation}
    \Phi^TG\mathbf p = -\Phi^T\Omega M^T\mathbf p = (M\Omega\Phi)^T\mathbf p = 0.
\end{equation}

In the current setup, building the reduced convective operator in an online phase would rely
on operations at the FOM level, which
would notably deteriorate the performance of the ROM. Following
\citet{stabile_finite_2018}, the reduced convective term can be rewritten as
\begin{equation}
\Phi^TC(\Phi\mathbf{a}_c)\Phi\mathbf{a} \equiv
[\underbrace{\Phi^TC(\Phi_1)\Phi}_{C_{r,1}}~C_{r,2}~\dots~C_{r,M}](\mathbf{a}_c\otimes\mathbf{a})
  = C_r(\mathbf{a}_c\otimes\mathbf{a}) = \left(\sum_{i=1}^M{a_{c,i}C_{r,i}}\right)\mathbf{a}.
  \label{eq:red_conv}
\end{equation}
\noindent With this approach, all the reduced convective matrices $C_{r,i}$ can be computed
offline in a setup phase that keeps the complexity of online operations at
ROM level, while all FOM-size operations are performed offline (only once).

Hence, the ROM derived in Eq.\eqref{eq:rom_inc} can be written as
\begin{equation}
\frac{d\mathbf{a}}{dt}
  = -C_r(\mathbf{a}_c\otimes\mathbf{a})+D_r\mathbf{a}+F_{r,0},
  \label{eq:rom_final}
\end{equation}
\noindent which has a complexity of $\mathcal{O}(M^3)$ \cite{ahmed_closures_2021}. In many cases, the required number of modes $M$ is low enough so that performing operations of this complexity is acceptable.
% \footnotemark[1].
For larger $M$, hyperreduction techniques such as the discrete empirical interpolation method might be used \cite{klein_energy-conserving_2024,farhat_structure-preserving_2015,barrault_empirical_2004} to reduce the complexity to $\mathcal O(M^2)$. While hyperreduction specifically targets this complexity reduction, for flows at high Reynolds numbers, closure models, e.g., {\cite{san2018extreme}}, can also be employed. These models ensure stability at very low values of M by modeling the dissipative effects of truncated scales, offering a complementary approach to maintain ROM feasibility.

In the case of non-homogeneous boundary conditions, a linear term arising from the convective term
interaction with the boundary conditions should be added to Eq.\eqref{eq:rom_final} \cite{rosenberger_no_2023}.
To consider the presence of non-homogeneous boundary conditions, the contribution from these non-homogenenous boundary conditions ($\mathbf u_\mathrm{inhom}$) has to be added to the ansatz. Therefore,
\begin{equation}
    \mathbf{u}_r(t) = \Phi\mathbf{a}(t)+\mathbf{u}_\mathrm{inhom}(t),
    \label{eq:inhom_ansatz}
\end{equation}
\noindent where $\mathbf{u}_\mathrm{inhom}$ is the lifting function defined by \citet{rosenberger_no_2023}. As explained in Appendix B in \cite{sanderse_non-linearly_2020}, the Navier-Stokes equations with inhomogeneous boundary conditions are given by
\begin{align}
    M\mathbf{u} &= \mathbf{y}_M \\
    \Omega\frac{d\mathbf{u}}{dt} &= -K((\Pi\mathbf{u}+\mathbf{y}_\Pi)\circ (A\mathbf{u}+\mathbf{y}_A)) + (D\mathbf{u}+\mathbf{y}_D) - (\Omega G\mathbf{p} + \mathbf{y}_G) + \mathbf{f},
    \label{eq:mom eq with inhom bc}
\end{align}
\noindent where $\Pi$ interpolates the fluxes to the staggered faces, and $A$ interpolates the velocities to the staggered faces. Figure {\ref{fig:stagg}} represents a staggered grid, where the velocities are calculated at the faces of the collocated grid (gray), and the staggered faces are represented in blue for the $x$-velocity component and in red for the $y$-velocity component.

\begin{figure}[h]
    \centering
\begin{tikzpicture}[scale=3, >=stealth]
    % Draw the main grid
    \draw[step=1cm, gray, ultra thick] (0,0) grid (2,2);
    \draw[blue,dashed] (0.5,0) rectangle (1.5,2);
    \draw[blue,dashed] (0,0) -- (2,0);
    \draw[blue,dashed] (0,1) -- (2,1);
    \draw[blue,dashed] (0,2) -- (2,2);

    \draw[red,dashed] (0,0.5) rectangle (2,1.5);
    \draw[red,dashed] (0,0) -- (0,2);
    \draw[red,dashed] (1,0) -- (1,2);
    \draw[red,dashed] (2,0) -- (2,2);

    % Define Styles
    \tikzset{
        u_vel/.style={->, blue, thick},
        v_vel/.style={->, red, thick},
        p_node/.style={fill=black, circle, inner sep=1pt}
    }

    % 1. Pressure points (Centers)
    \foreach \x in {0.5, 1.5}
        \foreach \y in {0.5, 1.5}
            \node[p_node, label={below right:$p$}] at (\x,\y) {};

    % 2. U-velocity arrows (Vertical faces/Left-Right edges)
    % Placed at (i, j+0.5)
    \foreach \x in {0, 1, 2}
        \foreach \y in {0.5, 1.5}
            \draw[u_vel] (\x-0.2, \y) -- (\x+0.2, \y) node[above] {$u$};

    % 3. V-velocity arrows (Horizontal faces/Top-Bottom edges)
    % Placed at (i+0.5, j)
    \foreach \x in {0.5, 1.5}
        \foreach \y in {0, 1, 2}
            \draw[v_vel] (\x, \y-0.2) -- (\x, \y+0.2) node[right] {$v$};

\end{tikzpicture}
\caption{Representation of a two-dimensional staggered grid, where the collocated grid is depicted in gray, the $x$-staggered grid in blue, and the $y$-staggered grid in red.}
\label{fig:stagg}
\end{figure}

As described in \cite{rosenberger_advances_2021}, the boundary contributions $\mathbf{y}_*, \, *\in\{A,D,\Pi,G,M\}$ can be expressed as linear functions of a vector $\mathbf{y}_{bc}$ that fully describes the prescribed boundary conditions, $\mathbf{y}_* = B_* \mathbf{y}_{bc}$. The boundary condition vector is approximated as $\mathbf{y}_{bc}(t)\approx \Phi_{bc} \mathbf{a}_{bc}(t)$, where $\mathbf{a}_{bc}\in\mathbb{R}^{M_{bc}}$, and $M_{bc}$ being the number of modes used for the boundary conditions. Inserting this approximation and the Ansatz \eqref{eq:inhom_ansatz} with $\mathbf{u}_\mathrm{inhom}(t) = F_\mathrm{inhom}\mathbf{a}_{bc}(t)$ into \eqref{eq:mom eq with inhom bc} (projected onto $\Phi$), leads to
\begin{align}
    \frac{d\mathbf{a}}{dt} = \Phi^T\Big[&
    -K((\Pi\Phi \mathbf{a}_c + (\Pi B_\mathrm{inhom} + B_\Pi\Phi_{bc})\mathbf{a}_{bc})\circ(A\Phi \mathbf{a} + (AB_\mathrm{inhom} + B_A\Phi_{bc})\mathbf{a}_{bc})) \nonumber \\
    &+ (D\Phi \mathbf{a} + (DB_\mathrm{inhom}+B_D\Phi_{bc})\mathbf{a}_{bc}) - B_G\Phi_{bc} \mathbf{a}_{bc} + \mathbf{f}_c
    \Big].
\end{align}
Here, $\Phi^T G_c \mathbf{p}_c = 0$ has already been eliminated.
The convection operator can be decomposed into a term quadratic in $\mathbf a$,
\begin{align}
    -\Phi^T K((\Pi\Phi \mathbf{a}_c)\circ(A\Phi \mathbf{a} )) =: C_\mathrm{hom}(\mathbf{a}_c\otimes \mathbf{a}),
    \label{eq:chom_ac_a}
\end{align}
a term linear in both, $\mathbf{a}$ and $\mathbf{a}_{bc}$,
\begin{align}
    -\Phi^T K((\Pi\Phi \mathbf{a} )\circ((AB_\mathrm{inhom} + B_A\Phi_{bc})\mathbf{a}_{bc}))
    -K(((\Pi B_\mathrm{inhom} + B_\Pi\Phi_{bc})\mathbf{a}_{bc})\circ(A\Phi \mathbf{a})) =: C_\mathrm{hom,bc}(\mathbf{a}\otimes \mathbf{a}_{bc}),
    \label{eq:chombc_a_abc}
\end{align}
and a term quadratic in $\mathbf{a}_{bc}$,
\begin{align}
    -\Phi^T K(((\Pi B_\mathrm{inhom} + B_\Pi\Phi_{bc})\mathbf{a}_{bc})\circ( (AB_\mathrm{inhom} + B_A\Phi_{bc})\mathbf{a}_{bc})) =: C_{bc}(\mathbf{a}_{bc}\otimes \mathbf{a}_{bc}).
    \label{eq:cbc_abc_2}
\end{align}
Therefore, with the additional terms from Eqs.\eqref{eq:chom_ac_a}-\eqref{eq:cbc_abc_2}, the ROM reads
\begin{equation}
    \frac{d \mathbf a}{dt} = -C_\mathrm{hom}(\mathbf a_c\otimes\mathbf a) - C_\mathrm{hom,bc}(\mathbf a \otimes \mathbf a_{bc}) - C_{bc}(\mathbf a_{bc}\otimes\mathbf a_{bc}) + D_r\mathbf a + F_{r,0}.
    \label{eq:rom_nonhom}
\end{equation}
Therefore, considering Eq. \eqref{eq:rom_nonhom}, the linear contribution from the convective term in $\mathbf a$ is
\begin{equation}
  C_l(t) = \frac{\partial C_\mathrm{hom,bc}(\mathbf a \otimes \mathbf a_{bc})}{\partial\mathbf a}  = C_\mathrm{hom,bc}(I_{M}\otimes\mathbf{a}_{bc}(t)) = [C_{\mathrm{hom,bc},1}~\dots~C_{\mathrm{hom,bc},M}] (I_{M_{bc}}\otimes\mathbf{a}_{bc}(t)) = \sum_{i=1}^{M_{bc}}{a_{bc,i}(t)C_{l,i}},
  \label{eq:cl}
\end{equation}
\noindent where $C_\mathrm{hom,bc}\in\mathbb{R}^{M\times MM_{bc}} = [C_{\mathrm{hom,bc},1} \dots C_{\mathrm{hom,bc},M}]$ is the matrix containing
the reduced contribution of the boundary conditions to the convective term, and 
$\mathbf{a}_{bc}(t)\in\mathbb{R}^{M_{bc}}$ are the coefficients of the boundary condition at
time $t$ \cite{rosenberger_no_2023}.

% Henrik's suggestion:
% I suggest to simplify equations(18) and (19) to 
% \begin{align}
%     \frac{d\mathbf a}{dt} = C_\mathrm{hom}(\mathbf a\otimes \mathbf a) + C_\mathrm{hom,bc}(\mathbf a \otimes \mathbf a_{bc}) + C_{bc}(\mathbf a_{bc} \otimes \mathbf a_{bc}) + D_r \mathbf a + Y\mathbf a_{bc} + \mathbf f_c
% \end{align}
% with $D_r = \Phi^T D\Phi\in \mathbb{R}^{M\times M}$ (as before) and $Y = \Phi^T(DF_\mathrm{inhom}+F_D\Phi_{bc}-F_G\Phi_{bc})\in\mathbb R^{M\times M_{bc}}$.

% <<< End input from Preliminaries.tex
% >>> Begin input from redeig.tex
\section{Generalizing the self-adaptive timestep methods for ROMs}
\label{sec:redeig}

\subsection{Estimating the eigenbound of the reduced operators}

As for FOMs discussed in Section \ref{sec:existing}, the crucial element of self-adaptive timestepping methods for ROMs is the computation of eigenbounds. Since the convection operator depends on the current velocity, these eigenbounds must be computed for every timestep. 
To preserve the efficiency gain achieved by the adaptive timestepping, the computational costs of this eigenbound computation should be negligible compared to the overall costs of the numerical simulation. As described in Section \ref{sec:rom}, the simulation of each timestep has computational costs of $\mathcal O(M^3)$ if the convection term is computed exactly, and of $\mathcal O(M^2)$ if hyperreduction is used.
% Hence, the computational cost for computing the eigenbounds at each time step should be at most in $\mathcal O(M^2)$.
Although the present work focuses on the standard projection without hyperreduction (in order to ensure exact energy conservation and stability properties), we keep the requirement that the computational cost for computing the eigenbounds at each time step should be at most $\mathcal O(M^2)$.

The exact computation of the eigenvalues of a matrix in $\mathbb{R}^{M\times M}$ has computational costs of $\mathcal O(M^3)$ \cite{wilkinson_qr_1965,corbato_coding_1963,trefethen_numerical_1997} which is too expensive for practical purposes.
As an alternative, Gershgorin's circle theorem can be applied to the reduced operators as in \texttt{EigenCD} \cite{trias_self-adaptive_2011} to estimate the eigenbounds. This approach has costs in $\mathcal O(M^2)$ so still might contribute significantly to the overall costs. Moreover, this estimate can be quite inaccurate and hence result in timesteps significantly smaller than the largest stable timestep based on the exact eigenbounds.
Another alternative is \texttt{AlgEigCD} \cite{trias_efficient_2024}, which was proposed on the FOM level and leverages the structure of the convective and diffusive operators. A natural question that we will address now is whether this method can be transferred to the ROM level.

The concept behind \texttt{AlgEigCD} is approximating the eigenbounds of the operators by finding a factorization $H=AB$ such that the eigenbound of $A^TB^T$ can be approximated efficiently. The reduced operators, however, have the form $H=\Phi^TH_F$, where $H_F$ is a matrix with FOM dimensions. Following \texttt{AlgEigCD}, the eigenbound of a matrix $H$ would be obtained by approximating the eigenbounds of $H^*=(\Phi^TA_F)^T(B_F\Phi)^T$, where $A_FB_F$ is a factorization of $H_F$. In contrast to $H$, $H^*$ is a matrix of FOM dimension, therefore the eigenbound estimation for $H^*$ is more expensive than for $H$ itself. Thus, such a transfer does not result in an efficient method. Instead, we will propose a new approach that exploits the structure of the reduced operators \eqref{eq:red_conv} and \eqref{eq:cl}.

Concerning the diffusive operator, the application of the exact
calculation of the eigenvalues is straightforward, as the diffusive operator at
the ROM level is built as $\Phi^TD\Phi$, and it is a real symmetric
matrix. Thus, its eigenvalues are real, and the eigenbound
$\rho(D_r)$ can be computed exactly in an offline stage.
On the other hand, the convective operator at the ROM level is built as in Eq.
\eqref{eq:red_conv}. Simplifying the momentum equation as a particular case of a convection-diffusion equation with predefined convecting velocity, the Jacobian of the convective operator with respect to the convected velocity coefficients reads as
\begin{equation}
  J_C = \frac{\partial C_r}{\partial \mathbf{a}}= 
  \sum_{i=1}^M{a_{c,i}C_{r,i}},
\end{equation}
\noindent which yields a skew-symmetric Jacobian iff $C_{r,i}$ is skew-symmetric for all $i$. \ref{sec:jacobian} shows the computation of the Jacobian of the convective term for a general convection-diffusion equation. Thus, the linearized form that is used for the linear stability analysis corresponds to
\begin{equation}
  \sum_{i=1}^M{a_{c,i}C_{r,i}}.
\end{equation}

Note that at every timestep
$\mathbf{a}_c$ changes, yet $C_{r,i}$ does not as it only depends on
$\Phi_i$. Thus, the eigenbound of the reduced convective operator can be estimated by employing the property that eigenbounds are subadditive. Hence, the eigenbound of the reduced convective operator can be bound as
\begin{equation}
  \rho(C_r) \leq \sum_{i=1}^M{|a_{c,i}|\rho(C_{r,i})},
  \label{eq:ebound_Cr}
\end{equation}
\noindent where $\rho(C_{r,i})$ is calculated exactly in an offline
stage, as its complexity of $\mathcal{O}(M^3)$ can be tolerated. However, the online complexity to compute Eq.\eqref{eq:ebound_Cr} is just $\mathcal O (M)$, as it corresponds to a dot product. This complexity is negligible compared to the $\mathcal O(M^3)$ cost of evaluating the ROM every timestep (Eq.\eqref{eq:rom_final}). 

Even though Eq.\eqref{eq:ebound_Cr} only holds for skew-symmetric convective operators, it can be generalized to cases that do not fulfill this condition, i.e., upwinding schemes, by splitting each $C_{r,i}$ into its symmetric and skew-symmetric contributions similarly to Eq.\eqref{eq:split}, and therefore applying the same exact method for the symmetric (real) and skew-symmetric (imaginary) contributions.

In case of non-homogeneous boundary conditions, the ROM additionally contains a linear convective term which is in general neither skew-symmetric nor symmetric \cite{rosenberger_no_2023}. Given that this term
is time-dependent in case of time-dependent boundary conditions, its eigenvalues
should be estimated at every timestep.
As the non-homogeneous boundary conditions contributions are neither symmetric nor skew-symmetric in general, they are split into their symmetric and skew-symmetric contributions, i.e.,
\begin{subequations}
    \begin{align}
        C_{l,i}^S &= \frac{1}{2}(C_{l,i}+C_{l,i}^T), \\
        C_{l,i}^\Omega &= \frac{1}{2}(C_{l,i}-C_{l,i}^T).
    \end{align}
    \label{eq:split}
\end{subequations}
\noindent Therefore, the symmetric term contributes to the real bound, while the skew-symmetric term affects the imaginary bound.

The linear contribution (Eq.\eqref{eq:cl}) has the same structure as the convective operator (Eq.\eqref{eq:red_conv}. Therefore, the same approach can be used to estimate the eigenbounds for the symmetric and skew-symmetric terms, which leads to
\begin{subequations}
    \begin{align}
        \rho(C_l^S(t)) &\leq \sum_{i=1}^{M_{bc}}{|a_{bc,i}(t)|\rho(C_{l,i}^S)}, \\
        \rho(C_l^\Omega(t)) &\leq \sum_{i=1}^{M_{bc}}{|a_{bc,i}(t)|\rho(C_{l,i}^\Omega)},
    \end{align}
\end{subequations}
\noindent where $\rho(C_{l,i}^S)$ and $\rho(C_{l,i}^\Omega)$ are computed exactly in an offline stage, and $\mathbf{a}_{bc}(t)$ is already known at every timestep. Similarly to Eq.\eqref{eq:ebound_Cr}, the online complexity of these estimations is $\mathcal O(M_{bc})$. 

Hence, this new approach to estimate the eigenbounds used to determine the largest stable timestep is
determined by
\begin{subequations}
  \begin{align}
    \Re(\lambda_F) &\leq \tilde\Re(\lambda_F) :=  \rho(D_r) + \sum_{i=1}^{M_{bc}}{|a_{bc,i}(t)|\rho(C_{l,i}^S)}, \\
    \Im(\lambda_F) &\leq \tilde\Im(\lambda_F) :=\sum_{i=1}^M{|a_{c,i}|\rho(C_{r,i})}+\sum_{i=1}^{M_{bc}}{|a_{bc,i}(t)|\rho(C_{l,i}^\Omega)}.
  \end{align}
  \label{eq:ebound_F}
\end{subequations}
With all these ingredients, the stability-aware self-adaptive timestepping method for ROMs, \texttt{RedEigCD}, can be
summarized in Algorithm \ref{alg:redeigcd}, where $T$ is the simulation end time. 
\begin{algorithm}
  \caption{Stability-aware self-adaptive timestep method for ROMs (\texttt{RedEigCD})}
  \label{alg:redeigcd}
  \textbf{Offline stage:}
  \begin{algorithmic}[1]
    \State Compute $\Phi$ from snapshots.
    \State Compute $D_r$ as $\Phi^TD\Phi$.
    \State Compute $C_{r,i}$ as in Eq.\eqref{eq:red_conv} for $i=1,\dots,M$.
    \State Compute $C_{l,i}$ as in Eq.\eqref{eq:cl} for $i=1,\dots,M_{bc}$.
    \State $\star$ Compute $\rho(D_r)$ and $\rho(C_{r,i})$ for $i=1,\dots,M$.
    \State $\star$ Compute $\rho(C_{l,i}^S)$ and $\rho(C_{l,i}^\Omega)$ for $i=1,\dots,M_{bc}$.
    \end{algorithmic}
    \textbf{Online stage:}
    \begin{algorithmic}[1]
    \State Set $t=0$, $\mathbf{a}=\mathbf{a}_0$, $\Delta t=\Delta t_0$.
    \While{$t<T$}
    \State $\star$ Compute $\tilde\Re(\lambda_F)$ and $\tilde\Im(\lambda_F)$ as in Eq.\eqref{eq:ebound_F}.
    \State $\star$ Compute $z_{max}$ with the bisection method.
    \State $\star$ Compute $\Delta t$ as in Eq.\eqref{eq:dt}.
    \State Advance Eq.\eqref{eq:rom_final} in time with the chosen time integration scheme.
    \State $t = t + \Delta t$
    \EndWhile
  \end{algorithmic}
          \scriptsize\hspace{1em}$\star$ Steps required for a self-adaptive time integration with \texttt{RedEigCD}.
\end{algorithm}

Finally, the operational differences between the proposed framework and existing strategies is presented. Table \ref{tab:comparison} presents a side-by-side comparison of \texttt{EigenCD}, \texttt{AlgEigCD}, and \texttt{RedEigCD}, focusing on their algorithmic complexity and applicability to ROM simulations. 
\begin{table}[h]
\centering
\caption{Description and comparison of computational complexity (offline vs. online) for \texttt{EigenCD}, \texttt{AlgEigCD}, and the proposed \texttt{RedEigCD} method.}
\label{tab:comparison}
\begin{tabularx}{\textwidth}{l X c c} % 'l' for method, 'X' for description (auto-wrap), 'c' for math
\toprule
\textbf{Method} & \textbf{Description} & \textbf{Online complexity} & \textbf{Offline complexity} \\
\midrule
\texttt{EigenCD} \cite{trias_self-adaptive_2011} 
    & Applies Gershgorin to the explicit $D_r,C_r(\mathbf a)$ operators. Inaccurate for dense matrices in a reduced framework.
    & $\mathcal{O} (M^2)$ 
    & $0$ \\ \addlinespace
    
\texttt{AlgEigCD} \cite{trias_efficient_2024} 
    & Relies on the structure of $D_r, C_r(\mathbf a)$ to reduce the eigenvalue computation to a matrix-vector product of FOM size. 
    & $\mathcal{O}(N_F)$ 
    & $\mathcal{O}(N_F)$ \\ \addlinespace

\texttt{RedEigCD} (proposed)
    & Uses the tensorial decomposition to simplify the eigenvalue calculation for $C_r(\mathbf a)$. Uses the exact eigenvalues for $D_r$.
    & $\mathcal{O}(M)$ 
    & $\mathcal{O}(M^3)$ \\ 
\bottomrule
\end{tabularx}
\end{table}

\subsection{Analytical stability analysis of ROMs} \label{sec:biggerdt}

In practice, it has been observed that ROMs are stable for larger timesteps than the FOM. For instance, \citet{baiges_explicit_2013} indicated that for explicit time integration schemes, the timestep was not restricted by the CFL condition, as the Courant number could be larger than unity, opposite to the FOM.
Intuitively, cutting off some of the highest (spatial) frequencies,
which are represented by the smallest singular values, would imply removing
some of the most rapid (temporal) modes. These are the modes that make the system stiffer,
and thus are the most restrictive modes in terms of stability. By eliminating these modes,
the stability constraint is more relaxed, and the method yields, as a consequence, larger timesteps in
the reduced framework.

Bach et al.\ \cite{bach_stability_2018} analyzed the stability of ROMs compared to FOMs for linear
elasticity problems, which are linear problems with real symmetric
matrices (like diffusion), where the timestep at the ROM level is demonstrated to
be at least as large as the timestep at the FOM level, yet usually larger. To provide a mathematical foundation to this qualitative point of view, Bach et al.
\cite{bach_stability_2018} made use of the Poincaré separation theorem, which reads as follows:
\begin{lemma}
(Poincaré separation theorem \cite{stewart_matrix_2001, magnus_matrix_2019}) Let $V\in\mathbb{R}^{m\times k}$, where
$k<<m$, be a matrix with orthonormal column vectors, i.e.
  $V^TV=I_k$. Let
$A\in\mathbb{R}^{m\times m}$ be a real symmetric matrix. Then, the eigenvalues
of $V^TAV$ separate the eigenvalues of $A$ such that
\begin{equation}
  \lambda_i \leq \tilde{\lambda}_i \leq \lambda_{m-k+i},~~i\leq k
\end{equation}
\noindent where $\lambda_i$ denotes the eigenvalues of $A$, and $\tilde{\lambda}_i$ the
eigenvalues of $V^TAV$, ordered in increasing
  order, $\lambda_1\leq\lambda_2\leq\dots\leq\lambda_m$ and $\tilde{\lambda}_1\leq\tilde{\lambda}_2\leq\dots\leq\tilde{\lambda}_k$ respectively.
\label{th:pst}
\end{lemma}

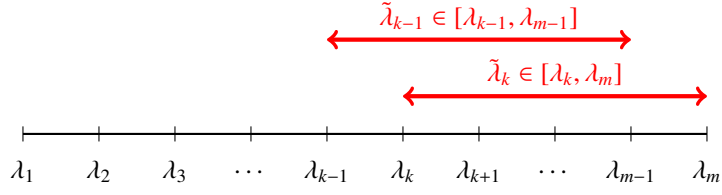
\begin{figure}[h]
  \centering
  \begin{tikzpicture}
    \draw[thick] (0,0) -- (9,0);
    \foreach \x in {0,1,2,3,4,5,6,7,8,9} {
      \draw (\x,0.1) -- (\x,-0.1);
    }
    \node (c) at (0,-0.5) {$\lambda_1$};
    \node (c) at (1,-0.5) {$\lambda_2$};
    \node (c) at (2,-0.5) {$\lambda_3$};
    \node (c) at (3,-0.5) {$\dots$};
    \node (c) at (4,-0.5) {$\lambda_{k-1}$};
    \node (c) at (5,-0.5) {$\lambda_k$};
    \node (c) at (6,-0.5) {$\lambda_{k+1}$};
    \node (c) at (7,-0.5) {$\dots$};
    \node (c) at (8,-0.5) {$\lambda_{m-1}$};
    \node (c) at (9,-0.5) {$\lambda_m$};

    \draw[<->, ultra thick, color=red] (5,0.5) -- (9,0.5) node [midway,above] {$\tilde{\lambda}_k\in[\lambda_k,\lambda_m]$};
    \draw[<->, ultra thick, color=red] (4,1.25) -- (8,1.25) node [midway,above] {$\tilde{\lambda}_{k-1}\in[\lambda_{k-1},\lambda_{m-1}]$};

  \end{tikzpicture}
  \caption{Graphical representation of the range of reduced eigenvalues for $i=k$ and $i=k-1$, in red, according
  to the Poincaré separation theorem.}
  \label{fig:pst}
\end{figure}

At the limit, $i=k$, the value $\tilde{\lambda}_k$ denotes the eigenbound 
of $V^TAV$. According to Lemma \ref{th:pst}, $\tilde{\lambda}_k$ lies in the range $[\lambda_k,\lambda_m]$. Therefore,
$\tilde{\lambda}_k\leq\lambda_m$, which determines the eigenbound of $A$. For the sake of
clarity, this has been represented graphically in Figure \ref{fig:pst}. This theorem
can be applied straightaway to the diffusive operator and shows that
$\rho(D_r) \leq \rho(D)$ as the diffusive operator corresponds to a real
symmetric matrix, and $\Phi$ is an $\Omega$-orthonormal matrix. Therefore, Lemma \ref{th:pst} holds,
as Bach et al. \cite{bach_stability_2018} showed that its result is not altered if the basis fulfills $V^T\Omega V = I_k$ instead of $V^TV=I_k$.

In contrast, Lemma \ref{th:pst} cannot be applied to the convective term as it is not symmetric. Therefore, the work of Bach et al. \cite{bach_stability_2018} is extended to the case of ROMs with non-symmetric operators. To do so, the work of Rao \cite{rao_separation_1979} is employed, which extends the Poincaré separation theorem to singular values instead of eigenvalues, and thus holds even for the case of non-symmetric matrices.

\begin{lemma}
(Poincaré separation theorem for singular values \cite{rao_separation_1979}) Let
$A\in\mathbb{R}^{m\times n}$. Let $B\in\mathbb{R}^{m\times k}$ and
$C\in\mathbb{R}^{n\times r}$ be orthonormal matrices, i.e. $B^TB=I_k$ and
$C^TC=I_r$. The $i$-th singular value of $A$ is denoted by $\sigma_i$ in
increasing order, as for the Poincaré separation theorem. Moreover,
$\tilde{\sigma}_i$ denotes the $i$-th singular value of $B^TAC$. Then, the following inequality holds:
\begin{equation}
  \sigma_i \leq \tilde{\sigma}_i \leq \sigma_{\tilde{m}-i+1}, ~\mathrm{for}~1\leq
  i\leq\mathrm{min}(r,k),
\end{equation}
  where $\tilde{m}=\mathrm{min}(m,n)$.
  \label{th:pst_sv}
\end{lemma}

Let, then, $A$ be a square matrix and $B=C$. In the current case, $A = C(\mathbf{u})$, and $B = \Phi$. Hence, this theorem can be applied to the convective term by
making use of the fact that the singular values of a real skew-symmetric matrix
correspond to the absolute values of the eigenvalues, which suffices to
estimate the eigenbound of the operator. 

\begin{figure}[h]
  \centering
  \begin{tikzpicture}
    \draw[thick] (0,0) -- (9,0);
    \foreach \x in {0,1,2,3,4,5,6,7,8,9} {
      \draw (\x,0.1) -- (\x,-0.1);
    }
    \node (c) at (0,-0.5) {$\sigma_1$};
    \node (c) at (1,-0.5) {$\sigma_2$};
    \node (c) at (2,-0.5) {$\dots$};
    \node (c) at (3,-0.5) {$\sigma_{k-1}$};
    \node (c) at (4,-0.5) {$\sigma_k$};
    \node (c) at (5,-0.5) {$\dots$};
    \node (c) at (6,-0.5) {$\sigma_{m-k}$};
    \node (c) at (7,-0.5) {$\sigma_{m-k+1}$};
    \node (c) at (8,-0.5) {$\dots$};
    \node (c) at (9,-0.5) {$\sigma_m$};

    \draw[<->, ultra thick, color=red] (4,0.5) -- (7,0.5) node [midway,above] {$\tilde{\sigma}_k\in[\sigma_k,\sigma_{m-k+1}]$};
    \draw[<->, ultra thick, color=red] (3,1.25) -- (6,1.25) node [midway,above] {$\tilde{\sigma}_{k-1}\in[\sigma_{k-1},\sigma_{m-k}]$};

  \end{tikzpicture}
  \caption{Graphical representation of the range of reduced singular values for $i=k$ and $i=k-1$, in red, according
  to the Poincaré separation theorem for singular values.}
  \label{fig:pst_sv}
\end{figure}
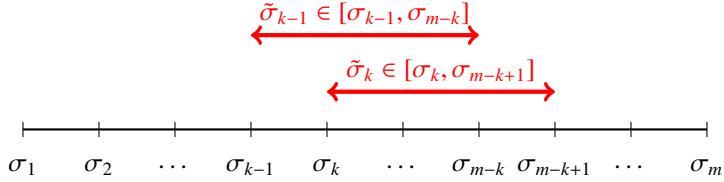

Hence, the inequality can be
rewritten for $r=k$, $m=n$ with $i=k$ to estimate the eigenbound
of the convective term, which reads as follows,
\begin{equation}
\sigma_k \leq \tilde \sigma_k \leq \sigma_{m-k+1} \leq \sigma_m,
\label{eq:inequality_sigma}
\end{equation}
\noindent as graphically represented in Figure \ref{fig:pst_sv}. Compared to Lemma \ref{th:pst}, Eq.\eqref{eq:inequality_sigma} does include the inequality $\sigma_{m-k+1} \leq \sigma_m$, as Lemma \ref{th:pst_sv} bounds the singular values of the reduced matrix with $\sigma_k$ and $\sigma_{m-k+1}$ as shown in Fig.\ \ref{fig:pst}, yet by the assumed order of the singular values $\sigma_{m-k+1}\leq\sigma_m$. As the convective operator is represented by a skew-symmetric matrix, it is a normal matrix, i.e. $C^*C=CC^*$, and thus the results in terms of singular values can be translated to eigenvalues according to Lemma \ref{th:evals}.  Thus, the eigenvalues of $C(\mathbf{u})$ are greater than
those of $C_r(\mathbf{a})$. Therefore, $\rho(C_r)\leq\rho(C)$.

\begin{lemma} \label{th:evals}
   Let $A\in\mathbb{C}^{n\times n}$ be a normal matrix (i.e. $A^*A=AA^*$). Let $\lambda_1,\dots,\lambda_n$ be the eigenvalues of $A$, listed with algebraic multiplicity. Let $\sigma_1,\dots,\sigma_n$ be the singular values of $A$ (the nonnegative square roots of the eigenvalues of $A^*A$), also listed with multiplicity. Then, the singular values are the absolute values of the eigenvalues:
    \begin{equation}
        \{\sigma_1,\dots,\sigma_n\} = \{|\lambda_1|,\dots,|\lambda_n |\}.
        \label{eq:equiv_ev}
    \end{equation}
\end{lemma}
\begin{proof}
    As $A$ is a normal matrix, it can be diagonalized with a unitary matrix $U$, which is a set of orthonormal eigenvectors such that $U^*U=I$, and a diagonal matrix $\Lambda$, which has the eigenvalues of $A$ as entries, i.e. $\Lambda = \mathrm{diag}(\lambda_1,\dots,\lambda_n)$ \cite[pp.\ 101--103]{johnson_unitary_1985}. Therefore,
    \begin{equation}
        A = U\Lambda U^*.
    \end{equation}
    \noindent Then, $A^*A$ can be written as
    \begin{equation}
        A^*A =(U\Lambda U^*)^*(U\Lambda U^*) = U\Lambda^*U^*U\Lambda U^* = U\Lambda^*\Lambda U^*,
    \end{equation}
    \noindent where $\Lambda^*\Lambda = \mathrm{diag}(\lambda_1^*\lambda_1,\dots,\lambda_n^*\lambda_n) = \mathrm{diag}(|\lambda_1|^2,\dots,|\lambda_n|^2)$. Therefore,
    \begin{equation}
        A^*A = U\mathrm{diag}(|\lambda_1|^2,\dots,|\lambda_n|^2)U^*.
    \end{equation}
    \noindent Then, the $i-$th column of $U$, $\mathbf{u}_i$, is going to be an eigenvector of $A^*A$ with eigenvalue $|\lambda_i|^2$, since $A^*A\mathbf{u}_i=\mathbf{u}_i|\lambda_i|^2$.    
    By definition, the singular values of $A$ are the eigenvalues of $A^*A$. Therefore, it shows that Eq.\eqref{eq:equiv_ev} holds and, with a proper reordering,
    \begin{equation}
        \sigma_i = |\lambda_i|,~\forall i.
    \end{equation}
\end{proof}

As Eq.\eqref{eq:ebound_F} indicates, the magnitude of the relevant eigenvalue
in the ROM is determined by the contributions of both diffusive and
convective term. 
Thus, the following theorem can be stated.
\begin{theorem}
  Let $\Delta t^{ROM}$ be the largest stable timestep for the ROM given a state $\mathbf{a}$, and $\Delta t^{FOM}$ the largest stable timestep for a state $\mathbf{u}$ that is in the span of $\Phi$, e.g.\ $\mathbf{u}=\Phi\mathbf{a}$, both using the same explicit time integration scheme. Then, $\Delta t^\mathrm{ROM} \geq \Delta t^\mathrm{FOM}$.
  \label{th:dt}
\end{theorem}
\begin{proof}
  Let $\lambda_{F,\mathrm{FOM}} = \rho(D) + i\rho(C(\mathbf{u}))$ be the eigenvalue of the FOM, and $\lambda_{F,\mathrm{ROM}}$ be calculated as in Eq.\eqref{eq:ebound_F}. As shown above, $\rho(D_r)\leq\rho(D)$ (Lemma \ref{th:pst}). Moreover, given $\mathbf{u}=\Phi \mathbf a$ is the FOM equivalent of $\mathbf{a}$, $\rho(C_r)\leq\rho(C)$ (Lemma \ref{th:pst_sv}). Therefore, $||\lambda_{F,\mathrm{FOM}}|| \geq ||\lambda_{F,\mathrm{ROM}}||$, where $||\cdot||$ denotes the complex modulus. Let $\Delta t$ be calculated as in Eq.\eqref{eq:dt}. As the stability region of the explicit time integration scheme is fixed, and $\Delta t$ is inversely proportional to $||\lambda_F||_2$, then $\Delta t^\mathrm{ROM} \geq \Delta t^\mathrm{FOM}$.
\end{proof}

\begin{remark}[Applicability to time-integration]
    Strictly speaking, Theorem \ref{th:dt} guarantees $\Delta t^\mathrm{ROM} \geq \Delta t^\mathrm{FOM}$ only when the convective field is identical for both FOM and ROM (i.e. $\mathbf{u}=\Phi\mathbf{a}$). In a practical time-integration scenario, the ROM trajectory $\mathbf{u}_r=\Phi\mathbf a(t)$ will deviate from the FOM solution $\mathbf u(t)$. Consequently, a rigorous inequality cannot be guaranteed for $t>0$.

    However, Theorem \ref{th:dt} serves as a strong theoretical proxy for the stability behavior of the ROM. Since the projection $\Phi$ inherently removes the high-frequency modes associated with the largest eigenvalues of $C(\mathbf u)$, the inequality $\rho(C_r)<\rho(C)$ is expected to hold robustly in practice. The theorem thus provides the structural justification for the increased stability limits observed in numerical experiments, even in the presence of trajectory divergence.
\end{remark}

% Comment by Henrik:
% As I see it, we have to compare $C(V)$ and $C_r(a) = \Phi^TC(\Phi a)\Phi$. We can use Lemma 2 with $A=C(\Phi a)$ and $B=C=\Phi$ to find that $\rho(C_r(a))=\rho(\Phi^T C(\Phi a)\Phi)\leq \rho(C(\Phi a))$. Consequently, rigorous statements can only be made for FOM velocities $V$ in the span of $\Phi$. For general velocities $V$ that cannot be represented as $\Phi a$, we cannot make any rigorous statement based on Lemma 2.

% Hence, the time step ratios in Fig. \ref{fig:bestApprox} do not directly follow from Lemma 2 - also because Case A uses AlgEigCD to approximate the eigenbounds (and the eigenbounds for Case D are also approximations?). Nevertheless, this figure does suggest that the timestep ratios below 1 in Fig. \ref{fig:sl_redeigcd} are due to the differences in the velocities at which the linearized convection operators are evaluated.

% <<< End input from redeig.tex
% >>> Begin input from results.tex
\section{Numerical experiments} \label{sec:results}

In this section, the new adaptive timestepping method outlined in Algorithm \ref{alg:redeigcd} is tested for two different fluid flow test cases. First of all, the robustness of the method in a case with periodic boundary conditions, namely the shear-layer roll-up, is presented. Moreover, the evolution of the maximum stable $\Delta t$ for different numbers of modes is presented, and compared with the maximum stable $\Delta t$ using the \texttt{AlgEigCD}. In the second case, the influence of non-homogeneous boundary conditions on the preservation of stability is considered for the unsteady wake behind an actuator disk \cite{sanderse_non-linearly_2020}. In both test cases, the classical explicit fourth-order RK scheme is used in a two-dimensional incompressible Navier-Stokes solver \cite{sanderse_energy-conserving_2013,sanderse_ins2d_2018}. Both methods have been run for $M=[16,32,64,128,200]$ modes. Table \ref{tab:methods} summarizes the methods used in this work.

\begin{table}[h]
\centering
\caption{Summary of the methods considered in this work.}
\label{tab:methods}
\begin{tabularx}{\textwidth}{clX}
\toprule
  \textbf{Case} & \textbf{Method} & \textbf{Description / Purpose} \\
\toprule
  \texttt{A} & FOM+\texttt{AlgEigCD} & Full-order model with staggered \texttt{AlgEigCD} (\ref{sec:stg_aecd}); used as reference. \\ \addlinespace
  \texttt{B} & ROM+\texttt{RedEigCD} & ROM with proposed self-adaptive timestep method. \\ \addlinespace
  \texttt{C} & ROM+constant & ROM with constant $\Delta t$: $0.1$ for the shear-layer roll-up, $4\pi/200$ for the actuator disk {\cite{sanderse_non-linearly_2020}}; used to compare accuracy of ROM+\texttt{RedEigCD}. \\ \addlinespace
  \texttt{D} & Best ROM & Best approximation ROM (Eq.~\eqref{eq:best_approximation}) with constant $\Delta t$; used to verify $\Delta t_{\text{ROM}} > \Delta t_{\text{FOM}}$ throughout simulation. \\
\bottomrule
\end{tabularx}
\end{table}

\subsection{Shear-layer roll-up}

In this case, the roll-up of a shear layer has been simulated using a square domain of side length $2\pi$ with full-periodic boundary conditions. The initial conditions for the case are the following:

\begin{equation}
u_0(x,y) = 1 + \left\{
\begin{array}{cc}
\tanh\left(\frac{y-\pi/2}{\delta}\right), & y \leq \pi, \\
\tanh\left(\frac{3\pi/2-y}{\delta}\right), & y > \pi, 
\end{array}
\right.
\hspace{1in}
v_0(x,y) = \epsilon\sin(x),
\end{equation}

\noindent where $\delta = \pi/15$ and $\epsilon=1/20$. The Reynolds number is 1000. The FOM discretization
used for this case employs $100\times100$ finite volumes, with variable timestep computed
using \texttt{AlgEigCD} (\ref{sec:stg_aecd}). Snapshots have been
extracted for $t=0$ to $t=20$ from \texttt{Case A} (Table \ref{tab:methods}). Figure \ref{fig:svd_sl} shows the decay of the singular values for this case.

\begin{figure}[h]
\centering
\includegraphics[height=0.3\textwidth]{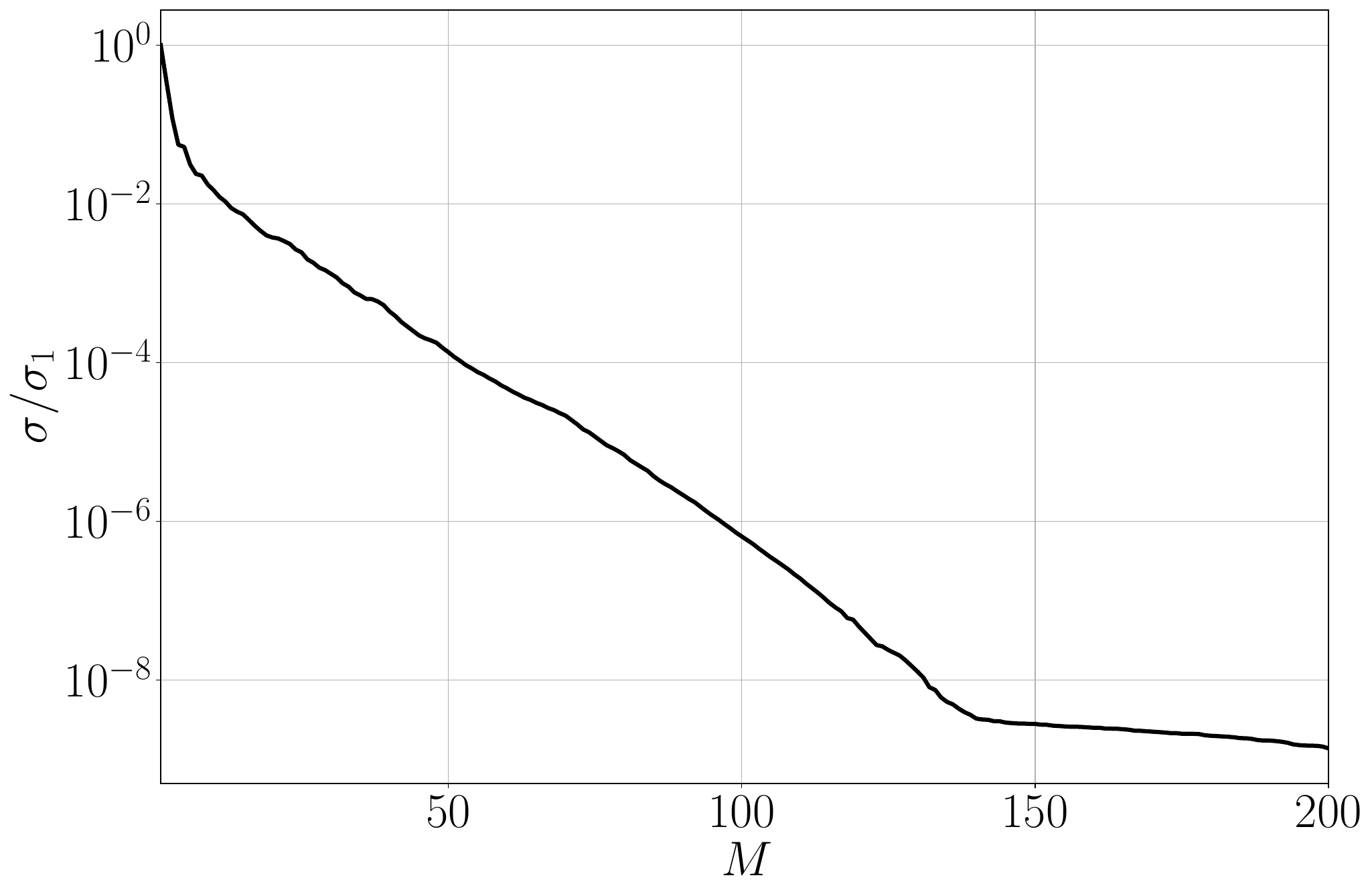}
\caption{Singular values for the Re=1000 shear-layer roll-up up to $M=200$.}
\label{fig:svd_sl}
\end{figure}

Moreover, note that even for the largest number of modes considered,
$M=200$, where the singular values have decayed up to $\mathcal{O}(10^{-9})$, the number of degrees of freedom is notably reduced compared to those of the FOM, which is $3\times10^4$, considering both $x-$ and $y-$
velocity components as well as the pressure. The simulation has been
run for \texttt{Case B} (ROM+\texttt{RedEigCD}) with the number of modes previously stated and the evolution of
$\Delta t$ is compared with \texttt{Case A}.

\begin{figure}[h]
\centering
\includegraphics[height=0.3\textwidth]{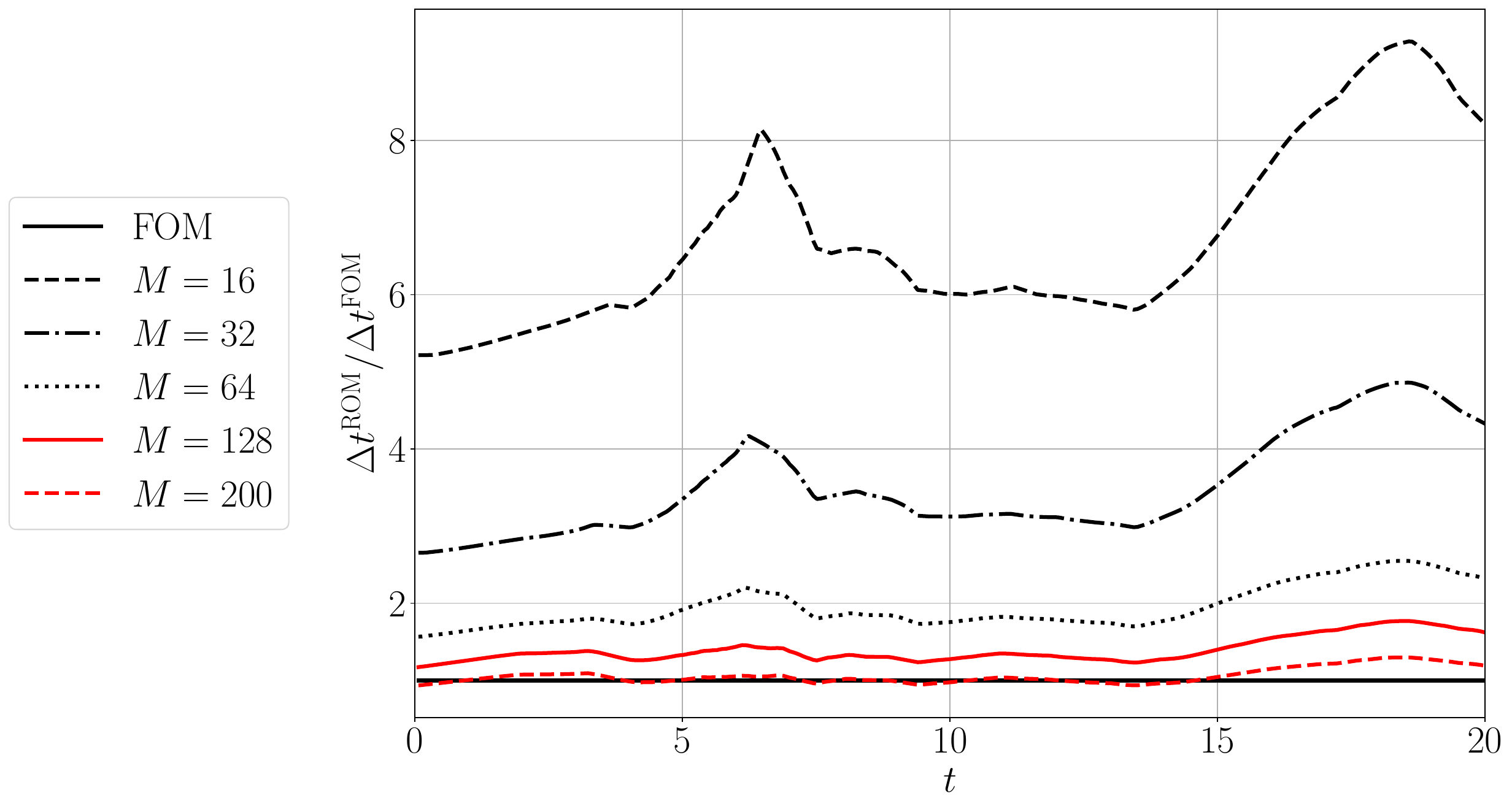}
\includegraphics[height=0.3\textwidth]{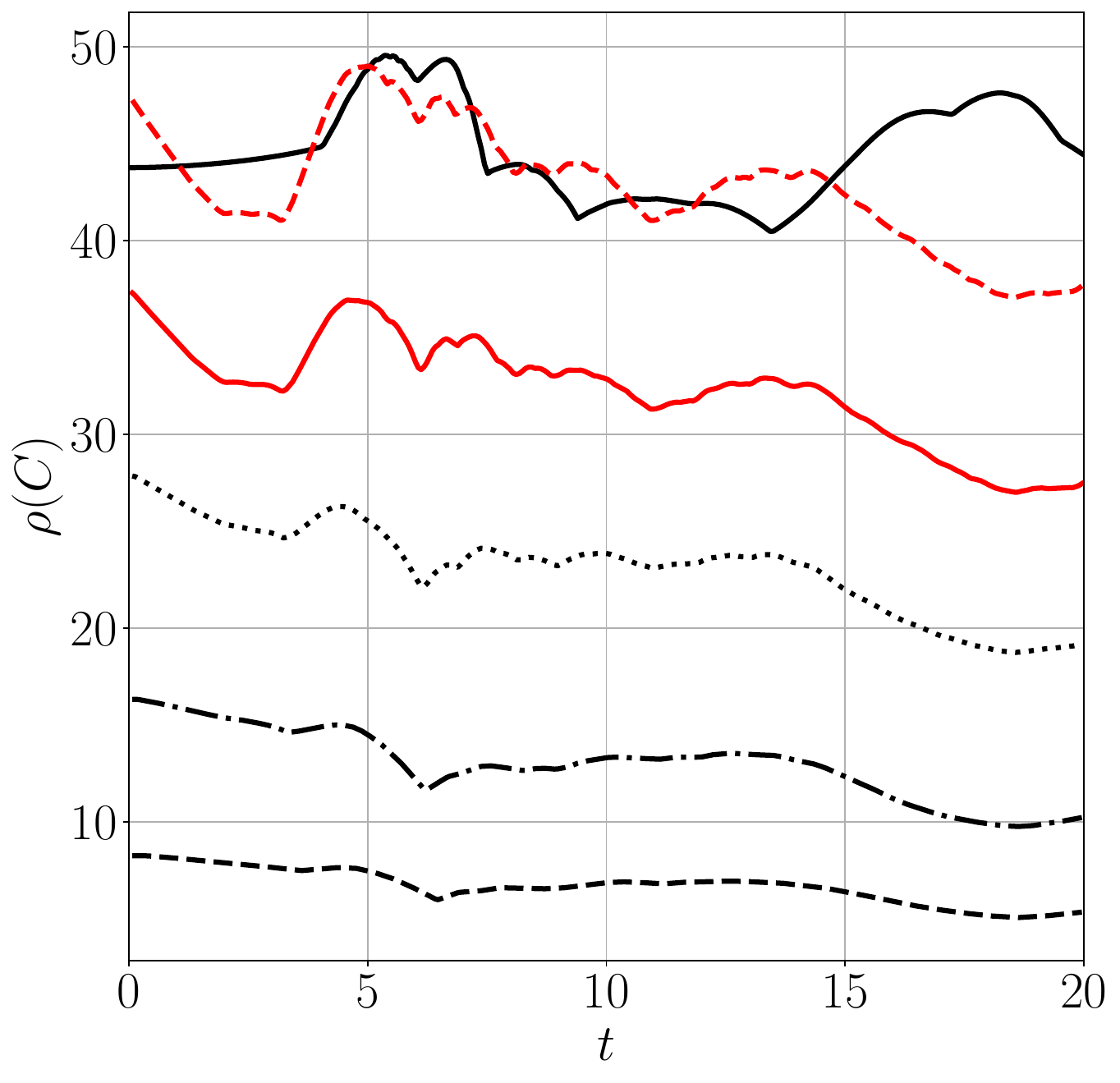}
  \caption{Timestep ratio between \texttt{Case B} and \texttt{Case A} up to $t=20$ (left)
  and evolution of the imaginary eigenbound (right) for the roll-up of a shear-layer at $\text{Re}=1000$.}
  \label{fig:sl_redeigcd}
\end{figure}

Figure \ref{fig:sl_redeigcd} (left) shows that the new self-adaptive timestepping method is effective in
providing a larger timestep for the ROM than for the FOM, as expected from Theorem \ref{th:dt}, as even for $M=200$, the largest number of modes considered, the ratio is around 1, which is indicated in the plot with a horizontal thicker line. Depending on the current
flow physics and number of modes of the ROM, this timestep can be up to 9 times larger than for the FOM, which significantly
accelerates the simulation by notably reducing the number of timesteps required to reach the final time. Furthermore, Figure \ref{fig:sl_redeigcd} (left) shows how the ratio between the timesteps obtained in  \texttt{Case B} and \texttt{Case A} converges to 1 for increasing ROM dimension, as expected from Theorem \ref{th:dt}, as the eigenvalues of the reduced operators interlace with those of the full-order operators. Physically, this trend is driven by the inclusion of higher-index POD modes. As these capture the steeper spatial gradients of the flow, the ROM's spectral radius monotonically approaches the FOM limit, as established in Section \ref{sec:biggerdt}. This is especially clear for $M=200$, where the singular values have decayed to $\mathcal{O}(10^{-9})$, and the truncation errors between the ROM and the FOM (Figure \ref{fig:svd_sl})
get closer to zero. Furthermore, note that at
certain points, especially at the beginning of the simulation, the timestep ratio
for $M=200$ is slightly smaller than 1.

With regards to the time-evolution of the imaginary eigenbound (Figure
\ref{fig:sl_redeigcd}, right), the expected behaviour as theoretically derived in Section \ref{sec:biggerdt} holds: the 
eigenbound of the reduced convective terms steadily increases with an increase in the number
of modes. For $M=200$, however, in some points the imaginary eigenbound for the ROM is slightly larger than for the FOM. (Note that representing the time
evolution for the diffusive operator has no relevance in this discussion, as it
remains constant throughout the whole simulation.)

The reason for the mismatch with theory reported for $M=200$ in Figure \ref{fig:sl_redeigcd} is as follows. Section \ref{sec:biggerdt} considers that the given matrix at FOM level and its reduced counterpart are evaluated at the equivalent state vector $\mathbf{u}$. However, in a ROM simulation at a given $t$, the reduced velocity $\mathbf{u}_r(t)=\Phi\mathbf{a}(t)$ is not identical to its FOM counterpart, $\mathbf{u}(t)$, and therefore both velocity fields lead to non-identical convective operators. Given these differences, Lemma \ref{th:pst_sv} acts as an approximation and thus it may happen that $\rho(C(\mathbf{u}(t)))\leq\rho(C(\mathbf{u}_r(t)))$ at a certain $t$ when $\mathbf u(t)\neq \mathbf u_r(t)$. 

Therefore, to test if the $\Delta t$ at ROM level using the equivalent state vector is always larger than
the one for the FOM actually holds, the best approximation ROM \cite{sanderse_non-linearly_2020} should be used, which is defined as

\begin{equation}
    \mathbf{a}_\mathrm{best}(t) = \Phi^T\Omega\mathbf{u}(t),
    \label{eq:best_approximation}
\end{equation}

\noindent which should fulfill Theorem \ref{th:dt} throughout the whole simulation as both $C(\mathbf{u})$ and $C_r(\mathbf{a}_\mathrm{best})$ are evaluated at the equivalent state vector. Hence, the reduced convective and
diffusive operators for every snapshot have been constructed, and the eigenbounds using the \texttt{RedEigCD} method have been computed. As shown in Figure \ref{fig:bestApprox}, the result from Theorem \ref{th:dt} is indeed preserved 
throughout the whole simulation, with the $\Delta t$ computed from the best approximation ROM for the largest number
of modes, this being the most restrictive case. 

\begin{figure}[h]
  \centering
  \includegraphics[height=0.3\textwidth]{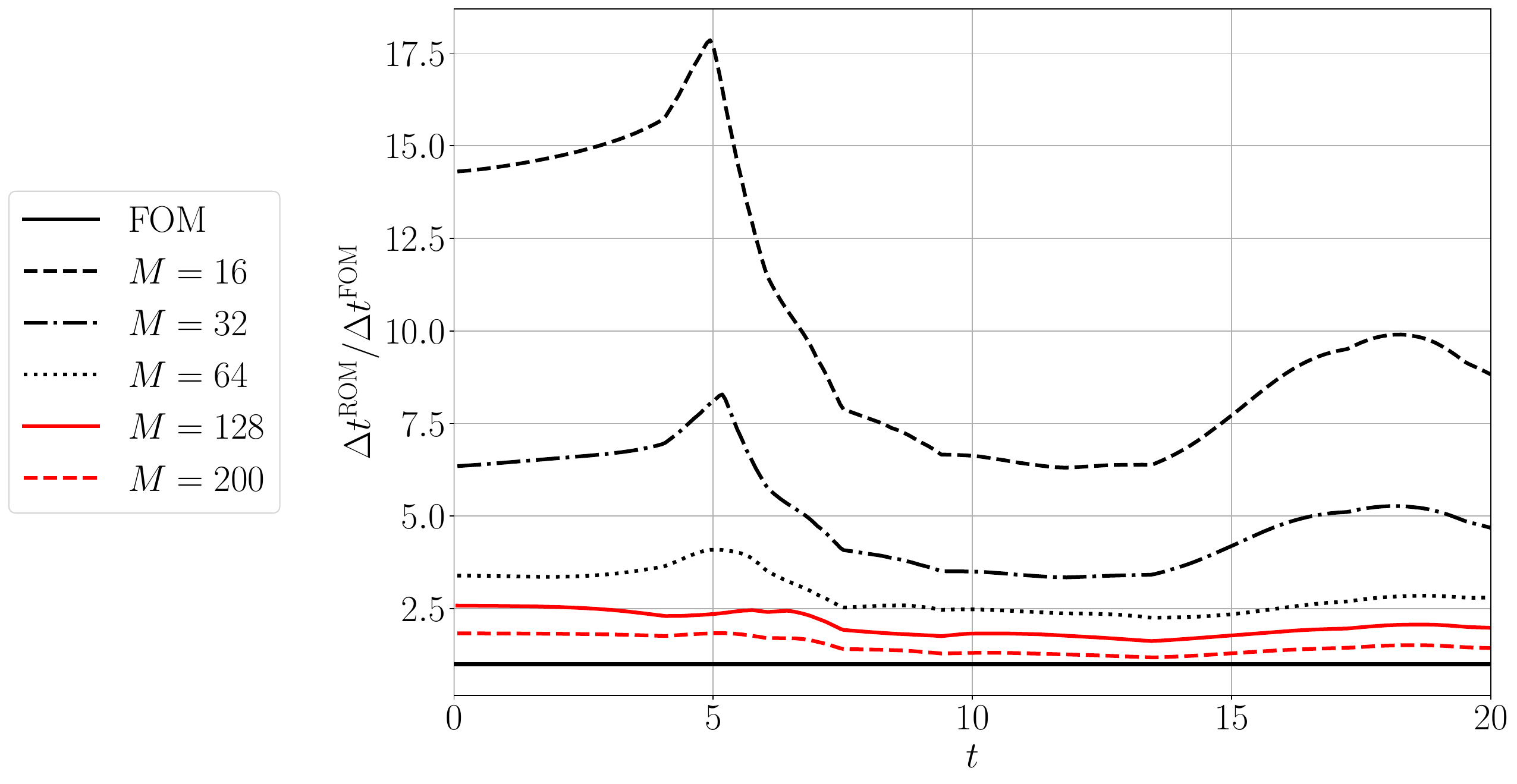}
  \caption{Timestep ratio between \texttt{Case D} and \texttt{Case A} for up to $t=20$ for the
  roll-up of a shear-layer of $\text{Re}=1000$.}
  \label{fig:bestApprox}
\end{figure}

Increasing the number of modes, the timestep for the ROM is becoming closer to that of the FOM, as observed in Figure \ref{fig:sl_redeigcd} (left) and \ref{fig:bestApprox}. By doing so, steeper gradients, which are the ones driving the magnitude of eigenbounds, are captured by the operators. This indeed pushes the magnitude of the reduced eigenbounds towards the magnitude of the FOM eigenbounds, thus justifying this observed trend.

Figure \ref{fig:error_comparison_sl} shows that the errors for both \texttt{Case B} and \texttt{Case C} have the same magnitude regardless of the number of modes ($M=16$, left; $M=200$, right). Thus, the use of the self-adaptive timestepping methods as performed in these cases
does not significantly affect the accuracy of the ROM compared to the error of the best approximation ROM,
as shown in Figure \ref{fig:error_comparison_sl}. In this case, $\Delta t$ for \texttt{Case C} is set to
0.01, as in the test case from \cite{sanderse_non-linearly_2020}. 
To isolate the temporal accuracy of the stability-controlled ROM from the spatial truncation error, we define the reduced error relative to the projected FOM snapshots, $\mathbf a_\mathrm{best}(t)$ which is defined as
\begin{equation}
    \varepsilon(\mathbf a_\mathrm{test}(t)) = \frac{||\mathbf a_\mathrm{test}(t) - \mathbf a_\mathrm{best}(t)||_2}{||\mathbf a_\mathrm{best}(t)||_2},
    \label{eq:error_rom}
\end{equation}
where $\mathbf a_\mathrm{test}$ indicates the reduced velocity field used to compute the error, i.e., $\mathbf{a}_\mathrm{\texttt{B}}$, $\mathbf{a}_\mathrm{\texttt{C}}$.
Figure \ref{fig:error_comparison_sl} shows that either for $M=16$ or for $M=200$ there is no loss of accuracy for the new self-adaptive timestepper (\texttt{Case B}) compared to \texttt{Case C}, which indicates that in this test case there is no loss of accuracy regardless of the number of modes. However, in \texttt{Case B} the use of \texttt{RedEigCD} provides potential gains of up to a factor of 9 compared to a constant $\Delta t$.
\begin{figure}[h]
  \centering
  \includegraphics[height=0.28\textwidth]{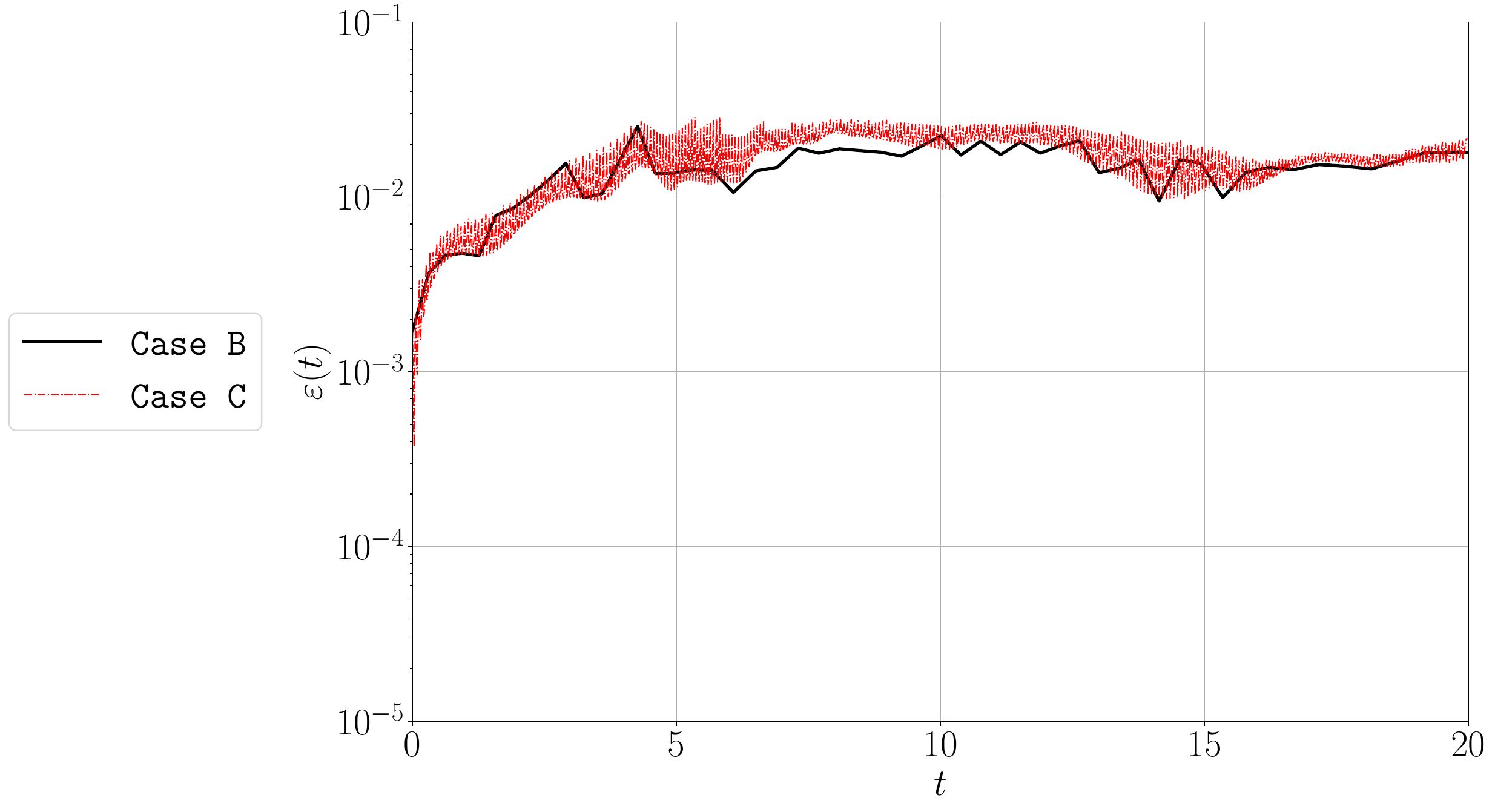}
  \includegraphics[height=0.28\textwidth]{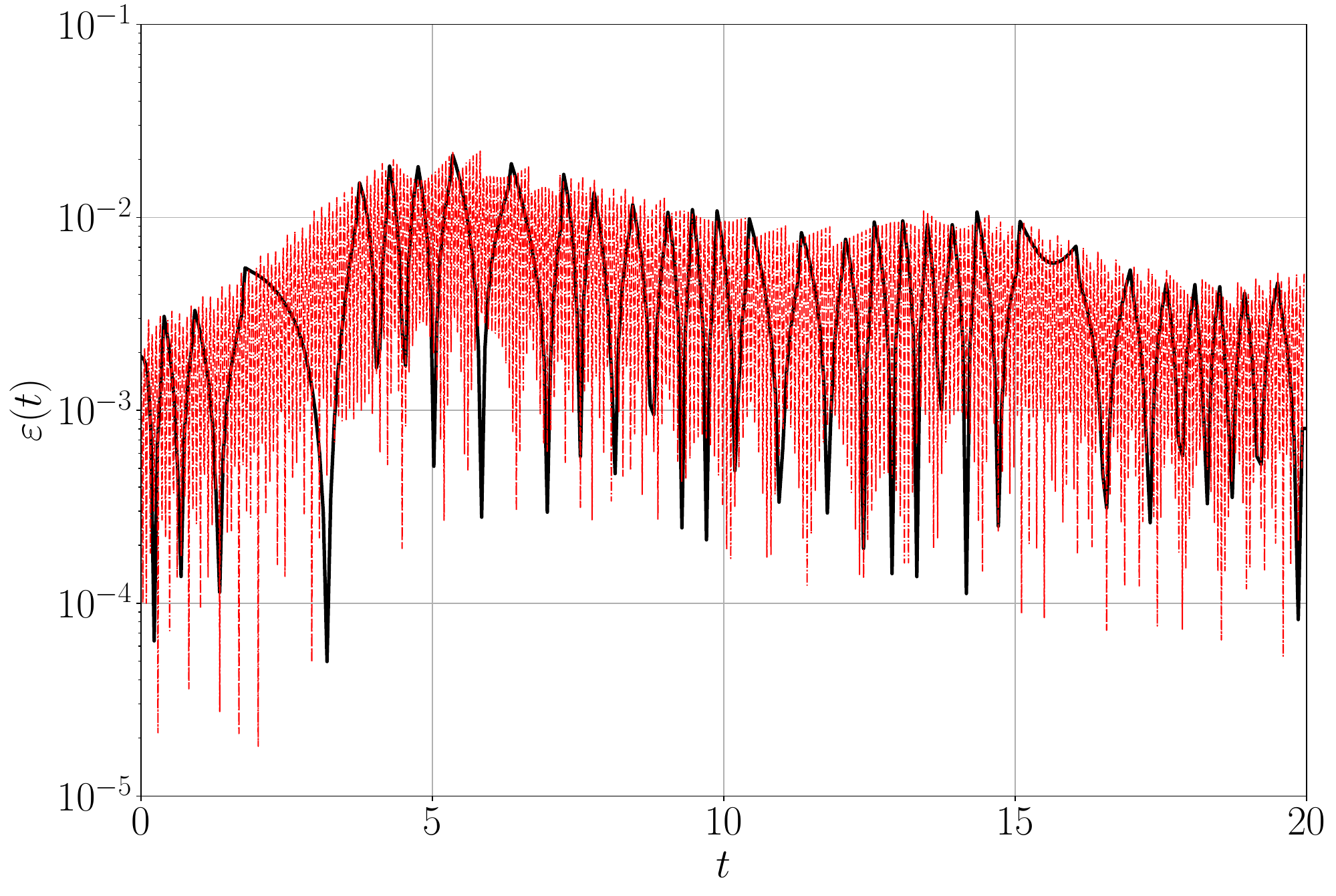}
  \caption{Error comparison of the ROM results for $M=16$ (left) and $M=200$ (right) in the shear-layer roll-up of $\text{Re}=1000$ for
  \texttt{Case B} and \texttt{C}.}
  \label{fig:error_comparison_sl}
\end{figure}

In terms of efficiency, the ROM with $M=64$ reduced the number of unknowns from $3\times10^4$ to 64, while simultaneously allowing timesteps up to almost 5 times larger than the FOM. This reduced the total number of timesteps, while the wall-clock time for each timestep was also reduced due to dimensionality reduction, thus having benefits from both factors. 

%Moreover, the accuracy of estimating the eigenvalues for the convective term using \texttt{RedEigCD} has been tested against the actual eigenvalues. To do so, the reduced convective operator for the FOM at $t=20$ has been constructed and its eigenbound has been precisely calculated and compared to the one estimated using \texttt{RedEigCD}, as reported in Section \ref{sec:comparison_eval}. 

\subsection{Flow through an actuator disk}

In this case, the air flow around a wind turbine is modelled as in 
\cite{rosenberger_no_2023}. The simulation geometrical domain is $[0,10]\times[-2,2]$, with a
time-dependent prescribed inlet velocity at the left boundary, defined as follows,

\begin{equation}
  u(0,y,t) = \cos(\alpha(t)), \hspace{0.5in} v(0,y,t) = \sin(\alpha(t)), \hspace{0.5in} \alpha(t) = \frac{\pi}{6}\sin\left(\frac{t}{2}\right).
\end{equation}

The other boundaries correspond to an outflow condition \cite[Section 6.4.2]{sanderse_energy-conserving_2013}.
The wind turbine has been modelled using an actuator disk of diameter 1, placed at $(2,0)$, which corresponds to
a momentum sink modelled as a constant force $f=0.25$ acting in the negative $x$-direction. The initial conditions for this
case are

\begin{equation}
  u_0(x,y) = 1, \hspace{0.5in} v_0(x,y) = 0.
\end{equation}

The Reynolds number is 100. The FOM discretization used for this case is $200\times80$ finite volumes,
with variable timestep computed using \texttt{AlgEigCD} (\texttt{Case A}). The numerical tests performed correspond to those of Table \ref{tab:methods}. Snapshots have been extracted for $t=0$ to $t=8\pi$.
Similarly to the roll-up of a shear-layer, the decay of the singular values is presented in Figure \ref{fig:svd_actuator}.
Note that the decay follows the same trend as in Figure \ref{fig:svd_sl}.

\begin{figure}[h]
  \centering
  \includegraphics[height=0.3\textwidth]{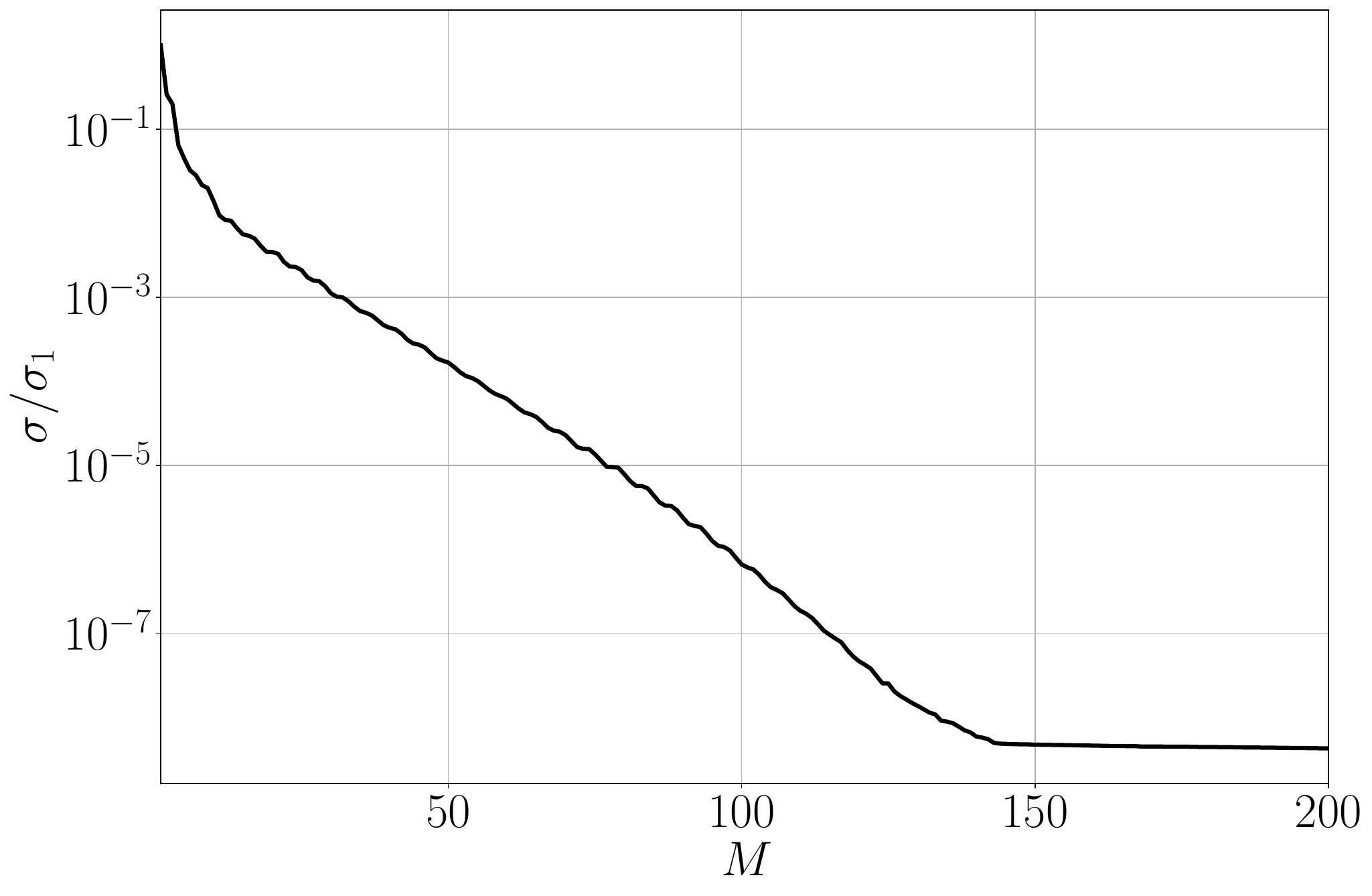}
  \caption{Singular values for the numerical simulation of an actuator disk of $\text{Re}=100$ up to 200 modes.}
  \label{fig:svd_actuator}
\end{figure}

The new stability-aware self-adaptive timestepping method is effective in significantly increasing the timestep
while maintaining stable simulations for non-homogeneous boundary conditions as shown in Figure \ref{fig:eb_actuator} (left). The obtained timestep for this case, 
depending on the flow conditions and the number of modes, can be up to 40 times larger, thus leading to significantly 
faster simulation times. Moreover, as expected from Section \ref{sec:biggerdt}, the evolution of the timestep in the 
case of the actuator disk shows that for $M=200$, for which the singular values do not decay anymore, the obtained 
timestep is marginally larger than the obtained for \texttt{Case A}.
As can be seen in Figure \ref{fig:eb_actuator} (right), the imaginary eigenbound is larger for larger ROM dimensions, analogously to the shear-layer roll-up testcase in Figure \ref{fig:sl_redeigcd} (right). Note
that in this case, the contribution of the boundary conditions appear in both the real and the imaginary term.

\begin{figure}[h]
\centering
\includegraphics[height=0.3\textwidth]{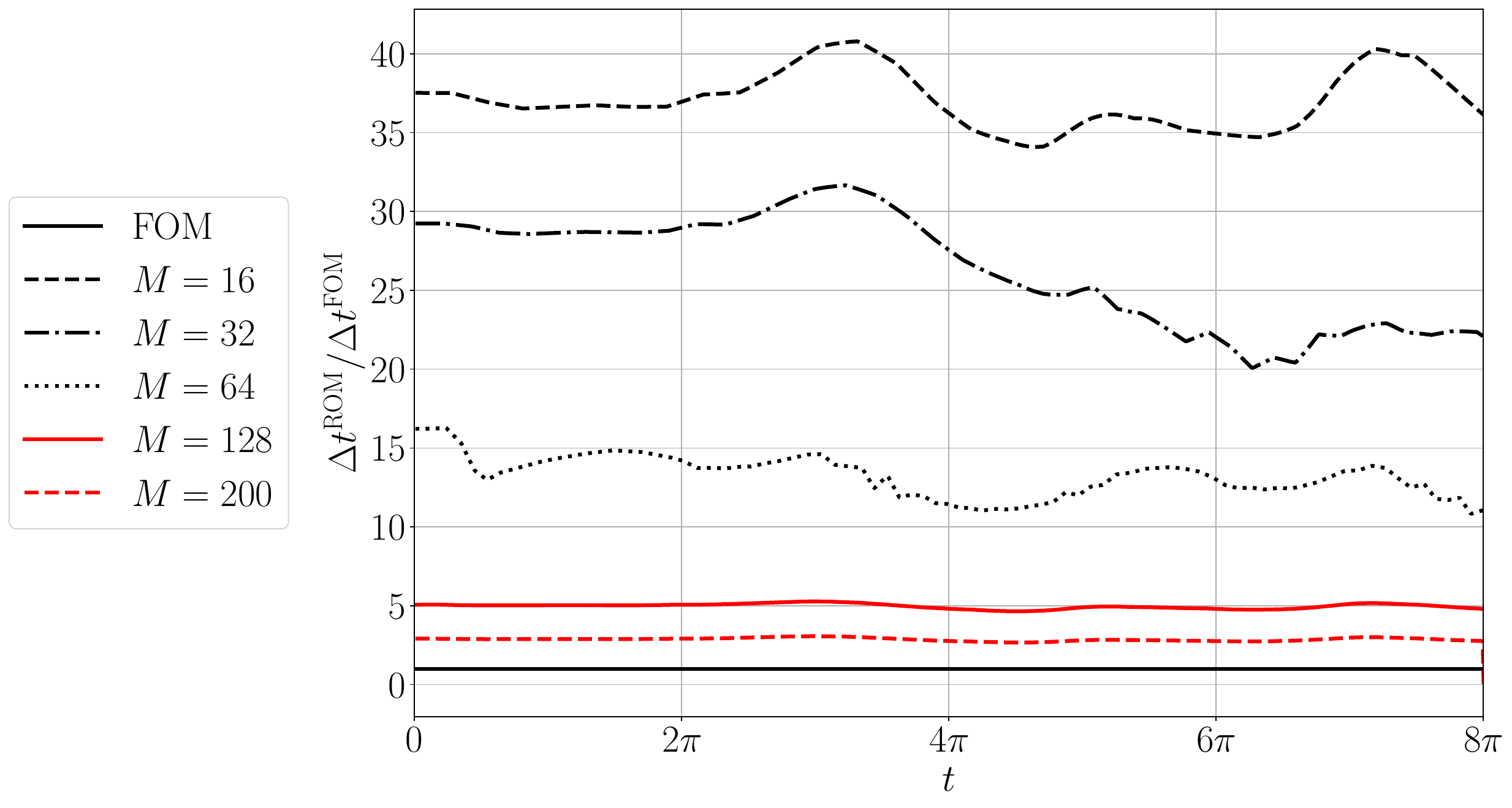}
\includegraphics[height=0.3\textwidth]{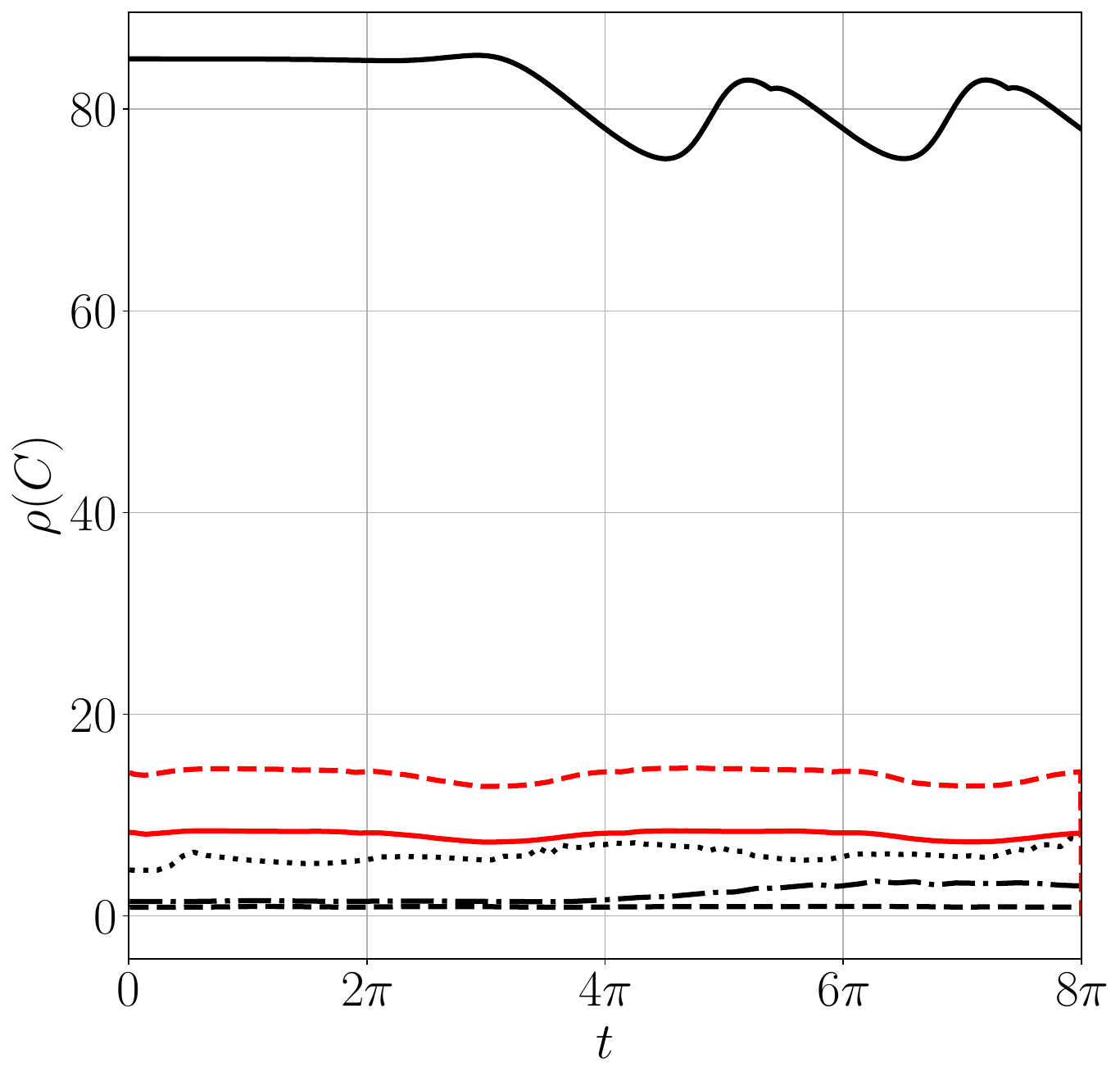}
  \caption{Timestep ratio between \texttt{Case B} and {Case A} up to $t=100$ (left)
  and evolution of the imaginary eigenbound (right) in the numerical simulation of a wind field of $\text{Re}=100$.}
  \label{fig:eb_actuator}
\end{figure}

The magnitude of the ROM error, calculated with Eq.\eqref{eq:error_rom}, is not affected by the variable $\Delta t$ compared to the reference result from \cite{sanderse_non-linearly_2020}, where $\Delta t$ was set to $4\pi/200$. As shown in Figure \ref{fig:error_comparison_actuator}, the accuracy of the ROM+\texttt{RedEigCD} (\texttt{Case B}) is comparable to that of the ROM with constant $\Delta t$ (\texttt{Case C}), for both $M=16$ and $M=200$. 
\begin{figure}[h]
  \centering
  \includegraphics[height=0.28\textwidth]{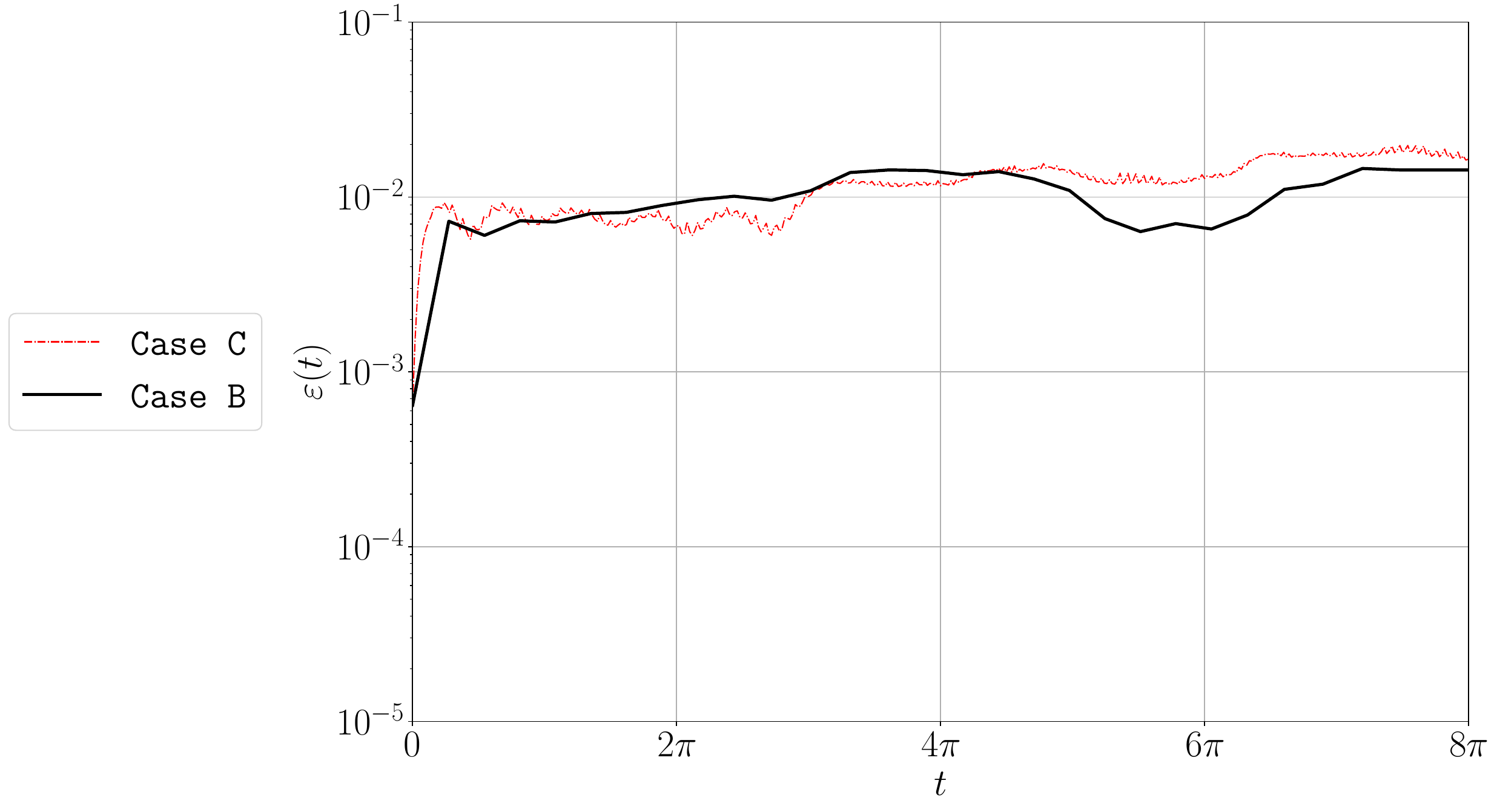}
  \includegraphics[height=0.28\textwidth]{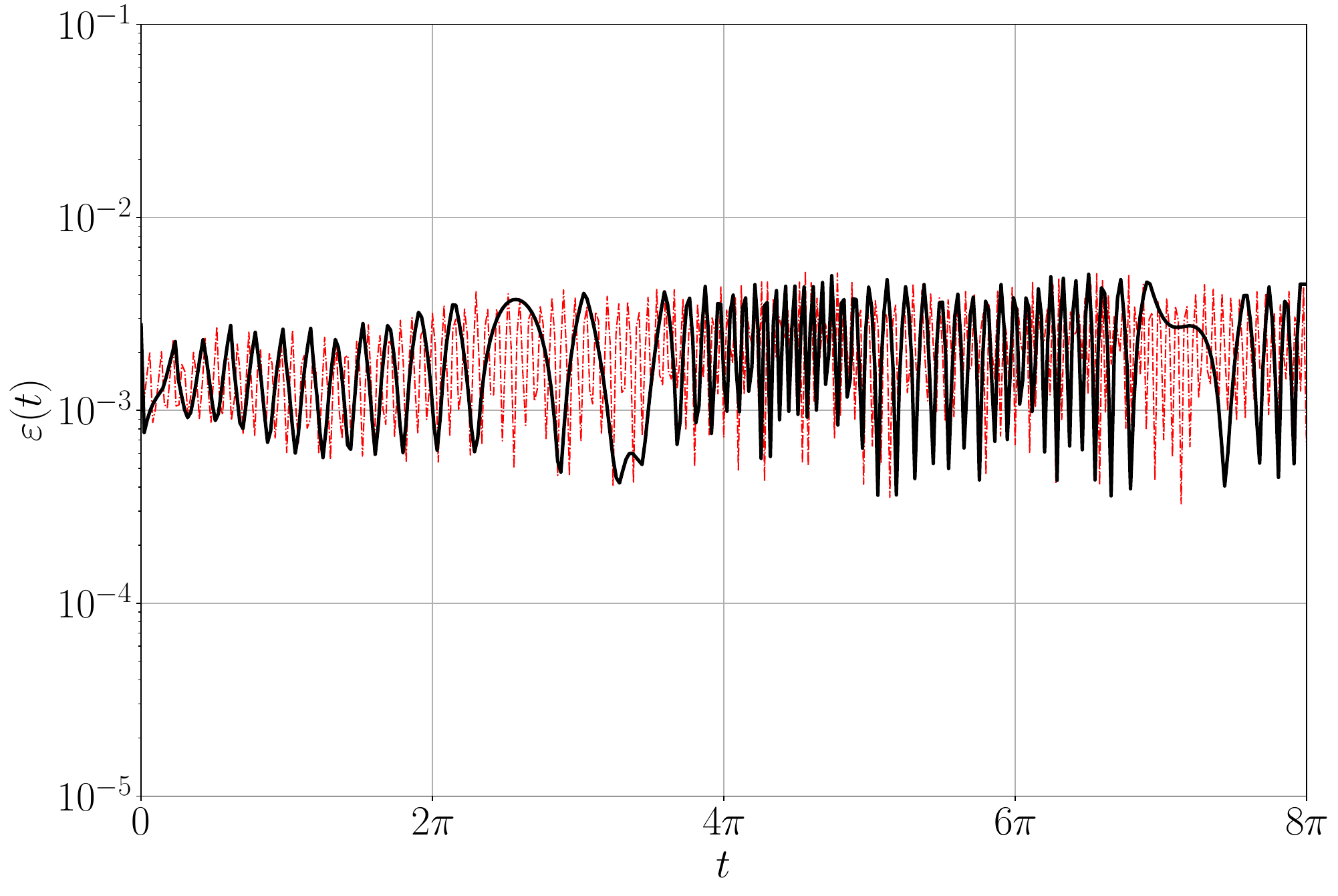}
  \caption{Error comparison of the ROM results for $M=16$ (left) and $M=200$ (right) in an actuator disk simulation with $\text{Re}=100$ for \texttt{Case B} and \texttt{Case C}.}
  \label{fig:error_comparison_actuator}
\end{figure}

\subsection{Accuracy of the eigenbound estimates} \label{sec:comparison_eval}

The accuracy of the eigenvalue estimation using \texttt{RedEigCD} has been tested for both shear-layer roll-up as well as the actuator case. The results also include the eigenbounds obtained by applying Gershgorin's circle theorem directly to $a_iC_{r,i}$, which would replicate \texttt{EigenCD} for ROMs.
Figure \ref{fig:eval_accuracy} presents the results.

\begin{figure}[h]
\centering
  \includegraphics[height=0.28\textwidth]{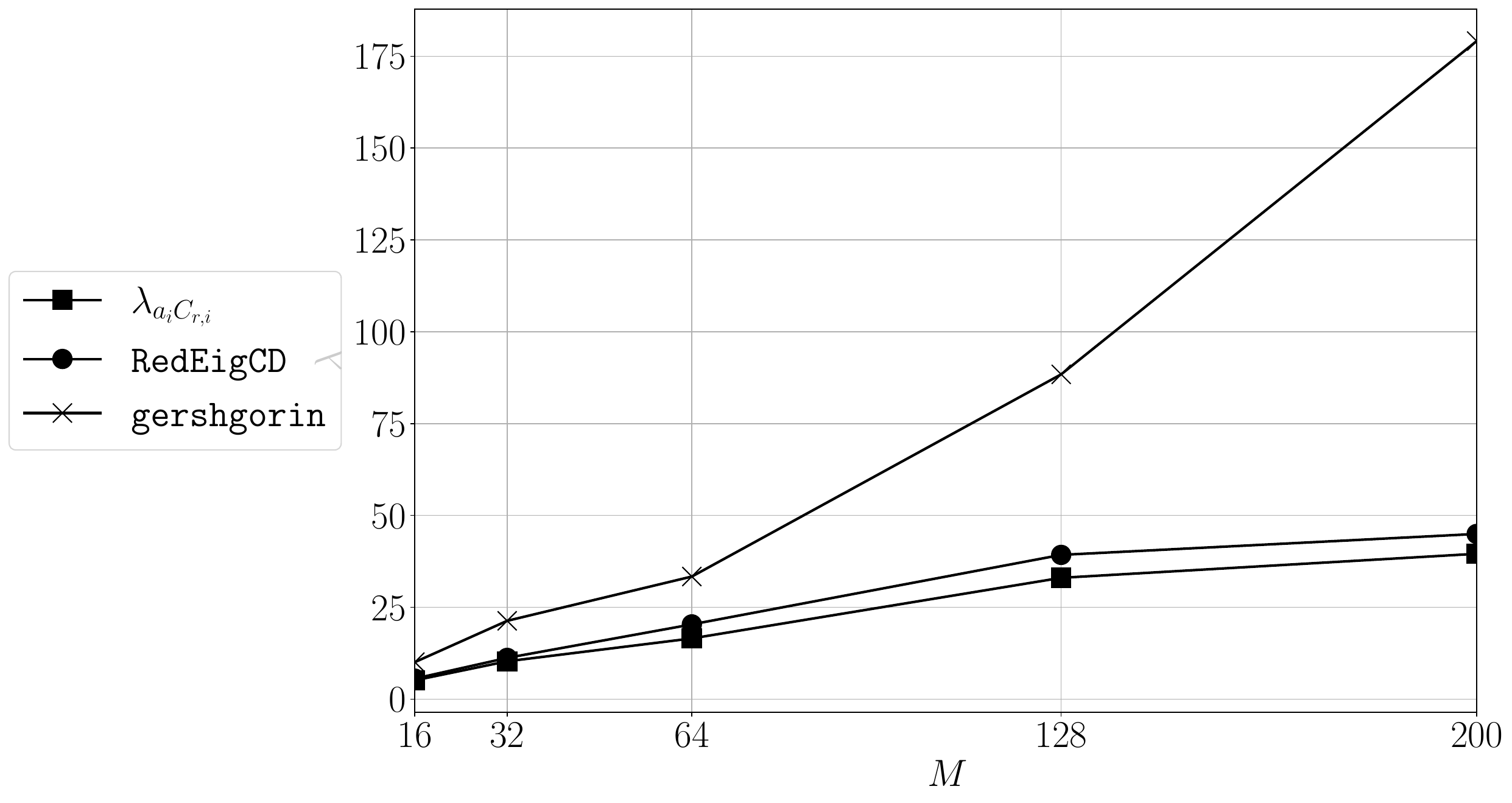}
  \includegraphics[height=0.28\textwidth]{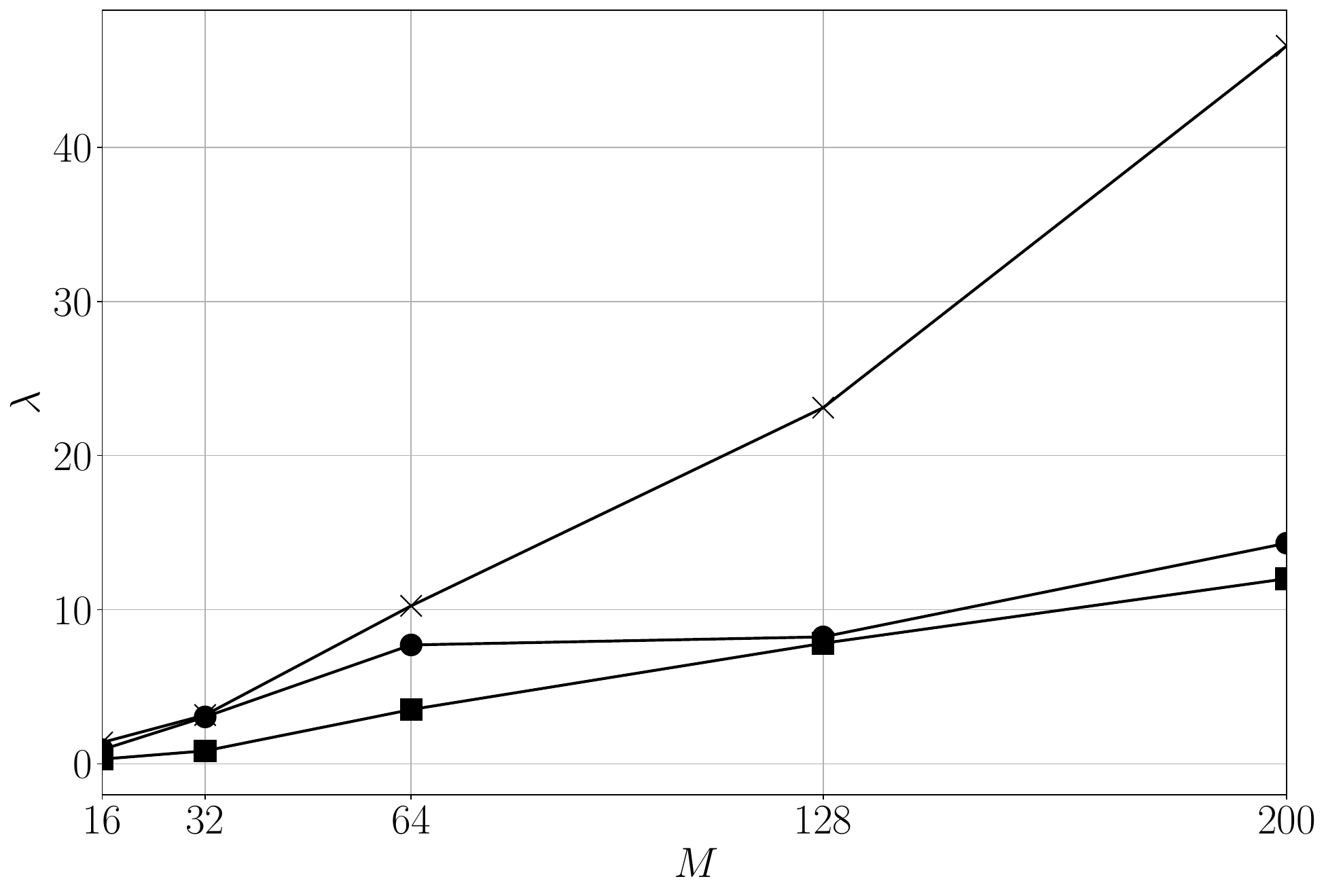}
  \caption{Comparison of the actual eigenvalues of the reduced convective operator at $t=8\pi$ for the shear layer case of Re=1000 (left) and the actuator case of Re=100 (right) against the estimated eigenvalues using \texttt{RedEigCD} and the corresponding Gershgorin circles.}
  \label{fig:eval_accuracy}
\end{figure}

To compute an overall error between the actual eigenvalues and the estimations, the following error is defined:

\begin{equation}
  \varepsilon_\mathrm{est} = \frac{||\lambda_{a_iC_{r,i}}-\lambda_\mathrm{est}||_\infty}{||\lambda_{a_iC_{r,i}}||_\infty},
  \label{eq:error_eigenvalues}
\end{equation}

\noindent where $\lambda_{a_iC_{r,i}}$ is the actual eigenbound of the Jacobian of the reduced convective operator at $t=20$, and $\lambda_\mathrm{est}$ is the estimated eigenbound. The errors for \texttt{RedEigCD} are 0.158 and 0.349 for the shear-layer and actuator case, respectively. On the other hand, for the Gershgorin approximation, the errors are 3.53 and 2.88, respectively. Therefore, the estimation of eigenbounds for both cases using \texttt{RedEigCD} is significantly more accurate than the one obtained using Gershgorin's circle theorem.

\subsection{Computational performance of \texttt{RedEigCD}}
Assessing the performance of \texttt{RedEigCD} in terms of its wall-clock time requires analyzing the online computational efficiency as well as the offline overhead. While a complexity analysis on the costs is already performed and defined in Table \ref{tab:comparison}, an assessment on actual runs is also provided here.

The online computational efficiency of the \texttt{RedEigCD} algorithm is evaluated by comparing the total online wall-clock time of the ROM against the FOM. We define the total speed-up factor, $\mathcal{S}$, as the product of the reduction in cost per iteration and the increase in the average timestep size:
\begin{equation}
\mathcal{S} = \frac{T_\mathrm{FOM}}{T_\mathrm{ROM}}\frac{{\Delta t}^\mathrm{ROM}}{\Delta t^\mathrm{FOM}},
\label{eq:total_speedup}
\end{equation}
where $T_\mathrm{ROM}$ and $T_\mathrm{FOM}$ denote the wall-clock time per iteration of the ROM and the FOM, respectively. Figure {\ref{fig:comp_eff}} visualizes these components for the shear-layer (left) and the actuator disk (right). The results indicate that while for small values of $M$ the computational complexity of $\mathcal O(M^3)$ is acceptable, for the larger values tested, hyper-reduction would be a preferable option. However, \texttt{RedEigCD} shows benefits in the whole range of investigated modes.
\begin{figure}
    \centering
    \includegraphics[height=0.3\textwidth]{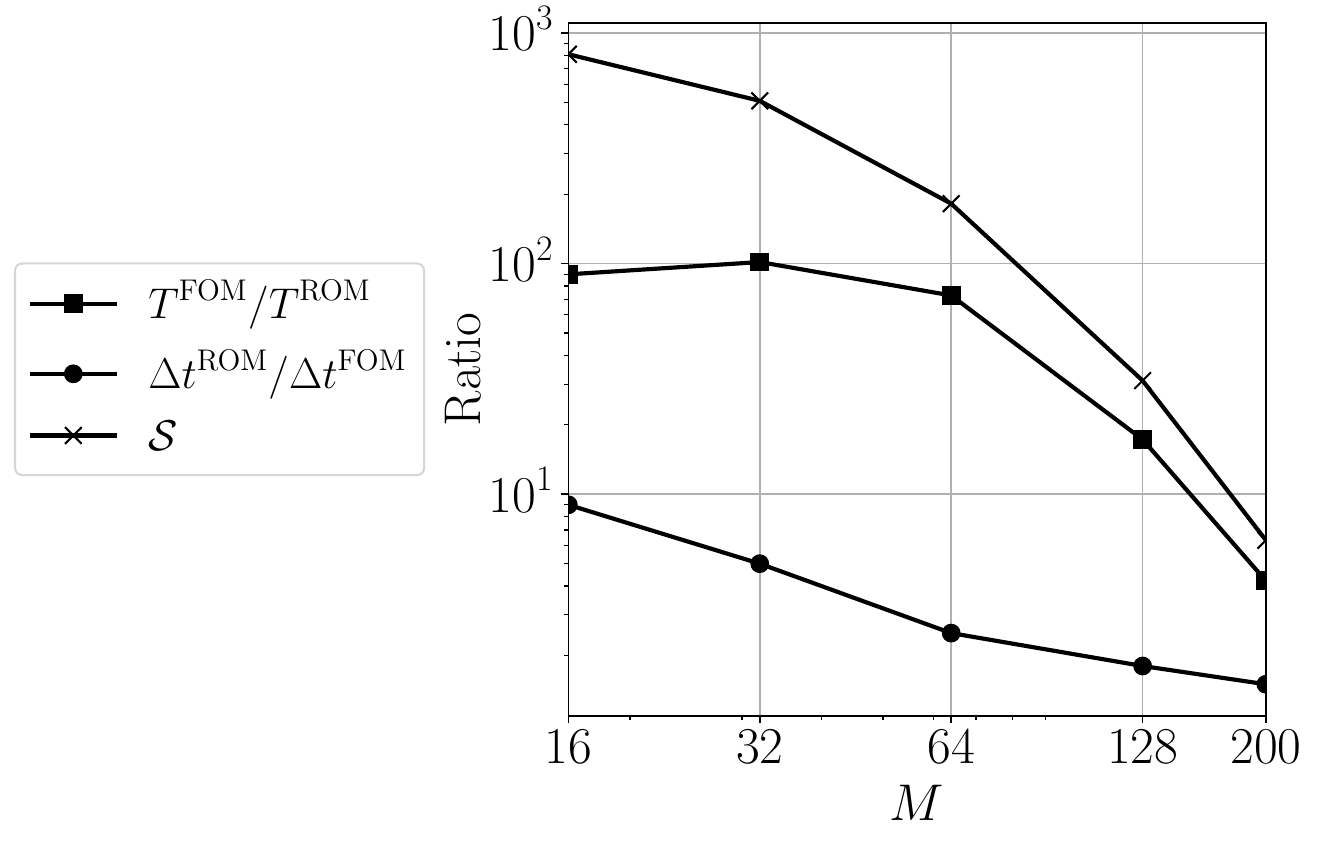}
    \includegraphics[height=0.3\textwidth]{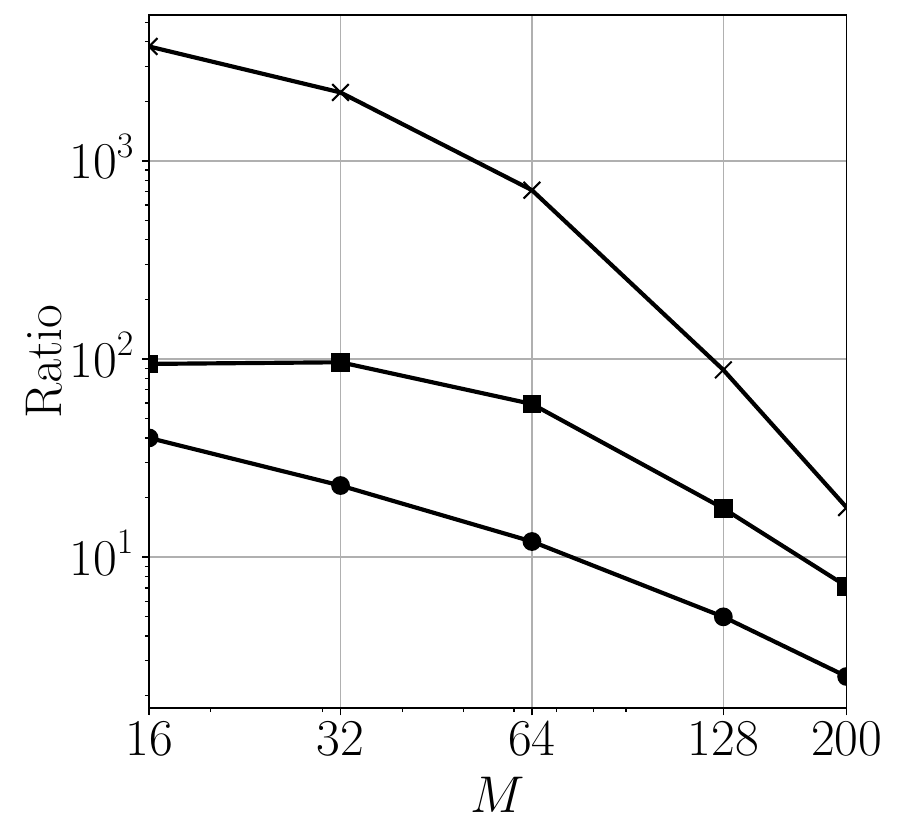}
    \caption{Total speed-up of the \texttt{RedEigCD} algorithm for the roll-up of a shear-layer at Re=1000 (left) and the numerical simulation of a wind field of Re=100 (right).}
    \label{fig:comp_eff}
\end{figure}

On the other hand, regarding the offline costs, we define a cost metric $\mathcal C$ as the equivalent of ROM timesteps required to perform the offline setup of \texttt{RedEigCD}:
\begin{equation}
   \mathcal C = \frac{T_\mathrm{offline}}{T_\mathrm{ROM}},
    \label{eq:offline_cost}
\end{equation}
where $T_\mathrm{offline}$ corresponds to the total wall-clock time spent in the offline stage. For the most expensive case ($M=200$), this yields an offline cost of approximately 20 ROM iterations. While for extremely short simulations this overhead may be perceptible, for production runs involving thousands of timesteps, the cost of the offline stage becomes less than 1\% of the whole simulation, which makes it insignificant.

% <<< End input from results.tex
% >>> Begin input from conclusion.tex
\section{Conclusion} 
\label{sec:conclusion}

In this paper, a novel self-adaptive timestepping framework, \texttt{RedEigCD}, has been presented for the time integration of POD-Galerkin ROMs of the incompressible Navier-Stokes equations. The methodology builds upon the conceptual foundations of \texttt{EigenCD} \cite{trias_self-adaptive_2011} and \texttt{AlgEigCD} \cite{trias_efficient_2024}, wherein the eigenbounds of the convective and diffusive operators are estimated to ensure the system remains at the edge of the stability region of the explicit time integration scheme. To preserve the online efficiency inherent to reduced-order modeling, \texttt{RedEigCD} avoids operations scaling with the FOM dimension by utilizing the sub-additivity of eigenbounds and Bendixson's rectangle. By extending the decomposition approach of \citet{sanderse_non-linearly_2020} and \citet{rosenberger_no_2023}, the reduced convective operator is split into quadratic and linear contributions, allowing the stability limit to be computed online as a simple linear combination of offline-calculated bounds.

Our proposed method specifically addresses the stability constraint, which defines the maximum allowable timestep to prevent numerical divergence. While in certain scenarios, particularly with implicit solvers, the maximum stable timestep can exceed that required for temporal accuracy, \texttt{RedEigCD} is tailored to explicit integration. Nonetheless, in cases where the accuracy limit may be exceeded, \texttt{RedEigCD} can be used in a dual-criterion approach.

The presence of non-homogeneous boundary conditions was addressed by splitting their contribution into symmetric and skew-symmetric parts. This allows the method to remain robust even when boundary interactions break the inherent skew-symmetry of the underlying symmetry-preserving ROM. 

Numerical validation on a periodic shear-layer roll-up demonstrated that the $\Delta t$ obtained via \texttt{RedEigCD} is systematically larger than the FOM limit. This is theoretically supported by the interlacing of eigenvalues, which ensures smaller eigenbounds for the reduced operators. The benefits in timestep size for the ROM+\texttt{RedEigCD} for the shear-layer roll-up are up to a factor of 9, depending on the modal truncation and flow conditions, yielding significant computational speedups.

Furthermore, the framework was tested in a complex wind field simulation featuring time-dependent boundary conditions and actuator disks. The results confirmed that the ROM stability limit remains consistently larger than its FOM counterpart even under transient forcing. In this configuration, \texttt{RedEigCD} achieved improvements in timestep size of up to a factor of 40. In both test cases, the method demonstrated superior accuracy compared to Gershgorin's circle theorem, providing a sharp and reliable estimate of the exact reduced eigenbounds.

Higher Reynolds numbers than those tested increase the spectral radius and tighten stability limits. In this situation, we expect \texttt{RedEigCD} to remain robust as numerical stability is governed by the eigenvalues of the linearized system. Though nonlinear stability analysis would offer a global perspective, the linear spectral bounding used here provides a more computationally efficient proxy for real-time adaptive timestepping. Flow regimes with high Reynolds number are often characterized by a slow decay of singular values, requiring a larger number of modes $M$ that severely increases the computational cost. Even though our method, which has been tested for cases with a faster decay in this study, will still provide the maximum stable timestep, the benefits in terms of timestep will be smaller in cases that require a larger number of modes, leading to timesteps closer to those of the FOM.

A key theoretical takeaway of this work is the formal connection established between linear stability theory and reduced-order modeling. Based on the combined theorems of Bendixson \cite{bendixson_sur_1902} and Rao \cite{rao_separation_1979}, it has been proven that the stable timestep achievable in a projection-based ROM is always larger than or equal to that of the FOM. This is justified by the fact that FOM stability is dictated by the smallest spatial scales (associated with the largest eigenvalues), which are precisely the components truncated during the POD projection. \texttt{RedEigCD} preserves ROM scalability with an online complexity of $\mathcal O(M)$, making the cost of the adaptive controller negligible compared to the $\mathcal O(M^3)$ cost of the state advancement.

Future work will explore the generalization of this framework to hyperreduced convective terms to further reduce online costs to $O(M^2)$. Additionally, the method will be evaluated in more complex geometries to test its robustness in highly irregular spectral environments. Finally, the proposed \texttt{RedEigCD} framework can be extended to a wide range of non-quadratic systems that admit a quadratic-bilinear representation through lifting transformations {\cite{qian_lift_2020}}, which recast many smooth nonlinearities into a form where \texttt{RedEigCD} can be directly applied. 

% <<< End input from conclusion.tex

\section*{Acknowledgements}

This work is supported by the \textit{Ministerio de Economía y Competitividad},
Spain, SIMEX project (PID2022-142174OB-I00). J.P-R. is also supported by the
Catalan Agency for Management of University and Research Grants (AGAUR), 2022
FI\_B 00810.
H.R. and B.S. are supported by the project "Discretize first, reduce next" (with project number
VI.Vidi.193.105 of the research programme NWO Talent Programme Vidi which is (partly) financed
by the Dutch Research Council (NWO)).

\appendix

% >>> Begin input from nonuniform_pod.tex
\section{Non-uniformly sampled Proper Orthogonal Decomposition of the snapshots matrix} \label{sec:pod}

The POD basis arises from the solution of the eigenvalue problem \cite{kunisch_optimal_2010}

\begin{equation}
  R\Phi_j = \lambda_j\Phi_j,
\end{equation}

\noindent where $R$ is a correlation matrix between the velocity fields and $\Phi_j$ are
the columns of the POD basis,

\begin{equation}
  R = \frac{1}{T}\int_0^T{\mathbf{u}(t)\mathbf{u}^T(t)~dt},
  \label{eq:r_pod}
\end{equation}

\noindent where $T$ is the total time of sampling. As $\Delta t$ is non-constant when performing an integration using
\texttt{AlgEigCD}, the integral from Eq.\eqref{eq:r_pod} does not reduce to
the classical $XX^T$ correlation matrix typically used in POD. Instead, the integral yields

\begin{equation}
  R = \frac{1}{T}\sum_{i=1}^m{\mathbf{u}_i\mathbf{u}_i^{T}\Delta t_i},
  \label{eq:dt_var}
\end{equation}

\noindent when using a suitable quadrature rule. Introducing $\Delta\in\mathbb{R}^{K\times K}$ as the diagonal matrix containing the
timesteps normalized with $T$ ($\Delta t_i/T$) from each step, Eq.
\eqref{eq:dt_var} yields

\begin{equation}
  R = X\Delta X^T.
\end{equation}

Let $\tilde{X}=X\Delta^{1/2}$, where performing the square root of $\Delta$ is
straightforward given it is a diagonal matrix. Then, the classical POD
optimization problem is obtained in such a way that $R=\tilde{X}\tilde{X}^T$.

To consider the weighted orthonormality condition, the weighted and
unweighted problems are related by $\Phi=\Omega^{-1/2}\hat{\Phi}$, where
$\hat{\Phi}$ stands for the solution of the unweighted problem. Therefore, the
snapshot matrix is also related in such a way that $\hat{X}
= \Omega^{1/2}\tilde{X} = \Omega^{1/2}X\Delta^{1/2}$. Hence, $\hat{\Phi}$ is a product of
the SVD of $\hat{X}$, i.e. 

\begin{equation}
\hat{X} = \hat{\Phi}\Sigma\Psi^*.
\end{equation}

\noindent Hence, the columns of $\hat{\Phi}$, denoted by $\hat{\Phi}_j$ correspond to the
eigenvectors of $\hat{X}\hat{X}^T$. Thus, the procedure to obtain $\Phi$ is
summarized as follows: (i) obtain snapshots of the velocity field to build
$X$ and as a consequence $\hat{X}$, (ii) perform the SVD of $\hat{X}$ to
obtain $\hat{\Phi}$ and as a consequence $\Phi$, (iii) truncate $\Phi$ to the
given number of modes $M$. For more details on this procedure, the reader is
referred to \cite{sanderse_non-linearly_2020}.
% <<< End input from nonuniform_pod.tex
%\input{Sections/additionalCosts}
% >>> Begin input from stg_aecd.tex
\section{Extending \texttt{AlgEigCD} to staggered grids}
\label{sec:stg_aecd}

The derivation performed by \citet{trias_efficient_2024} has been designed for general collocated and unstructured grids, making the methodology very versatile and applicable to practically any industrial application.
Even though the method is developed to be applied to any convective
scheme, it is tailored for a symmetry-preserving discretization
\cite{Trias2014} as all fields are computed at the same locations, thus
ensuring that the duality of the gradient and the divergence that build the Laplacian operator is preserved when those are applied to the velocity field,
together with the fact that for all velocity field components, the convective operator will be the same.
However, in the case of a staggered grid, given the shift in the meshes devoted to the velocity field components, this duality is not exactly preserved \cite{Verstappen2003}. Thus, the method cannot be generalized directly in this framework.

Before moving into the extension to staggered grids, it will become important to settle the base theorem used for \texttt{AlgEigCD} in order to simplify the estimates of the eigenbounds of operators to simple algebraic operations. To keep the idea from \texttt{AlgEigCD} \cite{trias_efficient_2024}, the same property is used:

\begin{theorem}
(Theorem 1 in \cite{trias_efficient_2024}). Let $A\in\mathbb{R}^{n\times m}$ and
$B\in\mathbb{R}^{m\times n}$ be rectangular matrices and $m\geq n$. Then, the
square matrices $AB\in\mathbb{R}^{n\times n}$ and $A^TB^T\in\mathbb{R}^{m\times
m}$ have the same eigenvalues except for the zero-valued ones.
  \label{th:abba}
\end{theorem}

In terms of the diffusive operator, a second-order symmetry-preserving operator
can be built by concatenating two first-order operators
\cite{Verstappen2003, sanderse_boundary_2014}, which leads to

\begin{equation}
  D = K\Lambda S,
  \label{eq:diffusive_stg}
\end{equation}

\noindent where $K\in\mathbb{R}^{N_V\times N_F}$ is a differencing matrix
from staggered faces ($N_F$) to staggered
cells ($N_V$), $\Lambda\in\mathbb{R}^{N_F\times N_F}$ corresponds to a diagonal matrix containing the diffusivities
at the staggered faces, and $S\in\mathbb{R}^{N_F\times N_V}$ is a differencing matrix from staggered cells
to staggered faces. Thus, from \cite{trias_efficient_2024} it follows that
a family of $\alpha$-dependent matrices can be built using Theorem
\ref{th:abba} such that their eigenbound is the same as the one from the diffusive term (Eq.
\eqref{eq:diffusive_stg}). Hence, applying Theorem \ref{th:abba} with $A=K\Lambda^\alpha$ and
$B=\Lambda^{1-\alpha}S$,
\begin{equation}
  \Lambda^\alpha K_h^TS_h^T\Lambda^{1-\alpha}.
  \label{eq:dtrans}
\end{equation}
Therefore, the Gershgorin circle theorem can be applied straightaway to Eq.
\eqref{eq:dtrans} as it will have the same eigenbound as the diffusive
term, which leads to
\begin{equation}
  \rho(D) \leq
  \mathrm{max}\{\mathrm{diag}(\Lambda^\alpha)\circ|K^TS^T|\mathrm{diag}(\Lambda^{1-\alpha})\},
\end{equation}

\noindent where $\circ$ denotes the Hadamard product, i.e., elementwise product.

The benefit of using this method instead of the Gershgorin circle theorem
straight-away to $D$ is that $|K^TS^T|$ can be build in a pre-process
stage.
Thus, every timestep, assuming $\Lambda$ varies as in a LES framework, only
a sparse matrix-vector product (\texttt{SpMV}) and possibly (if $\alpha\neq0$) an elementwise
vector-vector product (\texttt{axty}) should be performed, opposite to needing
to rebuild the matrix every step, as for \texttt{EigenCD}
\cite{trias_self-adaptive_2011}.

With regards to the convective term, \citet{sanderse_non-linearly_2020} defines its construction in a staggered framework as

\begin{equation}
  C(\mathbf{u}_s)\mathbf{u}_s
  = K((\Pi\mathbf{u}_s+\mathbf{y}_I)\circ(A\mathbf{u}_s+\mathbf{y}_A)),
\end{equation}

\noindent where $\Pi\in\mathbb{R}^{N_F\times N_V}$ interpolates the fluxes to
the staggered faces, and $A\in\mathbb{R}^{N_F\times N_V}$
interpolates the velocities to the staggered faces. Moreover,
$\mathbf{y}_A,\mathbf{y}_I\in\mathbb{R}^{N_F}$ apply
the boundary conditions to the interpolated fields in order to account for
them. However, $\mathbf{y}_A$ can be dropped for the stability analysis as it does not have an associated volume. Thus, they would cancel out when the volume matrix is added to the equation. However, the interpolated boundary conditions vector $y_I$ is associated with faces and thus they should still be considered. 

Thus, the convective term can be rewritten as

\begin{equation}
  C(\mathbf{u}_s) = KFA,
\end{equation}

\noindent where
$F=\mathrm{diag}(\Pi\mathbf{u}_s+\mathbf{y}_I)\in\mathbb{R}^{N_F\times N_F}$ is a matrix
containing the fluxes at the staggered faces in its diagonal.
\citet{trias_efficient_2024} shows that using Theorem \ref{th:abba} straight
onto this definition does not give any benefit, as more entries are
considered in the estimate, opposite to applying Gershgorin directly to
$C(\mathbf{u}_s)$ \cite{trias_efficient_2024}. Thus, let us recall the theorems used to develop the method.

\begin{theorem}
  (Perron-Frobenius (PF) theorem, \cite{perron_zur_1907,frobenius_uber_1912}). Given
  an irreducible matrix $A\in\mathbb{R}^{n\times n}$ such that $[A]_{ij}
  \geq 0,~\forall i,j$, it has a unique real and positive largest (in magnitude)
  eigenvalue $r\in\mathbb{R}^+$ and corresponding eigenvector with only positive entries, i.e.,
  \begin{equation}
    A\mathbf{v} = r\mathbf{v} \Rightarrow |\lambda|<
    r~\text{and}~v_i>0,~\forall i\in\{1,\dots,n\},
  \end{equation}
  \noindent where $\lambda$ denotes any eigenvalue of $A$ except $r$, and $r$ is the Perron-Frobeniuseigenvalue.
  \label{th:pf}
\end{theorem}

\begin{theorem}
  (Wielandt's theorem, \cite{gantmakher_applications_2005}). Given a matrix $A$
  satisfying Theorem \ref{th:pf}, and a matrix $B\in\mathbb{R}^{n\times n}$
  such that $|[B]_{ij}|\leq[A]_{ij}$, $\forall i,j$. Then, any $\lambda^B$
  satisfies $|\lambda^B|\leq r$, where $r$ is the PF eigenvalue of $A$.
  \label{th:wie}
\end{theorem}

\begin{theorem}
  (Nikiforov's second Lemma, \cite{nikiforov_chromatic_2007}). Let
  $A\in\mathbb{R}^{n\times n}$ be an irreducible, i.e. it is not similar via
  a permutation of a block upper triangular matrix, and a non-negative symmetric matrix
  $R\in\mathbb{R}^{n\times n}$ such that $[R]_{ii}
  = \sum_{j=1}^n{[A]_{ij}}$ and $[R]_{ij}=0$ if $i\neq j$. Then,
  \begin{equation}
    \rho\left(R+\frac{1}{q-1}A\right) \geq \frac{q}{q-1}\rho(A),
  \end{equation}
  where the equality holds iff $[R]_{ii}=[R]_{jj}$ for any combination of $i$
  and $j$.
  \label{th:niki}
\end{theorem}

Hence, Theorem \ref{th:wie} may be used to relate the eigenbound of the
convective operator and such from
\begin{equation}
  A^C\equiv |K||F||A|,
\end{equation}

\noindent which corresponds to a symmetric positive matrix. Thus, by construction of
$A^C$, its off-diagonal terms correspond to $2|C|$ as follows, $A^{C,off}
= A^C-\mathrm{diag}(\mathrm{diag}(A^C)) = 2|C|$, while the rowsums of
$A^{C,off}$ are equal to the diagonal terms of $A^C$ given the skew-symmetry of
$C$. Thus, $A^{C, off}$ satisfies PF, which implies it has a unique largest
eigenvalue. Hence, applying Theorem \ref{th:wie}, this can relate the
eigenvalues of $|C|$ and those from $C$. Thus, Theorem \ref{th:niki} can be
applied to $A^{C,off}$ and $\mathrm{diag}(\mathrm{diag}(A^C))$ such that with
$q=2$, the left-hand side becomes $\rho(A^C)$. This leads to

\begin{equation}
  \rho(A^C) \overset{\mathrm{Thm~\ref{th:niki}}}{\geq} 2\rho(A^{C,off}) = 4\rho(|C|)
  \overset{\mathrm{Thm~\ref{th:wie}}}{\geq} 4\rho(C),
  \label{eq:rel_rho}
\end{equation}

\noindent as in \cite{trias_efficient_2024}, yet with the matrices having
a different definition given their staggered nature. Now applying Eq.
\eqref{eq:rel_rho} together with Theorem \ref{th:abba} and Gershgorin circle
theorem, this leads to

\begin{equation}
  \rho(C) \leq \frac{1}{4}
  \mathrm{max}\{\mathrm{diag}(|F|^\alpha)\circ||K|^T|A|^T|\mathrm{diag}(|F|^{1-\alpha})\}.
\end{equation}

However, as presented in \cite{trias_efficient_2024}, the relevant operators are scaled by $\Omega^{-1}$. Therefore, the eigenbounds of both convective and diffusive operators read as

\begin{subequations}
  \begin{align}
    \rho(\Omega^{-1}D) \leq&
    \mathrm{max}\{\mathrm{diag}(\Lambda^\alpha)\circ|K^T\Omega^{-1}S^T|\mathrm{diag}(\Lambda^{1-\alpha})\},
    \\
    \rho(\Omega^{-1}C) \leq& \frac{1}{2}
    \mathrm{max}\{\mathrm{diag}(|F|^\alpha)\circ||K|^T\Omega^{-1}|A|^T|\mathrm{diag}(|F|^{1-\alpha})\}.
  \end{align}
  \label{eq:stg_aecd}
\end{subequations}

The original version of \texttt{AlgEigCD} \cite{trias_efficient_2024} showed that for a collocated solution of a Rayleigh-Bénard convection problem of $\text{Ra}=10^{10}$, the most suitable value of $\alpha$ should be 0, as it minimizes the bounds for both diffusive and convective operator regardless of the mesh used. This is tested for a 2D shear-layer roll-up for four different uniform meshes with $N=\{10, 20, 40, 80\}$ divisions per direction. Figure \ref{fig:alpha} shows that, while the eigenbounds for the diffusive operator are independent of $\alpha$ (left), the bounds for the convective operator are minimized for a value of $\alpha$ around 0, while still depending on the mesh size.

\begin{figure}[h]
    \centering
    \includegraphics[height=0.3\linewidth]{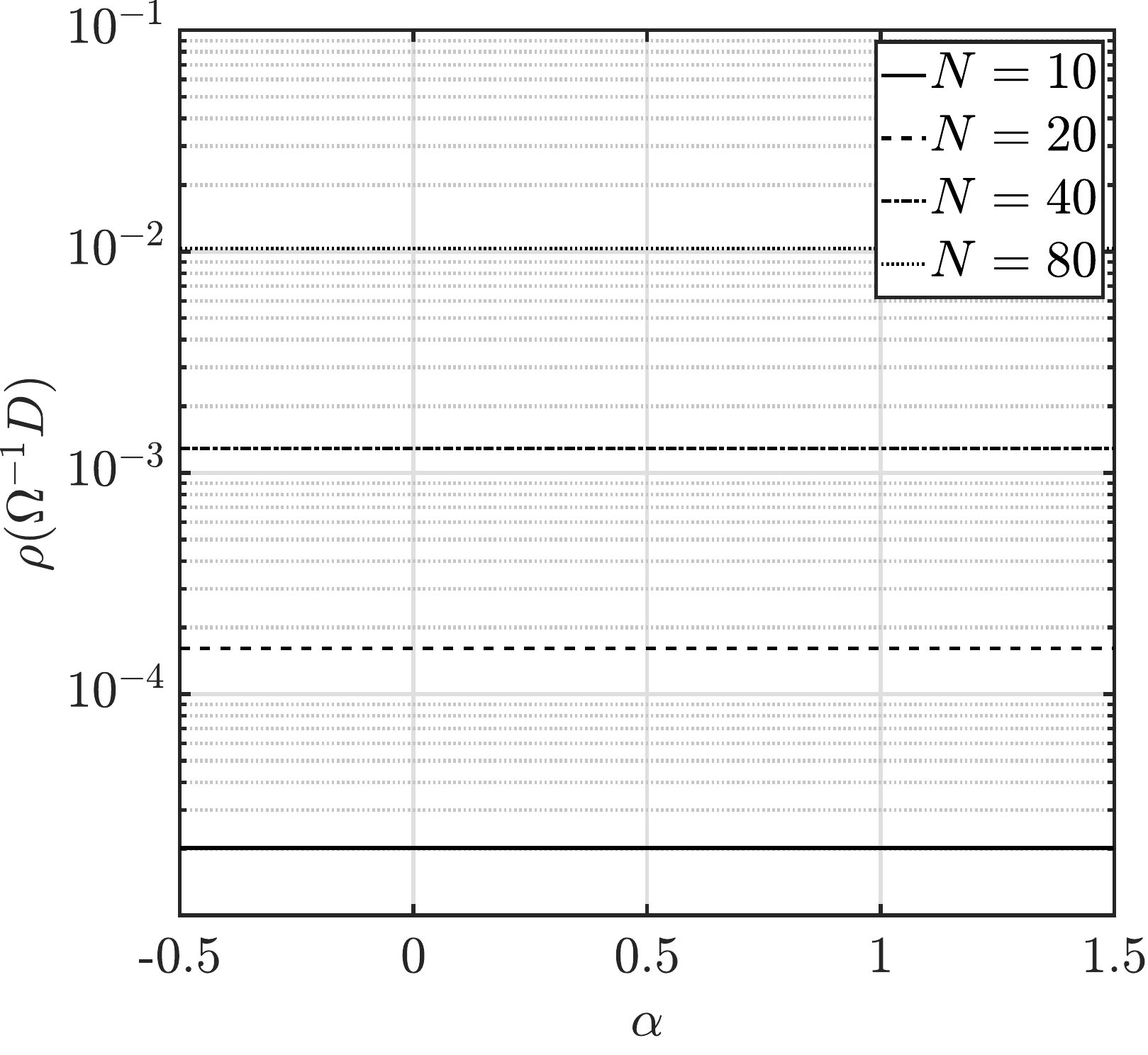}
    \includegraphics[height=0.3\linewidth]{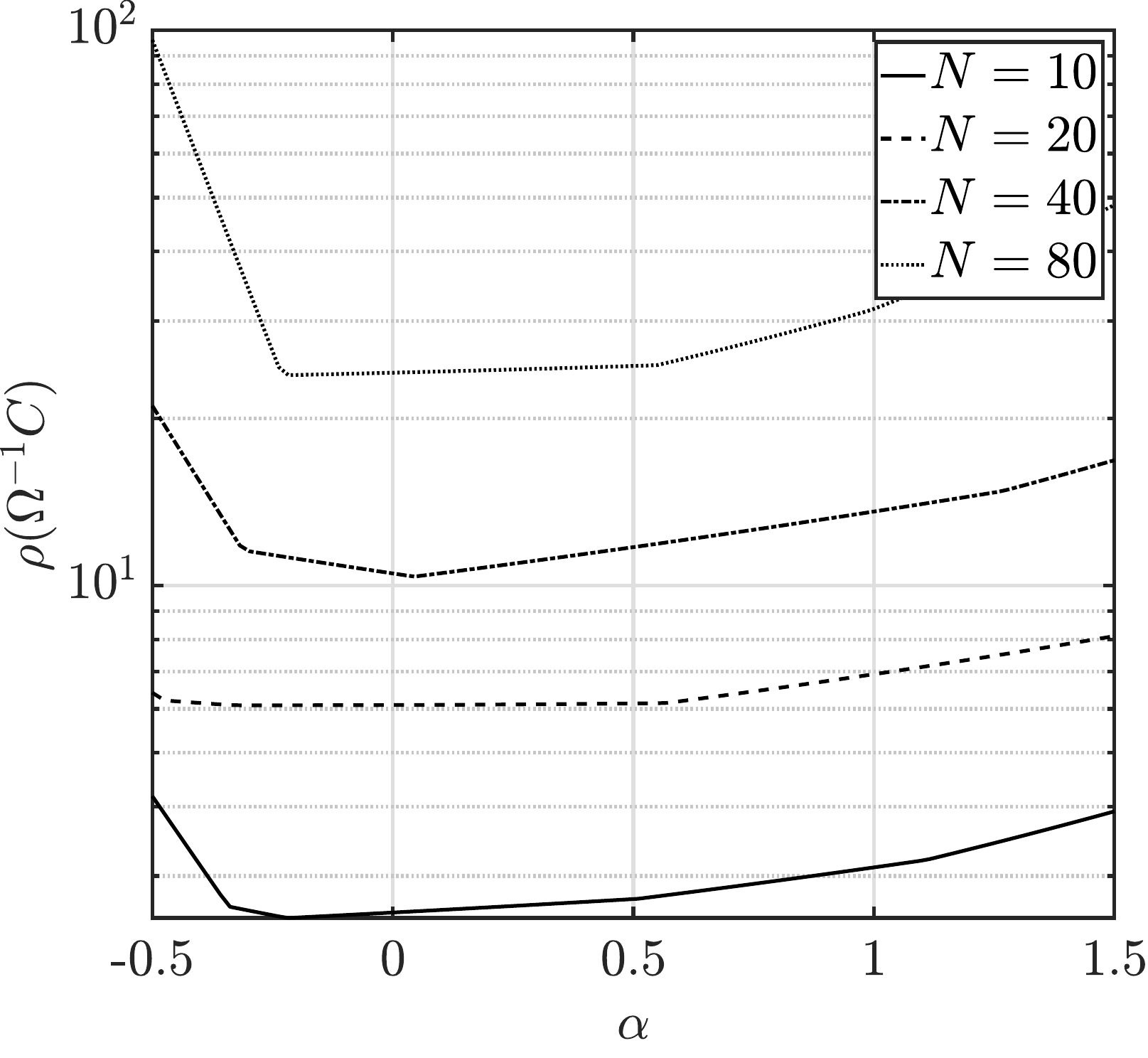}
    \caption{Eigenbound estimation for the diffusive (left) and convective (right) operators for the roll-up of a shear-layer of $\text{Re}=100$ with different mesh sizes as a function of $\alpha$.}
    \label{fig:alpha}
\end{figure}

However, the use of $\alpha\neq0$ introduces an additional operation in Eq.\eqref{eq:stg_aecd}, doubling the number of operations required. Therefore, as the potential benefits of $\alpha\neq0$ are almost negligible (Fig. \ref{fig:alpha}, right), $\alpha=0$ is proposed in this case, which yields

\begin{subequations}
  \begin{align}
    \rho(\Omega^{-1}D) \leq&
    \mathrm{max}\{|K^T\Omega^{-1}S^T|\mathrm{diag}(\Lambda)\},
    \\
    \rho(\Omega^{-1}C) \leq& \frac{1}{2}
    \mathrm{max}\{||K|^T\Omega^{-1}|A|^T|\mathrm{diag}(|F|)\}.
  \end{align}
  \label{eq:stg_aecd}
\end{subequations}

% <<< End input from stg_aecd.tex
% >>> Begin input from jacobianAppendix.tex
\section{Jacobian of the convective term for a generic convection-diffusion equation} \label{sec:jacobian}

The momentum equation of the Navier-Stokes equations (Eq.\eqref{eq:mom_ns}) is a particular case of a generic convection-diffusion equation,

\begin{equation}
    \frac{\partial\phi}{\partial t} + (\mathbf{u}\cdot\nabla)\phi = \frac{1}{\text{Pe}}\nabla^2\phi + S_\phi,
\end{equation}

\noindent where the transported variable $\phi = \mathbf{u}$, $S_\phi = -\nabla p$ and the Péclet number $\mathrm{Pe}=\mathrm{Re}$. After spatial discretization with a finite-volume method, this generic convection-diffusion equation is written as

\begin{equation}
    \Omega\frac{d\bm\phi}{dt}+C(\mathbf u_c)\bm\phi - D\bm\phi + \Omega\mathbf S = \mathbf 0,
    \label{eq:fom_cd}
\end{equation}

\noindent where $\mathbf u_c$ is the discrete convecting velocity field, $\mathbf S$ is the discrete source term and $\bm\phi\in\mathbb R^{N_C}$ is the discrete generic variable, with $N_C$ being the number of cells ($N_C\approx N_V/d$, being $d$ the number of dimensions). Following Section \ref{sec:rom}, a POD-Galerkin approximation is obtained for both $\mathbf u_c,\bm\phi$. The former is equivalent to Eq.\eqref{eq:ansatz}, while the latter reads as

\begin{equation}
    \bm\phi(t) \approx \bm\phi_r(t) = \Theta\mathbf b(t),
    \label{eq:ansatz_phi}
\end{equation}

where $\Theta\in\mathbb{R}^{N_C\times M}$ is the projection matrix for $\bm\phi$, and $\mathbf b(t)\in\mathbb R^M$ is the coefficient vector for $\bm\phi$. Substituting both Ansatz (Eqs. \eqref{eq:ansatz}, \eqref{eq:ansatz_phi}) with $\Theta$ and $\Phi$ obtained with the procedure detailed in \ref{sec:pod} and projecting it, the ROM equation reads as

\begin{equation}
    \frac{d\mathbf b}{dt} = - \Theta^TC(\Phi\mathbf a_c)\Theta\mathbf b + \Theta^TD\Theta\mathbf b + \Theta^T\mathbf S.
\end{equation}

The reduced convective term now reads $\Theta^TC(\Phi\mathbf a_c)\Theta\mathbf b$, which following the procedure from Eq.\eqref{eq:red_conv} is rewritten as

\begin{equation}
    C_r(\mathbf a_c, \mathbf b) = \Theta^TC(\Phi\mathbf a_c)\Theta\mathbf b = [\underbrace{\Theta^TC(\Phi_1)\Theta}_{C_{r,1}}~\dots~C_{r,M}](\mathbf a_c\otimes\mathbf b) = \left(\sum_{i=1}^M{a_{c,i}C_{r,i}}\right)\mathbf b.
    \label{eq:red_conv_phi}
\end{equation}

The linear stability analysis for this convective term now requires linearizing the reduced convective term. Therefore, the Jacobian of the convective term reads as

\begin{equation}
    J_C(\mathbf a_c) = \frac{\partial C_r (\mathbf a_c, \mathbf b)}{d\mathbf b} = [C_{r,1}~\dots~C_{r,M}](\mathbf a_c \otimes I_M) = \sum_{i=1}^M{a_{c,i}C_{r,i}}.
\end{equation}
% <<< End input from jacobianAppendix.tex
%\input{Sections/upwinding}

\bibliographystyle{elsarticle-num-names} 
\biboptions{sort&compress}
\bibliography{export}

\begin{thebibliography}{66}
\expandafter\ifx\csname natexlab\endcsname\relax\def\natexlab#1{#1}\fi
\providecommand{\url}[1]{\texttt{#1}}
\providecommand{\href}[2]{#2}
\providecommand{\path}[1]{#1}
\providecommand{\DOIprefix}{doi:}
\providecommand{\ArXivprefix}{arXiv:}
\providecommand{\URLprefix}{URL: }
\providecommand{\Pubmedprefix}{pmid:}
\providecommand{\doi}[1]{\href{http://dx.doi.org/#1}{\path{#1}}}
\providecommand{\Pubmed}[1]{\href{pmid:#1}{\path{#1}}}
\providecommand{\bibinfo}[2]{#2}
\ifx\xfnm\relax \def\xfnm[#1]{\unskip,\space#1}\fi
%Type = Book
\bibitem[{Hairer et~al.(1993)Hairer, Wanner, and
  Noersett}]{hairer_solving_1993}
\bibinfo{author}{E.~Hairer}, \bibinfo{author}{G.~Wanner},
  \bibinfo{author}{S.~P. Noersett}, \bibinfo{title}{Solving {Ordinary}
  {Differential} {Equations} {I}}, volume~\bibinfo{volume}{8} of
  \textit{\bibinfo{series}{Springer {Series} in {Computational}
  {Mathematics}}}, \bibinfo{publisher}{Springer Berlin Heidelberg},
  \bibinfo{address}{Berlin, Heidelberg}, \bibinfo{year}{1993}. \URLprefix
  \url{http://link.springer.com/10.1007/978-3-540-78862-1}.
  \DOIprefix\doi{10.1007/978-3-540-78862-1}.
%Type = Book
\bibitem[{Butcher(2016)}]{butcher_numerical_2016}
\bibinfo{author}{J.~C. Butcher}, \bibinfo{title}{Numerical methods for ordinary
  differential equations}, \bibinfo{edition}{3rd ed} ed.,
  \bibinfo{publisher}{Wiley}, \bibinfo{address}{Chichester, West Sussex},
  \bibinfo{year}{2016}.
%Type = Article
\bibitem[{Dormand and Prince(1980)}]{dormand_family_1980}
\bibinfo{author}{J.~Dormand}, \bibinfo{author}{P.~Prince},
\newblock \bibinfo{title}{A family of embedded {Runge}-{Kutta} formulae},
\newblock \bibinfo{journal}{Journal of Computational and Applied Mathematics}
  \bibinfo{volume}{6} (\bibinfo{year}{1980}) \bibinfo{pages}{19--26}.
  \URLprefix
  \url{https://linkinghub.elsevier.com/retrieve/pii/0771050X80900133}.
  \DOIprefix\doi{10.1016/0771-050X(80)90013-3}.
%Type = Article
\bibitem[{Fehlberg(1969)}]{fehlberg_klassische_1969}
\bibinfo{author}{E.~Fehlberg},
\newblock \bibinfo{title}{Klassische {Runge}-{Kutta}-{Formeln} fünfter und
  siebenter {Ordnung} mit {Schrittweiten}-{Kontrolle}},
\newblock \bibinfo{journal}{Computing} \bibinfo{volume}{4}
  (\bibinfo{year}{1969}) \bibinfo{pages}{93--106}. \URLprefix
  \url{http://link.springer.com/10.1007/BF02234758}.
  \DOIprefix\doi{10.1007/BF02234758}.
%Type = Article
\bibitem[{Courant et~al.(1928)Courant, Friedrichs, and
  Lewy}]{courant_uber_1928}
\bibinfo{author}{R.~Courant}, \bibinfo{author}{K.~Friedrichs},
  \bibinfo{author}{H.~Lewy},
\newblock \bibinfo{title}{Über die partiellen {Differenzengleichungen} der
  mathematischen {Physik}},
\newblock \bibinfo{journal}{Mathematische Annalen} \bibinfo{volume}{100}
  (\bibinfo{year}{1928}) \bibinfo{pages}{32--74}.
  \DOIprefix\doi{10.1007/BF01448839}.
%Type = Article
\bibitem[{Trias and Lehmkuhl(2011)}]{trias_self-adaptive_2011}
\bibinfo{author}{F.~X. Trias}, \bibinfo{author}{O.~Lehmkuhl},
\newblock \bibinfo{title}{A {Self}-{Adaptive} {Strategy} for the {Time}
  {Integration} of {Navier}-{Stokes} {Equations}},
\newblock \bibinfo{journal}{Numerical Heat Transfer, Part B: Fundamentals}
  \bibinfo{volume}{60} (\bibinfo{year}{2011}) \bibinfo{pages}{116--134}.
  \URLprefix
  \url{http://www.tandfonline.com/doi/abs/10.1080/10407790.2011.594398}.
  \DOIprefix\doi{10.1080/10407790.2011.594398}.
%Type = Article
\bibitem[{Trias et~al.(2024)Trias, Álvarez Farré, Alsalti-Baldellou,
  Gorobets, and Oliva}]{trias_efficient_2024}
\bibinfo{author}{F.~X. Trias}, \bibinfo{author}{X.~Álvarez Farré},
  \bibinfo{author}{A.~Alsalti-Baldellou}, \bibinfo{author}{A.~Gorobets},
  \bibinfo{author}{A.~Oliva},
\newblock \bibinfo{title}{An efficient eigenvalue bounding method: {CFL}
  condition revisited},
\newblock \bibinfo{journal}{Computer Physics Communications}
  \bibinfo{volume}{305} (\bibinfo{year}{2024}).
  \DOIprefix\doi{10.1016/j.cpc.2024.109351}, \bibinfo{note}{publisher: Elsevier
  B.V.}
%Type = Article
\bibitem[{Sanderse(2020)}]{sanderse_non-linearly_2020}
\bibinfo{author}{B.~Sanderse},
\newblock \bibinfo{title}{Non-linearly stable reduced-order models for
  incompressible flow with energy-conserving finite volume methods},
\newblock \bibinfo{journal}{Journal of Computational Physics}
  \bibinfo{volume}{421} (\bibinfo{year}{2020}).
  \DOIprefix\doi{10.1016/j.jcp.2020.109736}, \bibinfo{note}{arXiv: 1909.11462
  Publisher: Academic Press Inc.}
%Type = Article
\bibitem[{Leblond et~al.(2011)Leblond, Allery, and
  Inard}]{leblond_optimal_2011}
\bibinfo{author}{C.~Leblond}, \bibinfo{author}{C.~Allery},
  \bibinfo{author}{C.~Inard},
\newblock \bibinfo{title}{An optimal projection method for the reduced-order
  modeling of incompressible flows},
\newblock \bibinfo{journal}{Computer Methods in Applied Mechanics and
  Engineering} \bibinfo{volume}{200} (\bibinfo{year}{2011})
  \bibinfo{pages}{2507--2527}. \URLprefix
  \url{https://www.sciencedirect.com/science/article/pii/S0045782511001617}.
  \DOIprefix\doi{10.1016/j.cma.2011.04.020}.
%Type = Article
\bibitem[{Reyes and Codina(2020)}]{reyes_projection-based_2020}
\bibinfo{author}{R.~Reyes}, \bibinfo{author}{R.~Codina},
\newblock \bibinfo{title}{Projection-based reduced order models for flow
  problems: {A} variational multiscale approach},
\newblock \bibinfo{journal}{Computer Methods in Applied Mechanics and
  Engineering} \bibinfo{volume}{363} (\bibinfo{year}{2020})
  \bibinfo{pages}{112844}. \URLprefix
  \url{https://www.sciencedirect.com/science/article/pii/S0045782520300256}.
  \DOIprefix\doi{10.1016/j.cma.2020.112844}.
%Type = Article
\bibitem[{Prakash and Zhang(2024)}]{prakash_projection-based_2024}
\bibinfo{author}{A.~Prakash}, \bibinfo{author}{Y.~J. Zhang},
\newblock \bibinfo{title}{Projection-based reduced order modeling and
  data-driven artificial viscosity closures for incompressible fluid flows},
\newblock \bibinfo{journal}{Computer Methods in Applied Mechanics and
  Engineering} \bibinfo{volume}{425} (\bibinfo{year}{2024})
  \bibinfo{pages}{116930}. \URLprefix
  \url{https://www.sciencedirect.com/science/article/pii/S0045782524001865}.
  \DOIprefix\doi{10.1016/j.cma.2024.116930}.
%Type = Incollection
\bibitem[{Lumley(1967)}]{lumley_structure_1967}
\bibinfo{author}{J.~L. Lumley},
\newblock \bibinfo{title}{The structure of inhomogeneous turbulence},
\newblock in: \bibinfo{booktitle}{Atmospheric {Turbulence} and {Wave}
  {Propagation}}, \bibinfo{address}{Moscow}, \bibinfo{year}{1967}, pp.
  \bibinfo{pages}{166--178}.
%Type = Article
\bibitem[{Sirovich(1987)}]{sirovich_turbulence_1987}
\bibinfo{author}{L.~Sirovich},
\newblock \bibinfo{title}{Turbulence and the dynamics of coherent structures.
  {I}. {Coherent} structures},
\newblock \bibinfo{journal}{Quarterly of Applied Mathematics}
  \bibinfo{volume}{45} (\bibinfo{year}{1987}) \bibinfo{pages}{561--571}.
  \URLprefix
  \url{https://www.ams.org/qam/1987-45-03/S0033-569X-1987-0910462-6/}.
  \DOIprefix\doi{10.1090/qam/910462}.
%Type = Book
\bibitem[{Holmes et~al.(1996)Holmes, Lumley, and
  Berkooz}]{holmes_turbulence_1996}
\bibinfo{author}{P.~Holmes}, \bibinfo{author}{J.~L. Lumley},
  \bibinfo{author}{G.~Berkooz}, \bibinfo{title}{Turbulence, {Coherent}
  {Structures}, {Dynamical} {Systems} and {Symmetry}}, \bibinfo{edition}{1}
  ed., \bibinfo{publisher}{Cambridge University Press}, \bibinfo{year}{1996}.
  \URLprefix
  \url{https://www.cambridge.org/core/product/identifier/9780511622700/type/book}.
  \DOIprefix\doi{10.1017/CBO9780511622700}.
%Type = Article
\bibitem[{Quarteroni and Rozza(2007)}]{quarteroni_numerical_2007}
\bibinfo{author}{A.~Quarteroni}, \bibinfo{author}{G.~Rozza},
\newblock \bibinfo{title}{Numerical solution of parametrized
  {Navier}–{Stokes} equations by reduced basis methods},
\newblock \bibinfo{journal}{Numerical Methods for Partial Differential
  Equations} \bibinfo{volume}{23} (\bibinfo{year}{2007})
  \bibinfo{pages}{923--948}. \URLprefix
  \url{https://onlinelibrary.wiley.com/doi/10.1002/num.20249}.
  \DOIprefix\doi{10.1002/num.20249}.
%Type = Article
\bibitem[{Manzoni(2014)}]{manzoni_efficient_2014}
\bibinfo{author}{A.~Manzoni},
\newblock \bibinfo{title}{An efficient computational framework for reduced
  basis approximation and \textit{a posteriori} error estimation of
  parametrized {Navier}–{Stokes} flows},
\newblock \bibinfo{journal}{ESAIM: Mathematical Modelling and Numerical
  Analysis} \bibinfo{volume}{48} (\bibinfo{year}{2014})
  \bibinfo{pages}{1199--1226}. \URLprefix
  \url{http://www.esaim-m2an.org/10.1051/m2an/2014013}.
  \DOIprefix\doi{10.1051/m2an/2014013}.
%Type = Incollection
\bibitem[{Lassila et~al.(2014)Lassila, Manzoni, Quarteroni, and
  Rozza}]{lassila_model_2014}
\bibinfo{author}{T.~Lassila}, \bibinfo{author}{A.~Manzoni},
  \bibinfo{author}{A.~Quarteroni}, \bibinfo{author}{G.~Rozza},
\newblock \bibinfo{title}{Model {Order} {Reduction} in {Fluid} {Dynamics}:
  {Challenges} and {Perspectives}},
\newblock in: \bibinfo{booktitle}{Reduced {Order} {Methods} for {Modeling} and
  {Computational} {Reduction}}, \bibinfo{publisher}{Springer International
  Publishing}, \bibinfo{address}{Cham}, \bibinfo{year}{2014}, pp.
  \bibinfo{pages}{235--273}. \URLprefix
  \url{http://link.springer.com/10.1007/978-3-319-02090-7}.
  \DOIprefix\doi{10.1007/978-3-319-02090-7}.
%Type = Article
\bibitem[{Amsallem and Farhat(2012)}]{amsallem_stabilization_2012}
\bibinfo{author}{D.~Amsallem}, \bibinfo{author}{C.~Farhat},
\newblock \bibinfo{title}{Stabilization of projection-based reduced-order
  models},
\newblock \bibinfo{journal}{International Journal for Numerical Methods in
  Engineering} \bibinfo{volume}{91} (\bibinfo{year}{2012})
  \bibinfo{pages}{358--377}. \URLprefix
  \url{https://onlinelibrary.wiley.com/doi/abs/10.1002/nme.4274}.
  \DOIprefix\doi{10.1002/nme.4274}, \bibinfo{note}{\_eprint:
  https://onlinelibrary.wiley.com/doi/pdf/10.1002/nme.4274}.
%Type = Article
\bibitem[{Balajewicz et~al.(2016)Balajewicz, Tezaur, and
  Dowell}]{balajewicz2016minimal}
\bibinfo{author}{M.~Balajewicz}, \bibinfo{author}{I.~Tezaur},
  \bibinfo{author}{E.~Dowell},
\newblock \bibinfo{title}{Minimal subspace rotation on the stiefel manifold for
  stabilization and enhancement of projection-based reduced order models for
  the compressible navier--stokes equations},
\newblock \bibinfo{journal}{Journal of Computational Physics}
  \bibinfo{volume}{321} (\bibinfo{year}{2016}) \bibinfo{pages}{224--241}.
%Type = Article
\bibitem[{Rezaian and Wei(2021)}]{rezaian_global_2021}
\bibinfo{author}{E.~Rezaian}, \bibinfo{author}{M.~Wei},
\newblock \bibinfo{title}{A global eigenvalue reassignment method for the
  stabilization of nonlinear reduced-order models},
\newblock \bibinfo{journal}{International Journal for Numerical Methods in
  Engineering} \bibinfo{volume}{122} (\bibinfo{year}{2021})
  \bibinfo{pages}{2393--2416}. \URLprefix
  \url{https://onlinelibrary.wiley.com/doi/abs/10.1002/nme.6625}.
  \DOIprefix\doi{10.1002/nme.6625}, \bibinfo{note}{\_eprint:
  https://onlinelibrary.wiley.com/doi/pdf/10.1002/nme.6625}.
%Type = Article
\bibitem[{Wang et~al.(2012)Wang, Akhtar, Borggaard, and
  Iliescu}]{wang_proper_2012}
\bibinfo{author}{Z.~Wang}, \bibinfo{author}{I.~Akhtar},
  \bibinfo{author}{J.~Borggaard}, \bibinfo{author}{T.~Iliescu},
\newblock \bibinfo{title}{Proper orthogonal decomposition closure models for
  turbulent flows: {A} numerical comparison},
\newblock \bibinfo{journal}{Computer Methods in Applied Mechanics and
  Engineering} \bibinfo{volume}{237-240} (\bibinfo{year}{2012})
  \bibinfo{pages}{10--26}. \URLprefix
  \url{https://linkinghub.elsevier.com/retrieve/pii/S0045782512001429}.
  \DOIprefix\doi{10.1016/j.cma.2012.04.015}.
%Type = Article
\bibitem[{Ahmed et~al.(2021)Ahmed, Pawar, San, Rasheed, Iliescu, and
  Noack}]{ahmed_closures_2021}
\bibinfo{author}{S.~E. Ahmed}, \bibinfo{author}{S.~Pawar},
  \bibinfo{author}{O.~San}, \bibinfo{author}{A.~Rasheed},
  \bibinfo{author}{T.~Iliescu}, \bibinfo{author}{B.~R. Noack},
\newblock \bibinfo{title}{On closures for reduced order models—{A} spectrum
  of first-principle to machine-learned avenues},
\newblock \bibinfo{journal}{Physics of Fluids} \bibinfo{volume}{33}
  (\bibinfo{year}{2021}) \bibinfo{pages}{091301}. \URLprefix
  \url{https://pubs.aip.org/pof/article/33/9/091301/1030870/On-closures-for-reduced-order-models-A-spectrum-of}.
  \DOIprefix\doi{10.1063/5.0061577}.
%Type = Article
\bibitem[{Rosenberger and Sanderse(2023)}]{rosenberger_no_2023}
\bibinfo{author}{H.~Rosenberger}, \bibinfo{author}{B.~Sanderse},
\newblock \bibinfo{title}{No pressure? {Energy}-consistent {ROMs} for the
  incompressible {Navier}-{Stokes} equations with time-dependent boundary
  conditions},
\newblock \bibinfo{journal}{Journal of Computational Physics}
  \bibinfo{volume}{491} (\bibinfo{year}{2023}) \bibinfo{pages}{112405}.
  \URLprefix
  \url{https://linkinghub.elsevier.com/retrieve/pii/S0021999123005004}.
  \DOIprefix\doi{10.1016/j.jcp.2023.112405}.
%Type = Article
\bibitem[{Klein and Sanderse(2024)}]{klein_energy-conserving_2024}
\bibinfo{author}{R.~B. Klein}, \bibinfo{author}{B.~Sanderse},
\newblock \bibinfo{title}{Energy-conserving hyper-reduction and temporal
  localization for reduced order models of the incompressible {Navier}-{Stokes}
  equations},
\newblock \bibinfo{journal}{Journal of Computational Physics}
  \bibinfo{volume}{499} (\bibinfo{year}{2024}) \bibinfo{pages}{112697}.
  \URLprefix
  \url{https://www.sciencedirect.com/science/article/pii/S0021999123007921}.
  \DOIprefix\doi{10.1016/j.jcp.2023.112697}.
%Type = Article
\bibitem[{Farhat et~al.(2015)Farhat, Chapman, and
  Avery}]{farhat_structure-preserving_2015}
\bibinfo{author}{C.~Farhat}, \bibinfo{author}{T.~Chapman},
  \bibinfo{author}{P.~Avery},
\newblock \bibinfo{title}{Structure-preserving, stability, and accuracy
  properties of the energy-conserving sampling and weighting method for the
  hyper reduction of nonlinear finite element dynamic models},
\newblock \bibinfo{journal}{International Journal for Numerical Methods in
  Engineering} \bibinfo{volume}{102} (\bibinfo{year}{2015})
  \bibinfo{pages}{1077--1110}. \URLprefix
  \url{https://onlinelibrary.wiley.com/doi/abs/10.1002/nme.4820}.
  \DOIprefix\doi{10.1002/nme.4820}, \bibinfo{note}{\_eprint:
  https://onlinelibrary.wiley.com/doi/pdf/10.1002/nme.4820}.
%Type = Article
\bibitem[{Barrault et~al.(2004)Barrault, Maday, Nguyen, and
  Patera}]{barrault_empirical_2004}
\bibinfo{author}{M.~Barrault}, \bibinfo{author}{Y.~Maday},
  \bibinfo{author}{N.~C. Nguyen}, \bibinfo{author}{A.~T. Patera},
\newblock \bibinfo{title}{An ‘empirical interpolation’ method: application
  to efficient reduced-basis discretization of partial differential equations},
\newblock \bibinfo{journal}{Comptes Rendus Mathematique} \bibinfo{volume}{339}
  (\bibinfo{year}{2004}) \bibinfo{pages}{667--672}. \URLprefix
  \url{https://www.sciencedirect.com/science/article/pii/S1631073X04004248}.
  \DOIprefix\doi{10.1016/j.crma.2004.08.006}.
%Type = Article
\bibitem[{Chaturantabut and Sorensen(2012)}]{chaturantabut_state_2012}
\bibinfo{author}{S.~Chaturantabut}, \bibinfo{author}{D.~C. Sorensen},
\newblock \bibinfo{title}{A {State} {Space} {Error} {Estimate} for {POD}-{DEIM}
  {Nonlinear} {Model} {Reduction}},
\newblock \bibinfo{journal}{SIAM Journal on Numerical Analysis}
  \bibinfo{volume}{50} (\bibinfo{year}{2012}) \bibinfo{pages}{46--63}.
  \URLprefix \url{https://epubs.siam.org/doi/10.1137/110822724}.
  \DOIprefix\doi{10.1137/110822724}, \bibinfo{note}{publisher: Society for
  Industrial and Applied Mathematics}.
%Type = Article
\bibitem[{Ahmed and San(2018)}]{ahmed_stabilized_2018}
\bibinfo{author}{M.~Ahmed}, \bibinfo{author}{O.~San},
\newblock \bibinfo{title}{Stabilized principal interval decomposition method
  for model reduction of nonlinear convective systems with moving shocks},
\newblock \bibinfo{journal}{Computational and Applied Mathematics}
  \bibinfo{volume}{37} (\bibinfo{year}{2018}) \bibinfo{pages}{6870--6902}.
  \URLprefix \url{http://link.springer.com/10.1007/s40314-018-0718-z}.
  \DOIprefix\doi{10.1007/s40314-018-0718-z}.
%Type = Article
\bibitem[{Stabile et~al.(2019)Stabile, Ballarin, Zuccarino, and
  Rozza}]{stabile_reduced_2019}
\bibinfo{author}{G.~Stabile}, \bibinfo{author}{F.~Ballarin},
  \bibinfo{author}{G.~Zuccarino}, \bibinfo{author}{G.~Rozza},
\newblock \bibinfo{title}{A reduced order variational multiscale approach for
  turbulent flows},
\newblock \bibinfo{journal}{Advances in Computational Mathematics}
  \bibinfo{volume}{45} (\bibinfo{year}{2019}) \bibinfo{pages}{2349--2368}.
  \URLprefix \url{http://link.springer.com/10.1007/s10444-019-09712-x}.
  \DOIprefix\doi{10.1007/s10444-019-09712-x}.
%Type = Article
\bibitem[{Li et~al.(2025)Li, Xu, and Feng}]{li_pressure-stabilized_2025}
\bibinfo{author}{X.~Li}, \bibinfo{author}{Y.~Xu}, \bibinfo{author}{M.~Feng},
\newblock \bibinfo{title}{A {Pressure}-{Stabilized} {Continuous} {Data}
  {Assimilation} {Reduced} {Order} {Model} for {Incompressible}
  {Navier}–{Stokes} {Equations}},
\newblock \bibinfo{journal}{Journal of Scientific Computing}
  \bibinfo{volume}{103} (\bibinfo{year}{2025}) \bibinfo{pages}{10}. \URLprefix
  \url{https://link.springer.com/10.1007/s10915-025-02828-x}.
  \DOIprefix\doi{10.1007/s10915-025-02828-x}.
%Type = Article
\bibitem[{Gkimisis et~al.(2025)Gkimisis, Aretz, Tezzele, Richter, Benner, and
  Willcox}]{gkimisis2025}
\bibinfo{author}{L.~Gkimisis}, \bibinfo{author}{N.~Aretz},
  \bibinfo{author}{M.~Tezzele}, \bibinfo{author}{T.~Richter},
  \bibinfo{author}{P.~Benner}, \bibinfo{author}{K.~E. Willcox},
\newblock \bibinfo{title}{Non-intrusive reduced-order modeling for dynamical
  systems with spatially localized features},
\newblock \bibinfo{journal}{Computer Methods in Applied Mechanics and
  Engineering} \bibinfo{volume}{444} (\bibinfo{year}{2025})
  \bibinfo{pages}{118115}.
%Type = Misc
\bibitem[{Bach et~al.(2018)Bach, Song, Erhart, and
  Duddeck}]{bach_stability_2018}
\bibinfo{author}{C.~Bach}, \bibinfo{author}{L.~Song},
  \bibinfo{author}{T.~Erhart}, \bibinfo{author}{F.~Duddeck},
  \bibinfo{title}{Stability conditions for the explicit integration of
  projection based nonlinear reduced-order and hyper reduced structural
  mechanics finite element models}, \bibinfo{year}{2018}. \URLprefix
  \url{https://arxiv.org/abs/1806.11404}.
  \DOIprefix\doi{10.48550/ARXIV.1806.11404}, \bibinfo{note}{version Number: 1}.
%Type = Article
\bibitem[{Puri et~al.(2025)Puri, Prakash, Kara, and Zhang}]{Puri2025}
\bibinfo{author}{V.~Puri}, \bibinfo{author}{A.~Prakash}, \bibinfo{author}{L.~B.
  Kara}, \bibinfo{author}{Y.~J. Zhang},
\newblock \bibinfo{title}{Snf-rom: Projection-based nonlinear reduced order
  modeling with smooth neural fields},
\newblock \bibinfo{journal}{Journal of Computational Physics}
  \bibinfo{volume}{532} (\bibinfo{year}{2025}) \bibinfo{pages}{113957}.
  \DOIprefix\doi{10.1016/j.jcp.2025.113957}.
%Type = Article
\bibitem[{Pawar et~al.(2025)Pawar, San, Rasheed, and
  Vedantam}]{pawar2025neural}
\bibinfo{author}{S.~Pawar}, \bibinfo{author}{O.~San},
  \bibinfo{author}{A.~Rasheed}, \bibinfo{author}{S.~Vedantam},
\newblock \bibinfo{title}{A neural network-based projection-based rom for
  incompressible flows with coarsened time-stepping},
\newblock \bibinfo{journal}{Journal of Computational Physics}
  \bibinfo{volume}{523} (\bibinfo{year}{2025}) \bibinfo{pages}{113615}.
  \DOIprefix\doi{10.1016/j.jcp.2024.113615}.
%Type = Article
\bibitem[{San and Maulik(2018)}]{san2018extreme}
\bibinfo{author}{O.~San}, \bibinfo{author}{R.~Maulik},
\newblock \bibinfo{title}{Extreme learning machine for reduced order modeling
  of turbulent geophysical flows},
\newblock \bibinfo{journal}{Physical Review E} \bibinfo{volume}{97}
  (\bibinfo{year}{2018}) \bibinfo{pages}{042322}.
%Type = Article
\bibitem[{Carlberg et~al.(2017)Carlberg, Barone, and
  Antil}]{carlberg2017galerkin}
\bibinfo{author}{K.~Carlberg}, \bibinfo{author}{M.~Barone},
  \bibinfo{author}{H.~Antil},
\newblock \bibinfo{title}{Galerkin v. least-squares {P}etrov-{G}alerkin
  projection in nonlinear model reduction},
\newblock \bibinfo{journal}{Journal of Computational Physics}
  \bibinfo{volume}{330} (\bibinfo{year}{2017}) \bibinfo{pages}{693--734}.
%Type = Article
\bibitem[{Xiao et~al.(2015)Xiao, Fang, Buchan, Pain, Navon, and
  Muggeridge}]{xiao2015non}
\bibinfo{author}{D.~Xiao}, \bibinfo{author}{F.~Fang},
  \bibinfo{author}{A.~Buchan}, \bibinfo{author}{C.~Pain},
  \bibinfo{author}{I.~Navon}, \bibinfo{author}{A.~Muggeridge},
\newblock \bibinfo{title}{Non-intrusive reduced order modelling of the
  {N}avier-{S}tokes equations},
\newblock \bibinfo{journal}{Computer Methods in Applied Mechanics and
  Engineering} \bibinfo{volume}{293} (\bibinfo{year}{2015})
  \bibinfo{pages}{522--541}.
%Type = Article
\bibitem[{Stabile and Rozza(2018)}]{stabile_finite_2018}
\bibinfo{author}{G.~Stabile}, \bibinfo{author}{G.~Rozza},
\newblock \bibinfo{title}{Finite volume {POD}-{Galerkin} stabilised reduced
  order methods for the parametrised incompressible {Navier}–{Stokes}
  equations},
\newblock \bibinfo{journal}{Computers \& Fluids} \bibinfo{volume}{173}
  (\bibinfo{year}{2018}) \bibinfo{pages}{273--284}. \URLprefix
  \url{https://linkinghub.elsevier.com/retrieve/pii/S0045793018300422}.
  \DOIprefix\doi{10.1016/j.compfluid.2018.01.035}.
%Type = Article
\bibitem[{Bendixson(1902)}]{bendixson_sur_1902}
\bibinfo{author}{I.~Bendixson},
\newblock \bibinfo{title}{Sur les racines d'une équation fondamentale},
\newblock \bibinfo{journal}{Acta Mathematica} \bibinfo{volume}{25}
  (\bibinfo{year}{1902}) \bibinfo{pages}{359--365}. \URLprefix
  \url{http://projecteuclid.org/euclid.acta/1485882119}.
  \DOIprefix\doi{10.1007/BF02419030}.
%Type = Article
\bibitem[{Rao(1979)}]{rao_separation_1979}
\bibinfo{author}{C.~Rao},
\newblock \bibinfo{title}{Separation theorems for singular values of matrices
  and their applications in multivariate analysis},
\newblock \bibinfo{journal}{Journal of Multivariate Analysis}
  \bibinfo{volume}{9} (\bibinfo{year}{1979}) \bibinfo{pages}{362--377}.
  \URLprefix
  \url{https://linkinghub.elsevier.com/retrieve/pii/0047259X79900940}.
  \DOIprefix\doi{10.1016/0047-259X(79)90094-0}.
%Type = Article
\bibitem[{Trias et~al.(2014)Trias, Lehmkuhl, Oliva, Pérez-Segarra, and
  Verstappen}]{Trias2014}
\bibinfo{author}{F.~X. Trias}, \bibinfo{author}{O.~Lehmkuhl},
  \bibinfo{author}{A.~Oliva}, \bibinfo{author}{C.~D. Pérez-Segarra},
  \bibinfo{author}{R.~W. Verstappen},
\newblock \bibinfo{title}{Symmetry-preserving discretization of
  {Navier}-{Stokes} equations on collocated unstructured grids},
\newblock \bibinfo{journal}{Journal of Computational Physics}
  \bibinfo{volume}{258} (\bibinfo{year}{2014}) \bibinfo{pages}{246--267}.
  \DOIprefix\doi{10.1016/j.jcp.2013.10.031}, \bibinfo{note}{publisher: Academic
  Press Inc.}
%Type = Article
\bibitem[{Verstappen and Veldman(2003)}]{Verstappen2003}
\bibinfo{author}{R.~W. C.~P. Verstappen}, \bibinfo{author}{A.~E.~P. Veldman},
\newblock \bibinfo{title}{Symmetry-preserving discretization of turbulent
  flow},
\newblock \bibinfo{journal}{Journal of Computational Physics}
  \bibinfo{volume}{187} (\bibinfo{year}{2003}) \bibinfo{pages}{343--368}.
  \DOIprefix\doi{10.1016/S0021-9991(03)00126-8}, \bibinfo{note}{publisher:
  Academic Press Inc.}
%Type = Article
\bibitem[{Chorin(1967)}]{chorin_numerical_1967}
\bibinfo{author}{A.~J. Chorin},
\newblock \bibinfo{title}{A numerical method for solving incompressible viscous
  flow problems},
\newblock \bibinfo{journal}{Journal of Computational Physics}
  \bibinfo{volume}{2} (\bibinfo{year}{1967}) \bibinfo{pages}{12--26}.
  \URLprefix
  \url{https://linkinghub.elsevier.com/retrieve/pii/002199916790037X}.
  \DOIprefix\doi{10.1016/0021-9991(67)90037-x}, \bibinfo{note}{publisher:
  Elsevier BV}.
%Type = Article
\bibitem[{Sanderse and Koren(2012)}]{sanderse_accuracy_2012}
\bibinfo{author}{B.~Sanderse}, \bibinfo{author}{B.~Koren},
\newblock \bibinfo{title}{Accuracy analysis of explicit {Runge}-{Kutta} methods
  applied to the incompressible {Navier}-{Stokes} equations},
\newblock \bibinfo{journal}{Journal of Computational Physics}
  \bibinfo{volume}{231} (\bibinfo{year}{2012}) \bibinfo{pages}{3041--3063}.
  \DOIprefix\doi{10.1016/j.jcp.2011.11.028}, \bibinfo{note}{publisher: Academic
  Press Inc.}
%Type = Article
\bibitem[{Agdestein and Sanderse(2025)}]{agdestein_discretize_2025}
\bibinfo{author}{S.~D. Agdestein}, \bibinfo{author}{B.~Sanderse},
\newblock \bibinfo{title}{Discretize first, filter next: {Learning}
  divergence-consistent closure models for large-eddy simulation},
\newblock \bibinfo{journal}{Journal of Computational Physics}
  \bibinfo{volume}{522} (\bibinfo{year}{2025}) \bibinfo{pages}{113577}.
  \URLprefix
  \url{https://www.sciencedirect.com/science/article/pii/S0021999124008258}.
  \DOIprefix\doi{10.1016/j.jcp.2024.113577}.
%Type = Book
\bibitem[{Hairer and Wanner(1996)}]{hairer_solving_1996}
\bibinfo{author}{E.~Hairer}, \bibinfo{author}{G.~Wanner},
  \bibinfo{title}{Solving {Ordinary} {Differential} {Equations} {II}},
  volume~\bibinfo{volume}{14} of \textit{\bibinfo{series}{Springer {Series} in
  {Computational} {Mathematics}}}, \bibinfo{publisher}{Springer},
  \bibinfo{address}{Berlin, Heidelberg}, \bibinfo{year}{1996}. \URLprefix
  \url{http://link.springer.com/10.1007/978-3-642-05221-7}.
  \DOIprefix\doi{10.1007/978-3-642-05221-7}.
%Type = Article
\bibitem[{Kraaijevanger(1991)}]{kraaijevanger_contractivity_1991}
\bibinfo{author}{J.~F. B.~M. Kraaijevanger},
\newblock \bibinfo{title}{Contractivity of {Runge}-{Kutta} methods},
\newblock \bibinfo{journal}{BIT} \bibinfo{volume}{31} (\bibinfo{year}{1991})
  \bibinfo{pages}{482--528}.
%Type = Article
\bibitem[{Gerschgorin(1931)}]{gerschgorin_uber_1931}
\bibinfo{author}{S.~Gerschgorin},
\newblock \bibinfo{title}{Über die {Abgrenzung} fer {Eigenwerte} einer
  {Matrix}},
\newblock \bibinfo{journal}{Bulletin de l'Académie des Sciences de l'URSS:
  Classe des sciences mathématiques et na}  (\bibinfo{year}{1931})
  \bibinfo{pages}{749--754}.
%Type = Article
\bibitem[{Benner et~al.(2015)Benner, Gugercin, and
  Willcox}]{benner_survey_2015}
\bibinfo{author}{P.~Benner}, \bibinfo{author}{S.~Gugercin},
  \bibinfo{author}{K.~Willcox},
\newblock \bibinfo{title}{A {Survey} of {Projection}-{Based} {Model}
  {Reduction} {Methods} for {Parametric} {Dynamical} {Systems}},
\newblock \bibinfo{journal}{SIAM Review} \bibinfo{volume}{57}
  (\bibinfo{year}{2015}) \bibinfo{pages}{483--531}. \URLprefix
  \url{http://epubs.siam.org/doi/10.1137/130932715}.
  \DOIprefix\doi{10.1137/130932715}.
%Type = Article
\bibitem[{Kunisch and Volkwein(2010)}]{kunisch_optimal_2010}
\bibinfo{author}{K.~Kunisch}, \bibinfo{author}{S.~Volkwein},
\newblock \bibinfo{title}{Optimal snapshot location for computing {POD} basis
  functions},
\newblock \bibinfo{journal}{ESAIM: Mathematical Modelling and Numerical
  Analysis} \bibinfo{volume}{44} (\bibinfo{year}{2010})
  \bibinfo{pages}{509--529}. \URLprefix
  \url{http://www.esaim-m2an.org/10.1051/m2an/2010011}.
  \DOIprefix\doi{10.1051/m2an/2010011}.
%Type = Masterthesis
\bibitem[{Rosenberger(2021)}]{rosenberger_advances_2021}
\bibinfo{author}{H.~Rosenberger}, \bibinfo{title}{Advances in
  {Structure}-{Preserving} {Reduced} {Order} {Models} for the {Incompressible}
  {Navier}-{Stokes} {Equations}}, Master's thesis, Technische Universiteit
  Eindhoven, \bibinfo{address}{Eindhoven}, \bibinfo{year}{2021}. \URLprefix
  \url{https://research.tue.nl/en/studentTheses/advances-in-structure-preserving-reduced-order-models-for-the-inc/}.
%Type = Article
\bibitem[{Wilkinson(1965)}]{wilkinson_qr_1965}
\bibinfo{author}{J.~H. Wilkinson},
\newblock \bibinfo{title}{The {QR} {Algorithm} for {Real} {Symmetric}
  {Matrices} with {Multiple} {Eigenvalues}},
\newblock \bibinfo{journal}{The Computer Journal} \bibinfo{volume}{8}
  (\bibinfo{year}{1965}) \bibinfo{pages}{85--87}. \URLprefix
  \url{https://academic.oup.com/comjnl/article-lookup/doi/10.1093/comjnl/8.1.85}.
  \DOIprefix\doi{10.1093/comjnl/8.1.85}.
%Type = Article
\bibitem[{Corbató(1963)}]{corbato_coding_1963}
\bibinfo{author}{F.~J. Corbató},
\newblock \bibinfo{title}{On the {Coding} of {Jacobi}'s {Method} for
  {Computing} {Eigenvalues} and {Eigenvectors} of {Real} {Symmetric}
  {Matrices}},
\newblock \bibinfo{journal}{Journal of the ACM} \bibinfo{volume}{10}
  (\bibinfo{year}{1963}) \bibinfo{pages}{123--125}. \URLprefix
  \url{https://dl.acm.org/doi/10.1145/321160.321161}.
  \DOIprefix\doi{10.1145/321160.321161}.
%Type = Book
\bibitem[{Trefethen and Bau(1997)}]{trefethen_numerical_1997}
\bibinfo{author}{L.~N. Trefethen}, \bibinfo{author}{D.~Bau},
  \bibinfo{title}{Numerical {Linear} {Algebra}}, \bibinfo{edition}{1} ed.,
  \bibinfo{publisher}{Society for Industrial and Applied Mathematics},
  \bibinfo{year}{1997}.
%Type = Article
\bibitem[{Baiges et~al.(2013)Baiges, Codina, and
  Idelsohn}]{baiges_explicit_2013}
\bibinfo{author}{J.~Baiges}, \bibinfo{author}{R.~Codina},
  \bibinfo{author}{S.~Idelsohn},
\newblock \bibinfo{title}{Explicit reduced‐order models for the stabilized
  finite element approximation of the incompressible {Navier}–{Stokes}
  equations},
\newblock \bibinfo{journal}{International Journal for Numerical Methods in
  Fluids} \bibinfo{volume}{72} (\bibinfo{year}{2013})
  \bibinfo{pages}{1219--1243}. \URLprefix
  \url{https://onlinelibrary.wiley.com/doi/10.1002/fld.3777}.
  \DOIprefix\doi{10.1002/fld.3777}.
%Type = Book
\bibitem[{Stewart(2001)}]{stewart_matrix_2001}
\bibinfo{author}{G.~W. Stewart}, \bibinfo{title}{Matrix {Algorithms}: {Volume}
  {II}: {Eigensystems}}, \bibinfo{publisher}{Society for Industrial and Applied
  Mathematics}, \bibinfo{year}{2001}. \URLprefix
  \url{http://epubs.siam.org/doi/book/10.1137/1.9780898718058}.
  \DOIprefix\doi{10.1137/1.9780898718058}.
%Type = Book
\bibitem[{Magnus and Neudecker(2019)}]{magnus_matrix_2019}
\bibinfo{author}{J.~R. Magnus}, \bibinfo{author}{H.~Neudecker},
  \bibinfo{title}{Matrix differential calculus with applications in statistics
  and econometrics}, Wiley series in probability and statistics,
  \bibinfo{edition}{3rd ed} ed., \bibinfo{publisher}{Wiley},
  \bibinfo{address}{Hoboken (N.J.)}, \bibinfo{year}{2019}.
%Type = Incollection
\bibitem[{joh(1985)}]{johnson_unitary_1985}
\bibinfo{title}{Unitary equivalence and normal matrices},
\newblock in: \bibinfo{editor}{C.~R. Johnson}, \bibinfo{editor}{R.~A. Horn}
  (Eds.), \bibinfo{booktitle}{Matrix {Analysis}}, \bibinfo{publisher}{Cambridge
  University Press}, \bibinfo{address}{Cambridge}, \bibinfo{year}{1985}, pp.
  \bibinfo{pages}{65--118}. \URLprefix
  \url{https://www.cambridge.org/core/books/matrix-analysis/unitary-equivalence-and-normal-matrices/01907C4F4E1B326B4E6E70CB316EFEB1}.
  \DOIprefix\doi{10.1017/CBO9780511810817.004}.
%Type = Phdthesis
\bibitem[{Sanderse(2013)}]{sanderse_energy-conserving_2013}
\bibinfo{author}{B.~Sanderse}, \bibinfo{title}{Energy-conserving discretization
  methods for the incompressible {Navier}-{Stokes} equations:application to the
  simulation of wind-turbine wakes}, Ph.D. thesis, Technische Universiteit
  Eindhoven, \bibinfo{year}{2013}. \URLprefix
  \url{https://research.tue.nl/en/publications/energyconserving-discretization-methods-for-the-incompressible-navierstokes-equations--application-to-the-simulation-of-windturbine-wakes(d188919d-72bb-4bbd-a15f-20360ec22e95).html}.
%Type = Misc
\bibitem[{Sanderse(2018)}]{sanderse_ins2d_2018}
\bibinfo{author}{B.~Sanderse}, \bibinfo{title}{{INS2D}}, \bibinfo{year}{2018}.
  \URLprefix \url{https://github.com/bsanderse/INS2D}.
%Type = Article
\bibitem[{Qian et~al.(2020)Qian, Kramer, Peherstorfer, and
  Willcox}]{qian_lift_2020}
\bibinfo{author}{E.~Qian}, \bibinfo{author}{B.~Kramer},
  \bibinfo{author}{B.~Peherstorfer}, \bibinfo{author}{K.~Willcox},
\newblock \bibinfo{title}{Lift \& {Learn}: {Physics}-informed machine learning
  for large-scale nonlinear dynamical systems},
\newblock \bibinfo{journal}{Physica D: Nonlinear Phenomena}
  \bibinfo{volume}{406} (\bibinfo{year}{2020}) \bibinfo{pages}{132401}.
  \URLprefix
  \url{https://www.sciencedirect.com/science/article/pii/S0167278919307651}.
  \DOIprefix\doi{10.1016/j.physd.2020.132401}.
%Type = Article
\bibitem[{Sanderse et~al.(2014)Sanderse, Verstappen, and
  Koren}]{sanderse_boundary_2014}
\bibinfo{author}{B.~Sanderse}, \bibinfo{author}{R.~Verstappen},
  \bibinfo{author}{B.~Koren},
\newblock \bibinfo{title}{Boundary treatment for fourth-order staggered mesh
  discretizations of the incompressible {Navier}–{Stokes} equations},
\newblock \bibinfo{journal}{Journal of Computational Physics}
  \bibinfo{volume}{257} (\bibinfo{year}{2014}) \bibinfo{pages}{1472--1505}.
  \URLprefix
  \url{https://linkinghub.elsevier.com/retrieve/pii/S0021999113006670}.
  \DOIprefix\doi{10.1016/j.jcp.2013.10.002}, \bibinfo{note}{publisher: Elsevier
  BV}.
%Type = Article
\bibitem[{Perron(1907)}]{perron_zur_1907}
\bibinfo{author}{O.~Perron},
\newblock \bibinfo{title}{Zur {Theorie} der {Matrices}},
\newblock \bibinfo{journal}{Mathematische Annalen} \bibinfo{volume}{64}
  (\bibinfo{year}{1907}) \bibinfo{pages}{248--263}. \URLprefix
  \url{http://link.springer.com/10.1007/BF01449896}.
  \DOIprefix\doi{10.1007/BF01449896}.
%Type = Article
\bibitem[{Frobenius(1912)}]{frobenius_uber_1912}
\bibinfo{author}{G.~Frobenius},
\newblock \bibinfo{title}{Über {Matrizen} aus nicht negativen {Elementen}},
\newblock \bibinfo{journal}{Sitzung der physikalisch-mathematischen Classe}
  (\bibinfo{year}{1912}) \bibinfo{pages}{456--477}.
%Type = Book
\bibitem[{Gantmakher(2005)}]{gantmakher_applications_2005}
\bibinfo{author}{F.~R. Gantmakher}, \bibinfo{title}{Applications of the theory
  of matrices}, \bibinfo{publisher}{Dover Publications},
  \bibinfo{address}{Mineola, N.Y.}, \bibinfo{year}{2005}. \bibinfo{note}{OCLC:
  58720731}.
%Type = Article
\bibitem[{Nikiforov(2007)}]{nikiforov_chromatic_2007}
\bibinfo{author}{V.~Nikiforov},
\newblock \bibinfo{title}{Chromatic number and spectral radius},
\newblock \bibinfo{journal}{Linear Algebra and its Applications}
  \bibinfo{volume}{426} (\bibinfo{year}{2007}) \bibinfo{pages}{810--814}.
  \URLprefix
  \url{https://linkinghub.elsevier.com/retrieve/pii/S0024379507002704}.
  \DOIprefix\doi{10.1016/j.laa.2007.06.005}.

\end{thebibliography}

\end{document}